\documentclass[11pt, reqno]{amsart}
\usepackage{fullpage}
\usepackage[OT1]{fontenc}

\RequirePackage{amsthm,amsmath,amsfonts,amssymb,color,bm}
\RequirePackage[authoryear]{natbib}%% uncomment this for author-year citations
\RequirePackage[colorlinks,citecolor=blue,urlcolor=blue]{hyperref}
\RequirePackage{graphicx}

\RequirePackage[linesnumbered,ruled]{algorithm2e}
\RequirePackage{cancel}
\RequirePackage{dsfont}
\RequirePackage{mathtools}
\RequirePackage{microtype} % please have this in the preamble, really helps with tight margins 
\RequirePackage{algorithmic}
\RequirePackage[english]{babel}
\RequirePackage[utf8]{inputenc}
\RequirePackage{mathrsfs}  
\RequirePackage{afterpage}
\RequirePackage{mathtools}
\RequirePackage{amsfonts}
\RequirePackage[colorinlistoftodos]{todonotes}
\RequirePackage{cleveref}
\let\Cref\Cref
\RequirePackage{enumitem}
\newcommand{\excise}[1]{}

\RequirePackage{newfloat}
\RequirePackage{multirow}

\theoremstyle{plain}

\newtheorem{theorem}{Theorem}[section]
\newtheorem{lemma}[theorem]{Lemma}
\newtheorem{proposition}[theorem]{Proposition}
\newtheorem{corollary}[theorem]{Corollary}
\newcommand{\circled}[1]{\textcircled{\raisebox{-0.9pt}{#1}}}

\theoremstyle{remark}
\newtheorem{definition}[theorem]{Definition}

\makeatletter
\newcommand{\pushright}[1]{\ifmeasuring@#1\else\omit\hfill$\displaystyle#1$\fi\ignorespaces}
\newcommand{\pushleft}[1]{\ifmeasuring@#1\else\omit$\displaystyle#1$\hfill\fi\ignorespaces}
\makeatother

\DeclareMathOperator\GP{GP}
\renewcommand{\arccos}{\operatorname{acos}}

\title{Statistical Inference for Rates of Change in Spatial Processes over Riemannian Manifolds}
\author{Didong Li$^{\star\diamond}$}
\author{Aritra Halder$^{\dagger\diamond}$}
\author{Sudipto Banerjee$^{\ddagger}$}

\thanks{$^\diamond$First and Second author contributed equally to this work}
\thanks{$^{\star}$Department of Biostatistics, University of North Carolina at Chapel Hill, NC}
\thanks{$^{\dagger}$Department of Biostatistics and Epidemiology, Drexel University, Philadelphia, PA}
\thanks{$^{\ddagger}$Department of Biostatistics, University of California, Los Angeles, CA}

\begin{document}
%new math symbols taking no arguments
% \newcommand{\parallelsum}{\mathbin{\!/\mkern-5mu/\!}}

\newcommand\numberthis{\addtocounter{equation}{1}\tag{\theequation}}

\newcommand{\given}{\, |\,}
\newcommand\<{\langle}
\newcommand{\T}{^{\mathrm{T}}}
\newcommand\CC{\mathbb{C}}
\newcommand\FF{\mathbb{F}}
\newcommand\NN{\mathbb{N}}
\newcommand\QQ{\mathbb{Q}}
\newcommand\RR{\mathbb{R}}
\newcommand\PP{\mathbb{P}}
\newcommand\ZZ{\mathbb{Z}}
\newcommand\EE{\mathbb{E}}
\newcommand\MM{\mathcal{M}}
\newcommand\WW{\mathcal{W}}
\newcommand\kk{\mathfrak{k}}
\newcommand\mm{\mathfrak{m}}
\newcommand\nn{\mathfrak{n}}
\newcommand\pp{\mathfrak{p}}
\newcommand\quq{\mathfrak{q}}
\newcommand\Gg{\mathfrak{g}}
\newcommand\GL{\mathrm{GL}}
\newcommand\SO{\mathrm{SO}}
\newcommand\Z{\mathcal{Z}}
\newcommand\from{\leftarrow}
\newcommand\into{\hookrightarrow}
\newcommand{\tr}{\operatorname{tr}}
\newcommand\onto{\twoheadrightarrow}
\newcommand\minus{\smallsetminus}
\newcommand\goesto{\rightsquigarrow}
\newcommand\dirlim{\varinjlim}
\newcommand\invlim{\varprojlim}
%redefined math symbols taking no arguments
\renewcommand\>{\rangle}
\renewcommand\iff{\Leftrightarrow}
\renewcommand\implies{\Rightarrow}
%\floatname{algorithm}{Procedure}
\renewcommand{\algorithmicrequire}{\textbf{Input:}}
\renewcommand{\algorithmicensure}{\textbf{Output:}}
\newcommand{\bfx}{{X}}
\newcommand{\cov}{\mathrm{Cov}}
\newcommand{\vect}{\mathrm{Vec}}
\newcommand{\sg}{\Sigma}
\newcommand{\bfv}{{v}}
\newcommand{\bfc}{{c}}
\newcommand{\cnh}{\widehat{c}}
\newcommand{\hbfv}{\widehat{v}}
\newcommand{\vnh}{\widehat{V}}
\newcommand{\rnh}{\widehat{r}}
\newcommand{\PD}{\mathrm{PD}}
\newcommand{\eps}{\varepsilon}
\newcommand{\tF}{\mathcal{F}}
\newcommand{\tX}{\mathcal{X}}
\newcommand{\tY}{\mathcal{Y}}
\newcommand{\tmu}{\widetilde{\mu}}
\newcommand{\tbeta}{\widetilde{\beta}}
\newcommand{\tepsilon}{\widetilde{\epsilon}}
\newcommand{\tvarphi}{\widetilde{\varphi}}
\newcommand\blfootnote[1]{%
  \begingroup
  \renewcommand\thefootnote{}\footnote{#1}%
  \addtocounter{footnote}{-1}%
  \endgroup
}

\newcommand\scalemath[2]{\scalebox{#1}{\mbox{\ensuremath{\displaystyle #2}}}}

%new math symbols taking arguments
\newcommand\ol[1]{{\overline{#1}}}

%redefined math symbols taking arguments
\renewcommand\mod[1]{\ (\mathrm{mod}\ #1)}
% Already in AOS template
% \newtheorem{theorem}{Theorem}
% \newtheorem{definition}{Definition}
% \newtheorem{lemma}{Lemma}
% \newtheorem{example}{Example}
% \newtheorem*{example*}{Example}
% \newtheorem{proposition}{Proposition}
% \newtheorem{corollary}{Corollary}
% \newtheorem{remark}{Remark}
% \newtheorem{note}{Note}
%math operators not in math italic font
% \DeclareMathOperator\gr{gr}
% \DeclareMathOperator\IN{\mathsf{in}}
% \DeclareMathOperator\ass{Ass}
% \DeclareMathOperator\GP{GP}
% \DeclareMathOperator\tor{Tor}
% \DeclareMathOperator\Ad{Ad}
% \DeclareMathOperator\im{im}
% \DeclareMathOperator\Sym{Sym}
% \DeclareMathOperator\Hess{Hess}
% \DeclareMathOperator\var{var}
% \DeclareMathOperator\diag{diag}
% \DeclareMathOperator\rees{\mathcal{R}}
% \DeclareMathOperator\spec{Spec}
% \DeclareMathOperator\length{length}
% \DeclareMathOperator\argmin{argmin}
% \DeclareMathOperator\argmax{argmax}
% \DeclareMathOperator\ad{ad}
% \DeclareMathOperator\Id{I}
% \DeclareMathOperator\eig{eig}
% \DeclareMathOperator\rank{rank}
\DeclarePairedDelimiterX{\infdivx}[2]{(}{)}{%
	#1\;\delimsize\|\;#2%
}
\newcommand{\KL}{KL\infdivx}
\DeclarePairedDelimiter{\norm}{\lVert}{\rVert}

%for easy 2 x 2 matrices
\newcommand\twobytwo[1]{\left[\begin{array}{@{}cc@{}}#1\end{array}\right]}

%for easy column vectors of size 2
\newcommand\tworow[1]{\left[\begin{array}{@{}c@{}}#1\end{array}\right]}

%for \marginpar to fit optimally
%hoffset=-1in
%\setlength\marginparwidth{2.2in}
%\setlength\marginparsep{1mm}

\newcommand\red[1]{\marginpar{\vspace{-1.4ex}\footnotesize{\color{red}#1}}}
\newcommand\score[1]{\marginpar{\colorbox{yellow}{#1/10}}}
\newcommand\total[1]{\marginpar{\colorbox{yellow}{\huge #1/100}}}
\newcommand\magenta[1]{\colorbox{magenta}{$\!$#1$\!$}}
\newcommand\yellow[1]{\colorbox{yellow}{$\!$#1$\!$}}
\newcommand\green[1]{\colorbox{green}{$\!$#1$\!$}}
\newcommand\cyan[1]{\colorbox{cyan}{$\!$#1$\!$}}
\newcommand\rmagenta[1]{\red{\magenta{\phantom{:}}\,: #1}}
\newcommand\ryellow[1]{\red{\yellow{\phantom{:}}\,: #1}}
\newcommand\rgreen[1]{\red{\green{\phantom{:}}\,: #1}}
\newcommand\rcyan[1]{\red{\cyan{\phantom{:}}\,: #1}}
\newcommand{\RNum}[1]{\uppercase\expandafter{\romannumeral #1\relax}}
%\newcommand{\GP}{$GP$}

% \xdefinecolor{dukeblue}{rgb}{0.004,0.129,0.412}
% \newcommand{\dl}[1]{\textcolor{dukeblue}{\textsf{DL: #1}}}
% \newcommand{\ah}[1]{\textcolor{teal}{\textsf{AH: #1}}}
\newcommand{\indep}{\perp \!\!\! \perp}
\DeclarePairedDelimiter{\ceil}{\lceil}{\rceil}
\begin{abstract}
Quantifying local rates of change in spatial processes is central to identifying transition zones, spatial boundaries, and physiological gradients. Existing spatial derivative methods rely on Euclidean coordinates, which introduce artificial distortion when modeling data supported on non-Euclidean geometries. We propose a model-based Bayesian framework for probabilistic inference on intrinsic directional derivatives of Gaussian processes on compact Riemannian manifolds, with or without smooth boundaries. By framing directional derivatives along scientifically motivated smooth vector fields, the approach characterizes intrinsic spatial change independently of ambient coordinate representations. We establish mean-square regularity conditions for differentiability, derive the exact cross-covariance structure between the parent Gaussian process and its derivative processes, and construct a scalable low-rank spectral computation scheme for joint predictive inference over large point clouds. The methodology is illustrated in two distinct application settings: inferring latent heart-rate dynamics along physical movement trajectories and decomposing directional variations in cortical thickness relative to anatomical folding in individuals with autism spectrum disorder.
\end{abstract}
\maketitle

\section{Introduction}\label{sec: intro}
Spatial statistical models are commonly used to describe random quantities indexed by location and to predict an underlying process from noisy or partially observed data. Beyond the reconstruction of the process itself, scientific interest often concerns its local behavior: where does the process change rapidly, in which direction does this change occur, and where is there statistical evidence for unusually large rates of change? For a sufficiently smooth spatial process $Z$ on a Euclidean domain, such questions can be addressed through derivatives of the latent random field. This perspective underlies statistical approaches to spatial gradients and wombling, where derivatives are treated as stochastic processes and inference is conducted jointly with the parent process \citep[see, e.g.][]{womble1951differential,morris1993bayesian,banerjee2003directional,majumdar2006gradients,liang2009bayesian,heaton2014wombling,terres2015using,quick2015bayesian,wang2016estimating,terres2016spatial,wang2018process,halder2024bayesian}.

Modern spatial data increasingly arise on domains that are more naturally represented as manifolds. Examples include environmental processes over spheres or ellipsoids, measurements on anatomical surfaces, and physiological responses indexed by trajectories through space. In these settings, ambient Euclidean coordinates need not provide meaningful directions of change. On a Riemannian manifold, $(\MM,g)$, where $g$ is the Riemannian metric \citep[see, e.g.,][]{lee_introduction_2018}, a scientifically meaningful direction at $x$ is instead represented by $V_x\in T_x\MM$, the tangent space to $\MM$ at $x$, and a smoothly varying collection of such directions forms a vector field $V$. The intrinsic rate of change is the directional derivative $D_VZ(x)=\langle \nabla_{\MM}Z(x),V_x\rangle_g$, when it exists. This is invariant to the choice of local coordinates and separates two components of the inferential problem: the stochastic behavior of $Z$ and the geometrically or scientifically meaningful direction $V$ along which its change is to be studied. Higher-order rates of change similarly require differential structure beyond that available from an ordinary covariance model for $Z$.

%Developing statistical inference for these quantities raises several questions that are largely absent in the Euclidean formulation. First, a Gaussian process on a manifold need not possess mean-square derivatives merely because its covariance is valid on the manifold. Conditions are required under which the parent process admits first- and higher-order differential processes. Second, the derivative process must itself possess a valid covariance structure, together with cross-covariances with the parent process, if uncertainty in local rates of change is to be propagated probabilistically. Third, manifolds encountered in applications may possess boundaries, requiring a careful definition of admissible directional derivatives at boundary points. Finally, practical manifolds are often available only through discrete representations such as triangular meshes, so the resulting theory must admit computational approximations that remain feasible for large point clouds.

A growing literature has developed Gaussian processes and valid covariance constructions on non-Euclidean domains. On spheres, such models have been extensively studied for global environmental and geophysical processes \citep[see, e.g.][]{banerjee2005geodetic,jun2008nonstationary,alegria2021f,jeong2015class,porcu2016spatio,guinness2016isotropic,guella2018strictly,de2018regularity}. More generally, replacing Euclidean distance in a familiar covariance function with geodesic distance does not necessarily preserve positive definiteness: even natural Matérn-type constructions fail for broad classes of curved spaces \citep[see, e.g.][]{gneiting2013strictly,feragen2015geodesic}. This has motivated intrinsically defined Gaussian processes based on the spectral geometry of the manifold, heat equations, diffusion constructions, and stochastic partial differential equations \citep[see, e.g.][]{lindgren2011explicit,lang2015isotropic,niu2019intrinsic,herrmann2020multilevel,borovitskiy2020matern,borovitskiy2021matern,dunson2022graph}. We adopt the spectral Matérn construction based on the eigensystem of the Laplace--Beltrami (LB) operator, which yields valid covariance functions on general compact Riemannian manifolds and connects stochastic regularity directly to the intrinsic geometry of $\MM$. Our interest differs from existing work on covariance construction and parameter estimation: we seek formal statistical inference for the differential processes induced by such Gaussian processes. Existing work has established inferential properties of covariance parameters for this class of models \citep[see, e.g.][]{li2023inference}. A general model-based construction of intrinsic directional differential processes for Gaussian random fields on compact Riemannian manifolds, together with their joint cross-covariance structure with the parent process for probabilistic inference, does not exist in the literature.

We develop a model-based framework for this problem. We first study mean-square regularity of Gaussian processes on compact $p$-dimensional Riemannian manifolds, possibly with a smooth boundary. For the spectral Whittle-Matérn family, we establish sufficient conditions % $\nu>\frac{p-1}{2}, \nu>\frac{p+1}{2}$ and $\nu>\frac{p+3}{2}$, 
for mean-square continuity and differentiability. We then define directional derivative processes along smooth vector fields % and directional curvature processes, 
and derive the covariance and cross-covariance structures with the parent Gaussian process. Consequently, the parent process and its derivative process can be represented as a joint multivariate Gaussian process, providing a coherent probability model for posterior inference on local rates of change. For manifolds with boundary, we formulate the differential processes using admissible one-sided directions and characterize the consequences of Neumann conditions.

The spectral representation also provides an efficient route to Bayesian computation. Covariance functions defined through the LB spectrum require truncation in practice, and we show that finite spectral truncation induces mean-square smoothness irrespective of the smoothness of the corresponding infinite expansion. Computationally, the same truncation permits Bayesian inference through a comparatively small collection of spectral coefficients and allows posterior samples of differential processes to be obtained jointly over large point clouds, thereby preserving their spatial dependence. For triangulated surfaces, we approximate the LB operator by the cotangent Laplacian and use finite-element representations to evaluate directional derivatives of its eigenfunctions. Simulation experiments on the sphere and the Stanford Bunny assess the resulting inference under both analytically tractable and complex mesh geometries and demonstrate the greater sensitivity of derivative inference, relative to process reconstruction, to spectral truncation.

%Simulation experiments illustrate the framework on two geometrically distinct domains. On the sphere, where aspects of the geometry and differential structure are available analytically, we assess recovery of a known directional derivative process. We then consider the Stanford Bunny, a complex triangulated surface with boundary, to investigate inference when the manifold and its differential operators are available only through a discrete mesh. The latter experiment also illustrates the effect of spectral truncation: recovery of the parent process is relatively insensitive to the truncation level, whereas derivative inference benefits more substantially from the inclusion of higher-frequency eigenfunctions.

Two applications illustrate how vector fields translate manifold geometry into scientifically interpretable rates of change. We first analyze heart-rate measurements recorded by wearable devices along cycling trajectories. Arc-length parameterization treats each trajectory intrinsically as a one-dimensional manifold, such that the derivative process measures change in latent heart rate per unit distance traveled, independent of the geographic coordinates used to represent the route. We then analyze cortical thickness from the Autism Brain Imaging Data Exchange-I (ABIDE-I) \citep[see, e.g.,][]{di_martino_autism_2014}. Treating the cortical surface as a two-dimensional manifold, we use the intrinsic gradient of sulcal depth to construct orthogonal across- and along-fold tangent directions. Posterior inference for the corresponding directional derivatives decomposes local cortical-thickness variation relative to cortical folding and yields summaries of its magnitude, orientation, and posterior evidence for elevated rates of change. These examples demonstrate how the same inferential framework accommodates both simple and highly complex manifold geometries while allowing the vector field to be chosen according to the scientific question. To our knowledge, model-based probabilistic inference for intrinsic directional derivatives of cortical morphometric fields relative to anatomically defined tangent directions has not previously been explored in the wider neuroimaging literature.

%Our contributions are therefore threefold. First, we develop mean-square smoothness theory and valid first- and second-order differential processes for Gaussian processes on compact Riemannian manifolds, including manifolds with boundary. Second, we derive a geometry-adapted low-rank Bayesian framework for joint posterior inference on these differential processes using the LB spectrum. Third, we demonstrate how scientifically defined tangent vector fields yield interpretable inferential targets on non-Euclidean domains. The balance of the paper evolves as follows. \Cref{sec:gp-manifolds} introduces Gaussian processes and spectral covariance constructions on compact Riemannian manifolds, including manifolds with boundary. \Cref{sec:cont} establishes mean-square continuity results, and \Cref{sec:diff} develops the first- and second-order differential processes and their joint covariance structure. \Cref{sec:trunc} considers the consequences of spectral truncation, and \Cref{sec:bayes} develops Bayesian computation and low-rank predictive inference. \Cref{sec:examples} presents relevant special cases for spheres and triangulated surfaces. Simulation experiments are reported in \Cref{sec:experiments}, followed by the wearable-trajectory and cortical-surface applications in \Cref{sec:app}. \Cref{sec: discussion} concludes with a discussion. Proofs and computational details are provided in the appendices.
Developing joint probabilistic inference for these differential processes raises key statistical challenges that cannot be resolved within standard Euclidean frameworks. Because ambient coordinates distort non-Euclidean geometry and local chart representations yield chart-dependent derivatives, formalizing probability distributions for spatial rates of change requires an intrinsic geometric foundation. Concepts from differential geometry enter our framework out of statistical necessity: framing directional gradients along smooth tangent vector fields using the Riemannian metric tensor $g$ allows intrinsic differentiation to act as a linear operator on the random field. This geometric foundation enables us to establish conditions for mean-square differentiability, derive closed-form cross-covariance functions $\operatorname{Cov}\{Z(x), D_V Z(x_0)\}$, and ensure valid joint distributions for $(Z, D_V Z)\T$. Differential geometry thus provides the exact distribution theory required for coherent  predictive inference and uncertainty quantification on curved spatial domains. The remainder of the paper is organized as follows: \Cref{sec:gp-manifolds} introduces Gaussian processes on manifolds, \Cref{sec:cont,sec:diff} establish the required mean-square smoothness and joint cross-covariance structures, \Cref{sec:bayes} details computation, and \Cref{sec:app} presents our applications to wearable physical trajectories and neuroimaging.

\section{Gaussian processes on manifolds}\label{sec:gp-manifolds}
We begin with some notation. Throughout this paper, we assume that $\MM$ is a compact $p$-dimensional Riemannian manifold, {  possibly with a smooth boundary $\partial\MM$, equipped} with Riemannian metric, $g$. {  If $\partial\MM\ne \emptyset$, we consider} the Laplace-Beltrami (LB) operator, $\Delta_g$. Let $\{\lambda_l,f_l\}_{l=0}^\infty$ be the spectrum of $-\Delta_g$, with {  $0=\lambda_0\leq \lambda_1\leq\lambda_2\leq\ldots$,} where {  $-\Delta_gf_l=\lambda_l f_l$, in the interior of $\MM$, $\MM^\circ=\MM\setminus\partial\MM$ and $\partial_n f_l=0$ on $\partial\MM$, where $n$ denotes the outward unit co-normal within $\MM$ to $\partial\MM$. The eigenfunctions are taken to be orthonormal in $L^2(\MM,dV_g)$, where $dV_g$ is the Riemannian volume element. If $\partial\MM=\emptyset$, this reduces to the usual LB system. We specifically focus on Neumann boundary conditions}. Compactness is crucial, as the spectrum is not necessarily discrete for a non-compact manifold posing challenges for constructing a valid covariance function. % We refer the reader to \cite{azangulov2024stationary,azangulov2024stationary2}, where mathematical constructions of kernels are discussed on non-compact manifolds, Lie groups, and their homogeneous spaces. The ensuing results are not directly extensible to a general non-compact manifold. 
We restrict our scope to compact Riemannian manifolds  with a smooth boundary in this paper. 

%{  We assume $\MM$ is a compact Riemannian manifold.}
% SB: the last sentence is repetitive, you can delete if you like
% AH: Which one? The one in purple? 
% SB: yes, the one in purple, it only adds the boundary bit to \MM

% possibly w/o boundary?
% by default manifold = manifold without boundary. Otherwise we need to specify manifold with boundary. We can clarify here.
%ok

Let $Z(x) \sim \GP(0, K(\cdot, \cdot))$ be a zero-centered GP endowing a probability law on the uncountable set $\{Z(x) : x\in \MM\}$, where $K(x,x')$ is a positive definite covariance function. $K$ is said to be isotropic if $K(x,x')=K(d_{\mathcal{M}}(x,x'))$, % i.e. it is 
a function of $d_{\mathcal{M}}(x,x')$, the geodesic distance between $x$ and $x'$. We investigate kernels, $K(x, x')$ of the Mat\'ern-type:
\begin{align}
 K(x,x') &= \frac{\sigma^2}{C_{\nu,\alpha}}\sum_{l=0}^\infty \left(\alpha^2+\lambda_l\right)^{-\nu-\frac{p}{2}}f_l(x)f_l(x'),\label{eqn:Matern}\\
 K(x,x') &= \frac{\sigma^2}{C_{\infty,\alpha}}\sum_{l=0}^\infty e^{-\frac{\lambda_l}{2\alpha^2}}f_l(x)f_l(x'),\label{eqn:RBF}
\end{align}
where $\theta = \{\sigma^2,\alpha,\nu\}$ are parameters. Specifically, $\nu$ is often termed the smoothness or the fractal parameter. { The normalizing constants, $C_{\cdot,\alpha}$ are such that the average variance over $\MM$ satisfies ${\rm vol}_g(\MM)^{-1}\int_{\MM} K(x,x)\; {\rm d}V_g(x) = \sigma^2$ and ${\rm vol}_g(\MM)=\int_\MM dV_g$.} 
% SB: I already defined Riemannian volume at the start of section 2
% AH: Ok, but the integration should be with respect to that volume, not dx, right?
% SB: integrating K should be wrt to the volume, yes, maybe move that deinition of volume up where $V_g$ first appears--ok so I replaced {\rm d}x with {\rm d}V_g(x). Please double check.
% SB: yes saw that, looks good to me

%is a normalizing constant depending on $\alpha$ and $\nu$. 
{  The covariance in \Cref{eqn:Matern} is the {\emph spectral Whittle-Mat\'ern} and \Cref{eqn:RBF} is the squared exponential covariogram or, the radial basis function (RBF). The spectral representation of the process being, 
\begin{equation}\label{eq:spectral}
    Z(x)=\sum\limits_{l=0}^{\infty}\varphi_lf_l(x), \; \quad \varphi_l\given \theta \stackrel{ind}{\sim}N(0,q_l(\theta)),
\end{equation}
where $q_l(\theta) = \frac{\sigma^2}{C_{\nu,\alpha}}\left(\alpha^2+\lambda_l\right)^{-\nu-\frac{p}{2}}$ or, $q_l(\theta)=\frac{\sigma^2}{C_{\infty,\alpha}}e^{-\frac{\lambda_l}{2\alpha^2}}$, depends on the respective kernel choice and $N(\cdot,\cdot)$ denotes the univariate Gaussian distribution.}

\noindent\textbf{Manifolds with boundary.}
{  For a $p$-dimensional manifold $\MM$, there are two tangent spaces at the boundary: $T_x(\MM)$, which is a $p$-dimensional tangent space; and $T_x(\partial\MM)$ which is $(p-1)$-dimensional with $T_x(\partial\MM)\subset T_x\MM$. At an interior point $x\in \MM^\circ$, every $v\in T_x\MM$ locally admits a smooth curve $\gamma:(-\delta,\delta)\to\MM$ satisfying $\gamma(0)=x$ and $\Dot{\gamma}(0)=v$. The differential quantities at interior points are defined through the usual two-sided limit, $t\to 0$, which corresponds locally to movement in both $v$ and $-v$ directions. At a boundary point $x \in \MM$, let $\eta_x=-n_x\in T_x\MM$ denote the inward-pointing unit co-normal to $\partial\MM$, characterized by $\|\eta_x\|=1$ and $\<\eta_x,w\>_g=0$, for every $w \in T_x(\partial\MM)$. The tangent space admits the orthogonal decomposition, $T_x\MM=T_x\partial\MM\oplus{\rm span}\{\eta_x\}$. As a result, every $V_x\in T_x\MM$ has the unique decomposition, $V_x=V_x^{\star}+\<V_x, \eta_x\>_g\eta_x$, where $V_x^{\star}\in T_x(\partial \MM)$. We define the inward tangent half-space by $T^{+}_x\MM=\{V_x\in T_x\MM:\<V_x,\eta_x\>_g\geq 0\}$. For $V_x\in T_x^{+}\MM$, an {\em admissible} (smooth) {\em curve}, $\gamma:[0,\delta)\to \MM$, satisfies $\gamma(0)=x$ and $\dot{\gamma}(0)=V_x$. If $v$ has a non-zero inward component, such a curve cannot be generally extended to negative $t$ while remaining in $\MM$, differential quantities are therefore defined using a one-sided limit $t\downarrow0$. If $V_x\in T_x\partial\MM$, the curve may be chosen locally within $\partial\MM$ on $(-\delta, \delta)$ and the usual two-sided limit remains viable in directions tangent to the boundary. For more details the interested reader can see \cite{lee_introduction_2012}, pp. 25--28, Proposition 3.12, p. 58.
}

\section{Continuity and Differentiability}
\subsection{Continuity}\label{sec:cont}
We develop inference on the gradient of $Z(\cdot)$ in the {\em mean-squared} sense. {  The processes in our scope are specified using the spectral Mat\`ern and RBF kernels in \Cref{eqn:Matern,eqn:RBF} constructed from either the usual or the Neumann LB eigen-system.} %We establish conditions for such processes to be mean-squared continuous (MSC) and $k$-th order mean-squared differentiable ($k$-MSD). We define continuity, which is followed by definitions and results concerning differentiability of the first and second orders. We connect the notion of process smoothness with parameters specifying the above covariance kernels. 
In the Euclidean domain, a GP specified using a Mat\'ern kernel is $\ceil{\nu}-1$ times mean-square differentiable ($(\ceil{\nu}-1)$--MSD) \citep[see, e.g.,][]{adler1981geometry,williams2006gaussian,banerjee2003directional}. We establish a similar result for Riemannian manifolds using mean-squared continuity %for process realizations 
on $\MM$. 

{ 
\begin{definition}
We define $Z$ as mean-square continuous (MSC or, 0-MSD) at $x\in\MM^\circ$ if $\lim\limits_{t \to 0}\EE(Z(\gamma(t))-Z(x))^2=0$ for any geodesic $\gamma$ with $\gamma(0)=x$. If $x\in\partial\MM$, we say that $Z$ is MSC at $x$ if $\lim\limits_{t \downarrow 0}\EE(Z(\gamma(t))-Z(x))^2=0$, for every admissible smooth curve, $\gamma:[0,\epsilon)\to \MM$. The process $Z$ is MSC on $\MM$ if it is MSC at any $x\in\MM$.   
\end{definition}
}

We %offers sufficient conditions for MSC when considering kernels. It 
connect MSC with the smoothness parameter $\nu$ in the spectral Whittle-Mat\'ern kernel. It's proof (see \Cref{apdx:proof_sec:cont}) shows the role played by compactness of $\MM$. 
{ 
\begin{theorem}\label{thm:MSC}
If $\nu>\frac{p-1}{2}$, then $Z$ is MSC. The RBF in \Cref{eqn:RBF} ensures that $Z$ is MSC. 
\end{theorem}
}

We compare this result with its Euclidean analogue: continuity of process realizations arising from a GP with a spectral Whittle-Mat\'ern kernel is solely determined by the smoothness (or fractal) parameter $\nu$. For example, the exponential kernel ($\nu=1/2$) produces MSC realizations. For compact Riemannian manifolds, \Cref{thm:MSC} posits that continuity is determined by both $\nu$ and the dimension, $p$, of $\MM$. The next result considers isotropic covariance functions, $K(x,x')=K(d_\MM(x,x'))$, where $d_\MM$ is the geodesic distance on $\MM$. The continuity of process realizations is determined by the behavior of $K$ near the origin. The following result resembles its Euclidean counterpart \citep[see, e.g.,][for almost sure, mean square theory for smoothness and further developments respectively]{kent1989continuity,stein1999interpolation,banerjee2003smoothness} and is true for {\em any} isotropic $K$.
\begin{proposition}\label{prop:MSC_iso}
If $K$ is isotropic, then $Z$ is MSC if and only if $K$ is continuous at $0$. 
\end{proposition}

\subsection{Differentiability}\label{sec:diff}
{ 
Mean-square differentiability provides the probabilistic regularity required for the differential-process inference developed here. In particular, $D_VZ(x)$ is defined as an $L^2$ limit and therefore constitutes a well-defined random variable without requiring almost-sure differentiability of individual sample paths. For Gaussian $Z$, collections of such derivatives are jointly Gaussian with the parent process, with their finite-dimensional distributions determined by derivatives of the covariance kernel \citep[see, e.g.][]{stein1999interpolation}. Our objective is inference for this stochastic differential process rather than pathwise differentiation of individual realizations. Sample-path differentiability constitutes a stronger regularity property and is not required for the joint probabilistic construction considered below.
}

% Mean-square differentiability of process realizations in $\MM$ is studied through two differential processes (a) the derivative process, and (b) the curvature process. They require process realization to be once and twice differentiable respectively in the mean-squared sense. We first consider formalizing the theory for derivative processes. In what follows, the process is assumed to be MSC. Derivatives along a direction are often the quantity of scientific interest. In the Euclidean case, this is achieved using directional derivatives that project the vector of partial derivatives along a chosen direction. The analogue of directional derivatives on manifolds is somewhat unclear and needs formalization for our subsequent developments. 

\subsection{The Directional Derivative Process}\label{sec:deriv}
%In this section, we study the differentiability as well as the derivative of a GP, called the derivative process. 
{ 
Let $\mathfrak{X}(\MM)$ denote the space of all smooth vector fields on $\MM$ and $\mathfrak{X}^{+}(\MM)=\{V\in \mathfrak{X}(\MM):V(x)\in T^{+}_x\MM, x\in \partial\MM\}$. We define {\em directional derivative} with respect to a vector field. The directional derivative process captures the rate of change in process realization $Z(x)$ at $x\in\MM$ along the vector field $V$.
\begin{definition}
The directional derivative process of $Z(x)$ with respect to $V(x)$, denoted by $D_VZ(x)$, $x\in\MM$ is defined as
\begin{equation}
    D_VZ(x)\coloneqq\begin{cases}
    \lim\limits_{t\to 0} \frac{Z(\exp_x(tV(x)))-Z(x)}{t}, & x\in \MM^\circ, V\in \mathfrak{X}(\MM),\\
    \lim\limits_{t\downarrow 0} \frac{Z(\gamma_{x,V}(t))-Z(x)}{t}, & x\in \partial\MM,  V\in \mathfrak{X}^{+}(\MM),
    \end{cases}
\end{equation}
where $\exp_x(\cdot):T_x\MM\to\MM$ is the Riemannian exponential map, where $V(x)\in T_x\MM^\circ$ and $V_x\in T^{+}_x\MM$ at boundary points and $\gamma_{x,V}$ is any admissible smooth curve satisfying $\gamma_{x,V}(0)=x$, $\dot{\gamma}_{x,V}(0)=V(x)$.
\end{definition}
% Alternative definition--for future pirposes
% Let $V\in\mathfrak{X(\MM)}$ be a smooth vector field on $\MM$, where $\mathfrak{X}(\MM)$ denotes the space of all smooth vector fields on $\MM$. For any $x\in\MM$ and $\epsilon>0$, let $\gamma_V:[0, \epsilon)\to\MM$ be the unique geodesic starting at $x$ with initial velocity $\gamma'_V(0)=V(x)$, so that $\gamma_V(t)=\exp_x(tV(x))$, where $\exp_x(\cdot):T_x\MM\to\MM$ is the Riemannian exponential map and $T_x\MM$ is the tangent space at $x\in\MM$, $V(x)\in T_x\MM$. The directional derivative of $Z$ along the geodesic flow of $V$ at $x$ is defined as
% \begin{equation}
%     D_VZ(x)\coloneqq\lim_{t\to 0} \frac{Z(\gamma_V(t))-Z(x)}{t}\;,
% \end{equation}
% \end{definition}

\begin{definition}\label{def:1-MSD} 
The process $Z$ is 1-MSD at $x\in\MM^\circ$ if for any geodesic $\gamma$ with $\gamma(0)=x$, $\lim\limits_{t \to 0}\EE\left(D_VZ(\gamma(t))-D_VZ(x)\right)^2=0$ and, at $x\in\partial\MM$, if $\lim\limits_{t \downarrow 0}\EE\left(D_VZ(\gamma(t))-D_VZ(x)\right)^2=0$ for any admissible curve $\gamma$ with $\gamma(0)=x$. The process $Z$ is said to be 1-MSD if it is 1-MSD at every $x\in\MM$.   
\end{definition}

% \begin{definition}\label{def:1-MSD} The process $Z$ is 1-MSD at $x\in\MM$ if for every velocity vector $v\in T_x\MM$ and its corresponding maximal geodesic $\gamma$ satisfying $\gamma(0)=x$ and $\gamma'(0)=v$, we have $\lim\limits_{t \to 0}\EE\left(D_V(Z(\gamma(t))-D_VZ(x)\right)^2=0$, where $V\in \mathfrak{X}(\MM)$ is a fixed smooth vector field. The process $Z$ is 1-MSD if this holds for any $x\in\MM$. 
% \end{definition}

\begin{theorem}\label{thm:1MSD}
The process $Z$ is 1-MSD if $\nu>\frac{p+1}{2}$. The GP defined using the RBF kernel in \Cref{eqn:RBF} is 1-MSD. 
\end{theorem}

\begin{proposition}[Directional derivative process at a Neumann boundary]\label{prop:MSD-boundary}
    Suppose $Z$ is 1-MSD and term-wise differentiation of \Cref{eq:spectral} is valid then, for every $x\in \partial\MM$, $D_\eta Z(x)=0$. Moreover, for $V(x) \in T^{+}_x\MM$, we have $D_VZ(x)=D_{V^{\star}}Z(x)$, where % $V(x)=V(x)^{\rm T}+\<V(x),\eta_x\>_g\eta_x$, 
    $V(x)^{\star}\in T_x(\partial\MM)$. 
\end{proposition}

We discuss important special cases within our scope. In case $\MM=[0,L]$, an interval, the boundary is a discrete set, $\partial\MM=\{0,L\}$ and $T_0(\partial\MM)=T_L(\partial\MM)=\{0\}$, implying the \emph{absence} of a non-zero tangential vector field, therefore $V^{\star}(x)\coloneqq0$ implying $D_VZ(0)=D_VZ(L)=0$. This is the case in \Cref{sec:fitrec}. For a two dimensional manifold or, surface, the boundary $\partial\MM$ is one dimensional. In this case, non-zero vector fields, $V^{\star}$ exist and $D_\eta Z(x)=0$ does \emph{not} require $D_{V^{\star}}Z(x)=0$, which is the case in \Cref{sec:abide}.
}In the remainder of this paper, we exclude the trivial case where, $V\equiv 0$, the zero vector field. Note that $D_VZ$ is well-defined in Definition~\ref{def:1-MSD}, where the geodesic, $\gamma(t) = \exp_x(tV(x))$. To simplify notation, we denote $V_x\coloneqq V(x)$. Ensuing developments will refer to $D_VZ$ as the derivative process, omitting ``directional'' to retain simplicity. The term ``gradient'' will be used to denote the mathematical operation.

\begin{lemma}\label{lem:Taylor_1}
If $Z$ is 1-MSD, { then for  $v\in T_x\MM$}, $Z(\exp_x(tv))=Z(x)+tD_VZ(x)+ r(x,tv)$, $\displaystyle \lim\limits_{t\to0}\frac{r(x,tv)}{t}=0$.
\end{lemma}

Our subsequent developments will rely on cross-covariance functions of joint processes. For elucidation purposes, let $W(x) = (W_1(x),\ldots,W_q(x))\T$ be a $q\times 1$ stochastic process, where each $W_i(x)$ is a real-valued stochastic process over $\MM$. For GPs, this process is specified completely using a mean function, $\mu_i(x) := \mathbb{E}[W_i(x)]$ and a $q\times q$ matrix-valued cross-covariance function, $C(x,x') = (C_{ij}(x,x'))$, where each element is $C_{ij}(x,x') = \mbox{Cov}(W_i(x), W_j(x'))$ for $i,j=1,\ldots,q$. Although there is no loss of generality in assuming the process mean to be zero by absorbing the mean into a separate regression component in the model, as we will do here, modeling the cross-covariance function requires care. From its definition, $C(x,x')$ does not need to be symmetric but must satisfy $C(x,x')\T = C(x',x)$. Also, since $\mbox{Var}\left(\sum_{k=1}^n a_{k}\T W(x_k)\right) > 0$ for any set of input vectors $\{x_1,\ldots,x_n\}$ and $q\times 1$ vectors $a_1,\ldots,a_n$, not all zero, we obtain $\sum_{i,j=1}^n a_i\T C(x_i,x_j)a_j > 0$, which implies that the $nq\times nq$ matrix $[C(x_i,x_j)]$ is positive definite. Characterizations of cross-covariance matrices in Euclidean domains are well-known and have also been investigated on spheres by \cite{porcu2016spatio}. Envisioning the joint process $\left(Z(x),D_VZ(x)\right)$, our assumption $Z\sim GP(0,K)$ has some immediate consequences. The following lemmas show that $D_VZ$ is also a GP and its covariance function is determined by $K$ in analytic form. The inner product is in the Riemannian sense, i.e., $\langle\cdot,\cdot\rangle=g(\cdot,\cdot)$.
\begin{lemma} \label{lem:D_VZ_cov}
The derivative process $D_VZ(x)$ is a valid GP on $\MM$ with mean function $\langle\nabla \mu(x),V_x\rangle$ and covariance function 
 \begin{equation}
     K_V(x,x') = \cov(D_VZ(x),D_VZ(x')) = (\nabla_{12} K(x,x'))(V_x,V_{x'}).
 \end{equation}
 where $x'\in \MM$ is another point, $\nabla_{12}$ represents the partial gradient of $K$, which is a function on the product manifold $\MM\times\MM$ with respect to the first and second coordinates.
\end{lemma}

\begin{lemma}\label{lem:ZD_VZ}
The covariance between the process $Z(x)$ in $x\in \MM$ and the derivative $D_VZ(x')$ in $x'\in\MM$ is given by
\begin{equation}\label{eqn:ZD_VZ}
\cov(Z(x),D_VZ(x'))= \nabla_2 K(x,x')(V_{x'}),
\end{equation}
where $\nabla_2$ denotes the partial gradient of $K$ with respect to the second coordinate.
\end{lemma}

The above lemmas combine to yield
\begin{equation}\label{eq:gp-grad}
    \begin{pmatrix}Z(x)\\D_VZ(x')\end{pmatrix}\sim GP\left(\left[\begin{array}{c}
        \mu(x)  \\
        \<\nabla \mu(x'),V_{x'}\>
    \end{array}\right], \left[\begin{array}{cc}
        K(x,x') & \nabla_2K(x,x')(V_{x'}) \\
        \nabla_2K(x',x)(V_{x'}) & (\nabla_{12} K(x',x'))(V_{x'},V_{x'})
    \end{array}\right]\right).
\end{equation} 
Define the conditional mean, $\mu_{D_V}=\<\nabla \mu(x'),V_{x'}\>+\nabla_2K(x',x)(V_{x'})K(x,x')^{-1} (Z(x)-\mu(x))$ and $\Sigma_{D_V} = K_V(x',x')-\nabla_2K(x',x)(V_{x'})K(x,x')^{-1}\nabla_2K(x,x')(V_{x'})$ then, we have $D_VZ(x')\mid Z(x)\sim N\left(\mu_{D_V},\Sigma_{D_V}\right)$ is a valid probability density. Using \Cref{lem:D_VZ_cov}, $K_V(x',x')=(\nabla_{12} K(x',x'))(V_{x'},V_{x'})$. In particular, when $K$ is the spectral Whittle-Mat\'ern or RBF kernel, the covariance $K_V$ and the covariance between $Z$ and $D_VZ$ admit simpler forms:
\begin{corollary}\label{cly:ZD_VZMat}
The cross-covariance for the spectral Whittle-Mat\'ern kernel is given by
\begin{align*}
    K_V(x,x')&=\frac{\sigma^2}{C_{\nu,\alpha}}\sum_{l=0}^\infty(\alpha^2+\lambda_l)^{-\nu-\frac{p}{2}}\nabla f_l(V_x)\nabla f_l(V_{x'})\; \mbox{ and } \\ 
    \cov(Z(x),D_VZ(x')) &= \frac{\sigma^2}{C_{\nu,\alpha}}\sum_{l=0}^\infty (\alpha^2+\lambda_l)^{-\nu-\frac{p}{2}}f_l(x)\nabla f_l(V_{x'})\;.
\end{align*}
The cross-covariance for the RBF kernel is given by
\begin{align*}
K_V(x,x')&=\frac{\sigma^2}{C_{\nu,\alpha}}\sum_{l=0}^\infty e^{-\frac{\lambda_l}{2\alpha^2}}\nabla f_l(V_x)\nabla f_l(V_{x'})\;\mbox{ and }\\ 
\cov(Z(x),D_VZ(x')) &= \frac{\sigma^2}{C_{\nu,\alpha}}\sum_{l=0}^\infty e^{-\frac{\lambda_l}{2\alpha^2}}f_l(x)\nabla f_l(V_{x'}).
\end{align*}
\end{corollary}

For isotropic covariance functions, we have the following equivalent definition. Evidently, $Z$ being isotropic does not necessarily imply that $D_VZ$ is also isotropic.
\begin{proposition}\label{prop:MSD_iso}
Assume $K$ is isotropic, $Z$ is 1-MSD if and only if $\lim_{t\to0} \frac{K(t)-K(0)}{t^2}<\infty$, that is, $K(t)=K(0)+O(t^2)$ for $t\approx 0$, which is equivalent to the condition, $K'(0)=0$. 
\end{proposition}

{ 
\Cref{sec:curv} discusses extensions to higher-order derivatives and associated results. 
}

% If $U\neq V$, the geodesic distance between the exponential maps $\exp_x(tU_x)$ and $\exp_x(sV_x)$ is almost intractable unless the manifold is Euclidean so, finding the exact coefficients in the Taylor expansion up to order-4 is extremely challenging. Fortunately, assumption $U=V$ is not unreasonable, as it is common in the literature for the Euclidean domain \citep[see, e.g.,][]{halder2024bayesian}. The next section elaborates on Bayesian inference for our proposed differential processes.

\section{Truncations}\label{sec:trunc}
Practical applications using covariance functions in \Cref{eqn:Matern,eqn:RBF} require truncation. The $T$-smallest modes are selected for the Neumann spectrum and $\widetilde{K}(x,x')=\sum\limits_{l=0}^{T-1}a_l f_l(x)f_l(x')$, where {  $a_l=\frac{\sigma^2}{C_{\nu,\alpha}}(\alpha^2+\lambda_l)^{-\nu-p/2}$ for the spectral Whittle-Mat\'ern and $a_l=\frac{\sigma^2}{C_{\infty,\alpha}}e^{-\frac{\lambda_l}{2\alpha^2}}$ for RBF. For the spectral representation in \Cref{eq:spectral} define, $\widetilde{Z}(x)=\sum\limits_{l=0}^{T-1}\varphi_l f_l(x)$, where $\tvarphi\given \theta\sim \mathcal{N}_T(0,Q_\theta)$, where $\tvarphi=(\varphi_0,\ldots,\varphi_{T-1})\T$ and $Q_\theta = {\rm diag}\{q_0(\theta),\ldots,q_{T-1}(\theta)\}$ is a $T\times T$ diagonal matrix. }

\begin{theorem}[Oversmoothing]\label{thm:truncation}
% f $0<\widetilde{K}<\infty$ then it 
The process $\widetilde{Z}(x)$ is always MSC and 1-MSD, regardless of $a_l$ or any other parameter.
\end{theorem}
{ 
% Spectral truncation over-smooths process realizations as seen in the result. 
\begin{theorem}[Error Bounds]\label{thm:truncation_error}
Let $K_{(r,s)}$ denote the covariance obtained by applying directional differential operators or orders $r$ and $s$ to the two arguments of $K$, $\widetilde{K}_{(r,s)}$ be its truncated approximation and $\widetilde{R}_{r,s}(x,x')=K_{(r,s)}(x,x')-\widetilde{K}_{(r,s)}(x,x')$ be the remainder. Under appropriate uniform eigenfunction bounds up to order $r \vee s$, if $\nu>\frac{p+r+s-1}{2}$ then, 
\begin{equation}
    \|\widetilde{R}_{r,s}(x,x')\|_\infty\leq C_{r,s}T^{-\frac{2\nu-p-r-s+1}{p}},
\end{equation}
for sufficiently large $T$, where $C_{r,s}=C_{r,s}(\MM,V,\alpha,\nu,\sigma^2)$.
\end{theorem}

Higher-order differential information converges more slowly under spectral truncation; consistent with this observation, our simulations show substantially greater sensitivity of derivative inference than process reconstruction to the truncation level (see \Cref{tab:trunc-sim}).

\begin{theorem}[Strict Positive Definiteness under Spectral Truncation]\label{thm:truncation_pd} 
Let $X_n=\{x_1,\ldots,x_n\}\subset\MM$ and $X_m^*=\{x_1^*,\ldots,x_m^*\}\subset\MM$ denote the locations at which the parent process is observed and the differential processes are evaluated, respectively. Define $W^*=\begin{pmatrix}
Z(X_n)\\
D_VZ(X_m^*)
\end{pmatrix}
\in\mathbb R^{n+m}$ and let $\Sigma_{W^*}=\operatorname{Cov}(W^*)$ be its covariance matrix. Assume that $\Sigma_{W^*}$ is strictly positive definite. The corresponding truncated joint process, $\widetilde{W}^*=
\begin{pmatrix}
\widetilde{Z}(X_n)\\
D_V\widetilde{Z}(X_m^*)
\end{pmatrix}$ and $\widetilde{\Sigma}_{W^*}=\operatorname{Cov}(\widetilde{W}^*)$. Suppose that the error bounds from \Cref{thm:truncation_error} holds. If $\nu>\frac{p+1}{2}$, then there exists a constant $C_W>0$, independent of $T,n$, and $m$, such that $\left\|\Sigma_{W^*}-\widetilde{\Sigma}_{W^*}\right\|_2\le(n+m)C_W
T^{-(2\nu-p-1)/p}$. Consequently,
$\lambda_{\min}(\widetilde{\Sigma}_{W^*})\ge\lambda_{\min}(\Sigma_{W^*})-(n+m)C_WT^{-(2\nu-p-1)/p}$.
A sufficient condition for strict positive definiteness is
\begin{equation*}
    T>\max\left\{n+m-1,\,\left[\frac{(n+m)C_W}{\lambda_{\min}(\Sigma_{W^*})}\right]^{\frac{p}{2\nu-p-1}}\right\}.
\end{equation*}
\end{theorem}

\Cref{thm:truncation_pd} serves as a theoretical existence guarantee rather than a practical selection criteria. The convergence rate of the truncated joint covariance is determined by the highest-order differential process included in the joint model. For $m=0$, \Cref{thm:truncation_pd} yields a strict positive-definiteness condition for the truncated parent-process covariance on a general $\MM$. This is related to the sphere-specific result of \cite{li2023inference} (see Theorem 15, which considers $S^2$, \citep[also see, e.g.,][]{hitczenko2012some}). % where a minimum-separation bound is used to control the smallest eigenvalue of the un-truncated covariance matrix. 
For large $n$ and $m$, choosing $T\ll n$ is computationally advantageous and motivates the low-rank formulation developed in \Cref{sec:lowrank}. Note that the remainder, $\widetilde{R}(x,x')$, and its differential analogues, are not available in closed-form for a general $\MM$. This is an important distinction from conventional low-rank approximations \citep[see, e.g.][]{finley_improving_2009,banerjee_high-dimensional_2017}.
}

% \section{Model-based (Bayesian) probabilistic inference}
\section{Bayesian Inference}\label{sec:bayes}
Our results in the previous section lay the foundations for the validity of the process $D_VZ(x)$ on $\MM$. Probabilistic inference for $W(x)$ requires the assumption of a distribution for $Z(x)$---we assume $Z(x)\sim GP(\mu(x;\beta),K)$, where $\mu(x;\beta)$ is twice differentiable in $\MM$ and $K$ is a covariance kernel. Let $\Z=\left(Z(x_1),Z(x_2),\ldots, Z(x_N)\right)\T$ be the observed realizations in $\MM$ with mean $m=\left(\mu(x_1,\beta),\mu(x_2,\beta),\ldots,\mu(x_N,\beta)\right)\T$ and covariance $C_{Z,Z}(x_i,x_j) = K(x_i,x_j;\theta)$, where $\theta$ is a vector of parameters specifying $K$, $i,j=1,2,\ldots, N$. At $x_0$, an arbitrary location in $\MM$, probabilistic inference on derivatives and curvatures follows from \Cref{eqn:C_W}, which yields $\Z,D_VZ(x_0)(x_0)\mid\theta \sim \mathcal{N}_{N+1}(m_0,\Sigma_0)$, where $\displaystyle m_0 = \left(
        m\T,
        \<\nabla \mu(x_0),V_{x_0}\>\right)\T$ %,
        %\nabla^2\mu(x_0)(V_{x_0},V_{x_0})\right)\T$ 
        and
\begin{equation}\label{eq:full-posterior}
    \begin{split}
      %  \mu_0 & = \left(\begin{array}{c} \widetilde{\mu}\\ \<\nabla \mu(x_0),V_{x_0}\>\\ \nabla^2\mu(x_0)(V_{x_0},V_{x_0}) \end{array}\right),\\
       \Sigma_0 &= \left[\begin{array}{ccc}
            C_{\Z,\Z} &  C_{\Z,D_VZ} \\
             C_{D_VZ,\Z} & K_V(x_0,x_0) 
        \end{array}\right].
        % \Sigma_0 &= \left[\begin{array}{ccc}
        %     C_{\Z,\Z} &  C_{\Z,D_VZ} & C_{\Z,D^2_VZ}\\
        %      C_{D_VZ,\Z} & K_V(x_0,x_0) & \nabla_{122} K(x_0,x_0)(V_{x_0},V_{x_0},V_{x_0})\\
        %      C_{D^2_VZ,\Z} & \nabla_{122} K(x_0,x_0)(V_{x_0},V_{x_0},V_{x_0}) & K_{V,V}(x_0,x_0)
        % \end{array}\right].
    \end{split}
\end{equation}
Here, $C_{\Z,\Z} =  (C_{Z,Z}(x_i,x_j))_{i,j=1,2,\ldots,N}$ is an $N\times N$ covariance matrix, the specific entries are, $C_{\Z,D_VZ} = \left(\nabla_2K(x_1,x_0)(V_{x_0}),\ldots,\nabla_2K(x_N,x_0)(V_{x_0})\right)\T$ and $C_{D_VZ,\Z} = C_{\Z,D_VZ}\T$. %,  $C_{\Z,D^2_VZ} = (\nabla_{22} K(x_1,x_0)(V_{x_0},V_{x_0}), \ldots, \nabla_{22} K(x_N,x_0)(V_{x_0},V_{x_0}))\T$ and $C_{\Z,D^2_VZ}\T = C_{D^2_VZ,\Z}$.
% \begin{equation*}
% \begin{split}
% C_{\Z,D_VZ} &= \left(\nabla_2K(x_1,x_0)(V_{x_0}),\ldots,\nabla_2K(x_N,x_0)(V_{x_0})\right)\T,\\
% C_{\Z,D^2_{U,V}Z} & = \left(\nabla_{22} K(x_1,x_0)(V_{x_0},V_{x_0}), \ldots, \nabla_{22} K(x_N,x_0)(V_{x_0},V_{x_0})\right)\T,\\
% C_{D_VZ,\Z} &= \left(\nabla_2K(x_0,x_1)(V_{x_1}),\ldots,\nabla_2K(x_0,x_N)(V_{x_N})\right),\\
% C_{D^2_{U,V}Z,\Z} &= \left(\nabla_{22} K(x_0,x_1)(V_{x_1},U_{x_1}), \ldots, \nabla_{22} K(x_0,x_N)(V_{x_N},U_{x_N})\right).
% \end{split}
% \end{equation*} 
% Clearly, the cross-covariance matrix $\Sigma_0$ is not symmetric. 
The posterior predictive distribution for the differential processes at $x_0$ is
\begin{equation}\label{eq:posterior_predicitive}
    P\left(D_VZ(x_0)\mid \Z\right) = \int P\left(D_VZ(x_0)\mid \Z,\theta\right)P\left(\theta\mid\Z\right)d\theta.
\end{equation}
Posterior sampling is performed one-for-one--drawing one instance of $D_VZ(x_0)$ for every posterior sample of $\theta$. The conditional predictive distribution for the differential processes $D_VZ(x_0)\mid\Z,\theta\sim \mathcal{N}_2(\mu_1,\Sigma_1)$ can be obtained from \Cref{eq:full-posterior},
\begin{equation}\label{eq:grad-posterior}
    \begin{split}
        \mu_1 &=  %\left(\begin{array}{c}
        % \<\nabla \mu(x_0),V_{x_0}\>\\
        % \nabla^2\mu(x_0)(V_{x_0},V_{x_0}) 
        % \end{array}\right)
        \<\nabla \mu(x_0),V_{x_0}\>
        + 
        % \left(\begin{array}{c}
        %     C_{D_VZ,\Z} \\
        %     C_{D^2_VZ,\Z}  
        % \end{array}\right)\T C_{\Z,\Z}^{-1}\left(\Z-\widetilde{\mu}\right)
       C_{D_VZ,\Z}\T C_{\Z,\Z}^{-1}\left(\Z-m\right)
        ,\\
        \Sigma_1 &= 
             K_V(x_0,x_0) -
            C_{D_VZ,\Z}\T C_{\Z,\Z}^{-1}
           C_{\Z,D_VZ} 
        % \left(\begin{array}{cc}
        %      K_V(x_0,x_0) & \nabla_{122} K(x_0,x_0)(V_{x_0},V_{x_0},V_{x_0})\\
        %      \nabla_{122} K(x_0,x_0)(V_{x_0},V_{x_0},V_{x_0}) & K_{V,V}(x_0,x_0)
        % \end{array}\right)\\&\pushright{-\left(\begin{array}{c}
        %     C_{D_VZ,\Z} \\
        %     C_{D^2_VZ,\Z}  
        % \end{array}\right)\T C_{\Z,\Z}^{-1}\left(\begin{array}{c}
        %    C_{\Z,D_VZ} \\
        %    C_{\Z,D^2_VZ}
        % \end{array}\right)}
        % \left(\begin{array}{cc}
        %      K_V(x_0,x_0) & \nabla_{122} K(x_0,x_0)(V_{x_0},V_{x_0},V_{x_0})\\
        %      \nabla_{122} K(x_0,x_0)(V_{x_0},V_{x_0},V_{x_0}) & K_{V,V}(x_0,x_0)
        % \end{array}\right)\\&\pushright{-\left(\begin{array}{c}
        %     C_{D_VZ,\Z} \\
        %     C_{D^2_VZ,\Z}  
        % \end{array}\right)\T C_{\Z,\Z}^{-1}\left(\begin{array}{c}
        %    C_{\Z,D_VZ} \\
        %    C_{\Z,D^2_VZ}
        % \end{array}\right)}
        .
    \end{split}
\end{equation}

Of particular relevance is inference for $W(x_0)$ conditional on the scattered response $Y(x)$, or partially observed process realizations in $\MM$, generated from a latent process $Z(x)$ corrupted by a white noise disturbance. For example, $Y(x)$ can be modeled as % a linear regression,  
\begin{equation}\label{eq:hierarchical_model}
    Y(x) = \mu(X(x); \beta)+Z(x)+\epsilon(x),
\end{equation}
where $X(x)$ are covariates, $\beta$ are regression coefficients, $Z(x)\sim GP(0,K(\cdot;\theta))$ is zero-centered and $\epsilon(x)$ is white-noise. Note that $Y(x)$ itself may not satisfy the process smoothness assumptions available for $Z(x)$. Nevertheless, we can infer the rates of change for $Z(x)$ conditional on the data obtained from $Y(x)$. We denote $\tY =\left(Y(x_1),\cdots,Y(x_N)\right)\T$ as a realization of $Y$.

{ 
\noindent\textbf{Low-rank Predictive Inference.}\label{sec:lowrank}
Define, $\tF  =\begin{pmatrix}
    f_0(x_1) &\ldots &f_{T-1}(x_1)\\
    \vdots &\ddots & \vdots\\
    f_0(x_N) &\ldots &f_{T-1}(x_N)
\end{pmatrix}$, a $N\times T$ matrix of eigenfunctions. We have a low-rank representation of the model in \Cref{eq:hierarchical_model},   $\tY =\mu(\tX;\beta)+\Z+\epsilon$, where $\Z=\tF  \tvarphi$ and $\epsilon\sim \mathcal{N}_N(0,\tau^2I_N)$. We assign priors on $\theta=(\sigma^2,\alpha,\tau^2,\nu)\T$, which leads to %the the following posterior,
\begin{equation}\label{eqn:posterior}
\begin{split}
    P(\theta\mid\tY )&\propto IG(\alpha\mid a_\alpha,b_\alpha) \times IG(\sigma^2\mid a_\sigma,b_\sigma)\times IG(\tau^2\mid a_\tau,b_\tau) \times U(\nu\mid a_\nu,b_\nu)\\&\pushright{\times\; \mathcal{N}_p(\beta\mid\mu_{\beta},\Sigma_{\beta})\times\mathcal{N}_N(\tY \mid \mu(\tX;\beta), C_{\Z,\Z}+\tau^2 I_N)},
\end{split}
\end{equation}
where $\sigma^2$ is the process variance, $\alpha$ is the length-scale, $\tau^2$ is the nugget (variance of a white noise process), $\nu$ is the smoothness parameter, $IG$ denotes the inverse Gamma distribution and $U$ denotes the uniform distribution. The posterior in \Cref{eqn:posterior} features a computationally stable, collapsed version obtained by marginalizing $\Z$. Note that the precision matrix can be efficiently computed using the Woodbury formula, $(C_{\Z,\Z}+\tau^2 I_N)^{-1}=(\tF  Q_\theta\tF  \T+\tau^2 I_N)^{-1}=\tau^{-2}-\tau^{-4}\tF  A_\theta^{-1}\tF  \T$, where $A_\theta=Q_\theta^{-1}+\tau^{-2}\tF  \T\tF  $. The determinant lemma yields, $\log|\tF  Q_\theta\tF  \T+\tau^2 I_N|=N\log\tau^2+\log |Q_\theta| + \log|A_\theta|$. The conditional posterior, $\tvarphi\given -\sim \mathcal{N}_T(m_{\tvarphi},V_{\tvarphi})$, where $V_{\tvarphi}=(Q_\theta^{-1}+\tau^{-2}\tF \T\tF)^{-1}$ and $m_{\tvarphi}=V_{\tvarphi}\tau^{-2}\tF  \T(\tY -\tX\beta)$. These computations are required for iterative Bayesian inference and involve inversion of $T\times T$ matrices. Unlike conventional low-rank approaches which are based on conditional expectation at a collection of knots \citep[see, e.g.][]{quinonero-candela_unifying_2005,banerjee2008gaussian,wikle2010low}, we condition over $T$ eigenfunctions of the Laplace–Beltrami operator, which form a geometry-adapted basis fixed by the underlying manifold. % Dimension reduction is therefore achieved by posterior inference on the $T$-dimensional coefficient vector $\tvarphi$, rather than through a set of spatial knots.

Within an Euclidean setting, predictive inference for derivative processes is conventionally done point-wise over a grid which implicitly assumes \emph{independence} in gradients across grid locations \citep[see, e.g.,][]{banerjee2006bayesian,halder2024bayesian,halder_bayesian_2026}. Such restrictions are motivated by computational simplification concerns. In the present context a point cloud replaces the grid. The low-rank representation allows us to perform predictive inference on derivative processes jointly over the grid thereby preserving their dependence structure. Let $\mathcal{X}^0=\{x^0_1,\ldots,x^0_{n_g}\}$ be a point cloud on $\MM$ of size $n_g$ and $D_VZ(\mathcal{X}^0)=\left(D_VZ(x^0_1),\ldots,D_VZ(x_{n_g}^0)\right)\T$. The joint posterior for the differential processes is obtained as $P(D_VZ(\mathcal{X}^0)\mid\tY )=\int P(D_VZ(\mathcal{X}^0)\mid\Z,\theta)\;P(\Z\mid \tY ,\theta)\;P(\theta\mid\tY )\,d\theta \,d\Z$. Define $D_V\tF(\mathcal{X}^0)=\begin{pmatrix}
    D_Vf_0(x^0_1) & \ldots & D_Vf_{T-1}(x^0_1)\\
    \vdots & \ddots & \vdots\\
    D_Vf_0(x^0_{n_g}) & \ldots & D_Vf_{T-1}(x^0_{n_g})
\end{pmatrix}$. It follows that, $D_V\tF(\mathcal{X}^0)\given \tY,\theta \sim \mathcal{N}_{n_g}(D_V\tF(\mathcal{X}^0)m_{\tvarphi}, D_V\tF(\mathcal{X}^0)V_{\tvarphi}D_V\tF(\mathcal{X}^0)\T)$. Clearly, we do not have independence since $\cov(D_VZ(x_i^0),D_VZ(x_j^0)\given\tY,\theta)=D_V\tF(x_i^0)V_{\tvarphi}D_V\tF(x_j^0)\T\ne 0$, where $D_V\tF(x_i^0)$ is the $i$-th row of $D_V\tF(\mathcal{X}^0)$. Similar developments follow for higher order derivatives. For larger point-clouds, $n_g\gg T$, although the posterior for derivatives is a degenerate distribution, computationally efficient sampling is done through, $D_VZ(\mathcal{X}^0)=D_V\tF(\mathcal{X}^0)\tvarphi$. 
}
%Note that \Cref{eq:grad-posterior} also benefits from the low-rank representation, $C_{\widetilde{Z}, D_VZ}=\tF  Q_\theta \tF  _V(x_0)$, $K_V(x_0,x_0)=\tF  _V(x_0)Q_\theta \tF  _V(x_0)\T$, $C_{\widetilde{Z}, D^2_VZ}=\tF  Q_\theta \tF  _{V,V}(x_0)$ and $K_{V,V}(x_0,x_0)=\tF  _{V,V}(x_0)Q_\theta \tF  _{V,V}(x_0)\T$, where $\tF  _V(x_0)=\left(D_Vf_{0}(x_0),\ldots,D_Vf_{T-1}(x_0)\right)\T$ and $\tF  _{V,V}(x_0)=\left(D^2_Vf_{0}(x_0),\ldots,D^2_Vf_{T-1}(x_0)\right)\T$.} % The smoothness parameter, $\nu$, for the Mat\'ern kernel is kept fixed. Alternatively, given an observed process on $\MM$, the smoothness can be learned by placing a prior on $\nu$. 
% We conclude by observing that the parameters, $\sigma^2$ and $\tau^2$ are often weakly identified and re-parameterizing the collapsed covariance, $\sigma^2R_{\Z,\Z}+\tau^2I_N=v\{\rho R_{\Z,\Z} + (1-\rho)I_N\}$, where $v = \sigma^2+\tau^2$, $\rho = \frac{\sigma^2}{v}$ and $R_{\Z, \Z}$ is the correlation matrix, is beneficial for achieving faster convergence. Prior distributions can then be placed on $v\sim IG(a_v, b_v)$ and $\rho\sim Beta(a_\rho, b_\rho)$. % Next, we study some computational aspects of implementing \Cref{eqn:posterior} with kernels in \Cref{eqn:Matern,eqn:RBF} within a practical setting.

% DL: I commented the above lines given that we decided to move the next section to the Appendix.

\begin{figure}[t]
    \centering
    \includegraphics[width=0.9\linewidth]{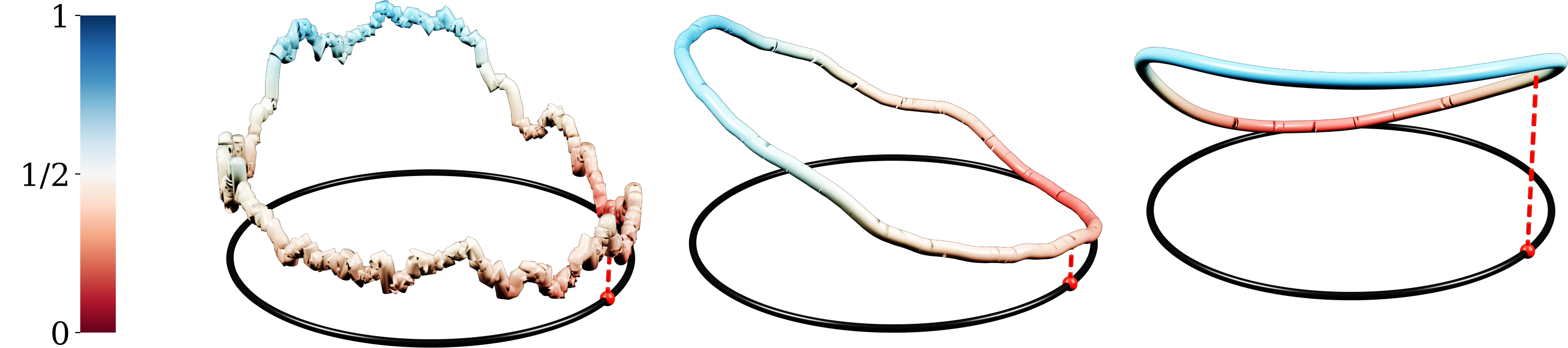}
    
    \vspace*{0.2cm}
    
    \includegraphics[width=0.9\linewidth]{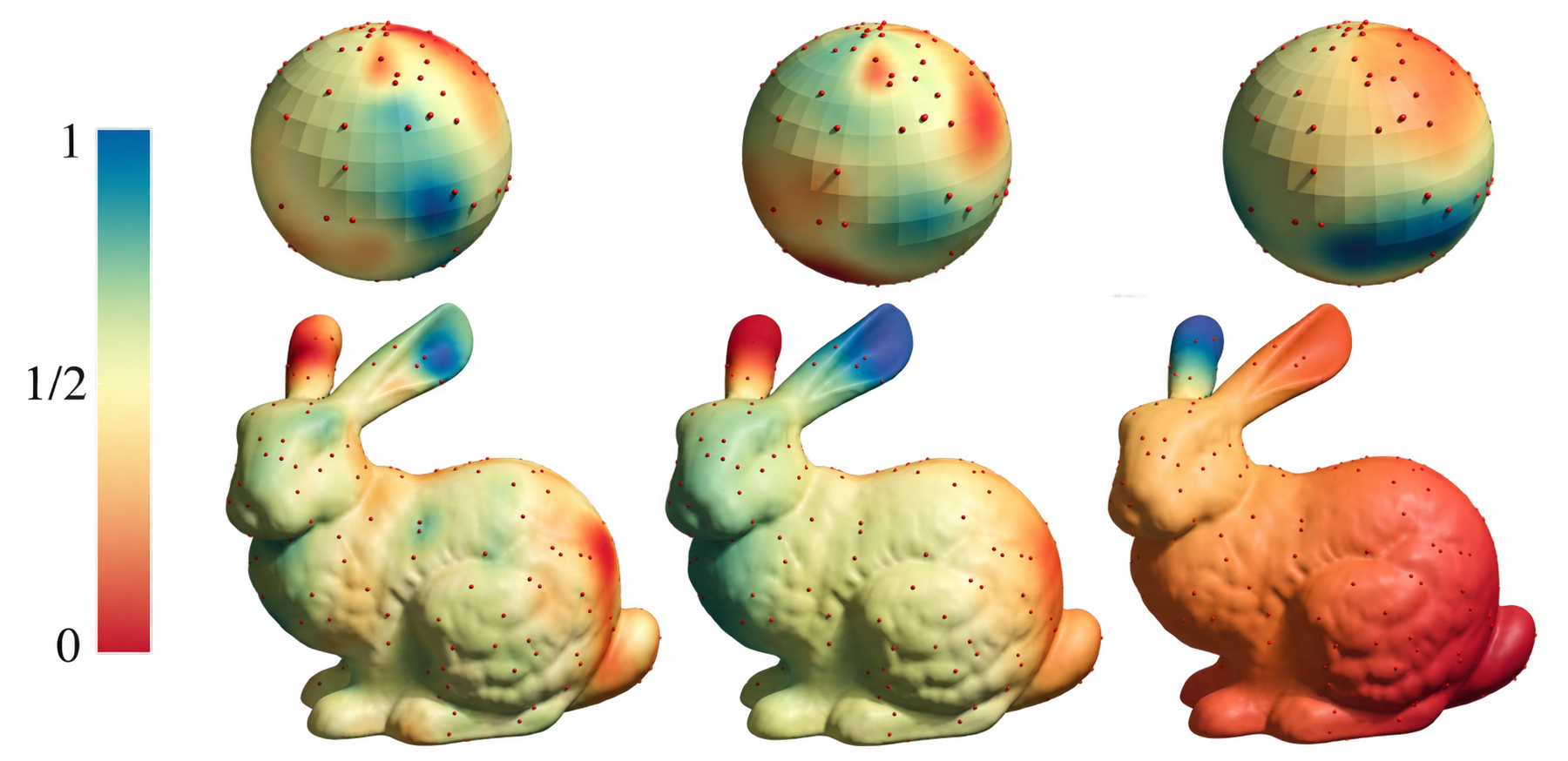}
    % \includegraphics[width=\linewidth]{plots/gp-combined_scale.png}
    % \includegraphics[width=\linewidth]{plots/bunny_sim-side_scale.png}
    % SB & DL: Uncomment the above line to see the same illustration for the Stanford Bunny (https://graphics.stanford.edu/data/3Dscanrep/)
    \caption{Interpolated process realizations on $\mathbb{S}^1$ (top row; height above circle indicates magnitude), $\mathbb{S}^2$ (middle row) and the Stanford Bunny (SB, bottom row) simulated from the Mat\'ern kernel for $\mathbb{S}^1$ (see \Cref{eqn:Matern0_S1,eqn:Matern1_S1,eqn:Matern2_S1}), the Legendre-Mat\'ern kernel for $\mathbb{S}^2$ (see \Cref{tab:S2_Kernels}), and the spectral Whittle-Mat\'ern-type kernel in \Cref{eqn:Matern} respectively with increasing smoothness from left to right. Simulated points are shown as red dots.} % The blue shaded regions indicate negative and red regions positive values.}
    \label{fig:S1andS2}
\end{figure}

\section{Special cases}\label{sec:examples}
Although our results are applicable in any compact Riemannian manifold, here we choose to provide specific examples that are computationally feasible and can be visualized. We first discuss the sphere, $\mathbb{S}^p$, which offers exact trigonometric expressions for eigenfunctions of the Laplacian and a closed-form expression for the exponential map and geodesic distance. This yields an exact spectrum. Next, we investigate surfaces in $\RR^3$ which are represented as triangulated meshes for practical purposes. Here, the manifold $\mathcal{M}$ is an arbitrary object featuring complex geometry, for example, a teapot or a bunny. % is not given in closed-form 
As a result, % eigenpairs, $(\lambda_l,f_l)$ are 
the spectrum is unknown and requires a numerical approximation, which is to be expected in practical applications. \Cref{fig:S1andS2} provides a visual reference for the results in \Cref{thm:MSC,thm:1MSD,thm:2MSD} for $\mathbb{S}^1$, $\mathbb{S}^2$ and the Stanford Bunny (SB), which serves as our example of choice for a surface in $\RR^3$. With increasing $s=0,1,2$ (top left to right) and $\nu=1,2,3$ (middle and bottom left to right), the process realizations become visibly smoother. %We focus on isotropic covariance functions. For $x,x'\in \mathbb{S}^p$, $K(x,x')=K(d_{\mathbb{S}^p}(x,x'))=K(\arccos\langle x,x'\rangle)$. The literature for studies on smoothness of process realizations is restricted to only the great circle \citep[see, e.g.,][pp. 145--146, Theorem 1]{guinness2016isotropic}. The results in the previous sections applied to $\mathbb{S}^p$ produce a global characterization of process smoothness. %We elaborate further in the following results. % pertaining to $\mathbb{S}^p$.

\subsection{Sphere \texorpdfstring{$\mathbb{S}^p$}{Sp}}\label{subsec:sphere} Our results in the previous sections, when applied to $\mathbb{S}^p$, produce a global and intrinsic characterization of smoothness for the process. This {  complements previous studies of smoothness for isotropic random fields on spheres} % differs from previous studies % for studies 
% concerning smoothness of process realizations over spheres which rely on restricting the process to a great circle that is isometric to $[0,1)$ 
\citep[see, e.g.,][pp. 145--146, Theorem 1]{guinness2016isotropic}. We begin with deriving the necessary and sufficient conditions on $K$ for the process to be MSC and 1-MSD %and 2-MSD
on $\mathbb{S}^p$. We assume $Z(x)\sim \GP(0,K)$ for $x\in\mathbb{S}^p$. % The following % slew of 
%results establish the validity of % the process 
%$D_VZ(x)$.

\begin{theorem}\label{thm:MSC1-MSD_iso}
The process $Z$ is MSC if and only if $K$ is continuous at $0$ and $Z$ is 1-MSD if and only if $K(t)=K(0) + \frac{1}{2}K''(0)t^2 + O(t^2)$, i.e., $K'(0)=0$ and $|K''(0)|<\infty$. 
\end{theorem}

Next, we assume that $Z$ is 1-MSD and derive the covariance of the derivative, $D_VZ$.
\begin{theorem}\label{thm:derivative_Sp}
The covariance $K_V(x,x')\coloneqq \cov(D_VZ(x),D_VZ(x'))$ is given by
\begin{equation}\label{eqn:D_VZ_Sp}
\begin{split}
    K_V(x,x')&=-\frac{K'(d)\<V_x,V_{x'}\>}{(1 - \<x,x'\>^2)^{1/2}}\\
    &+ \left(K''(d) - K'(d)\frac{\<x,x'\>}{(1 - \<x,x'\>^2)^{1/2}}\right)\frac{\<V_x,x'\>}{(1 - \<x,x'\>^2)^{1/2}} \frac{\<x,V_{x'}\> }{(1 - \<x,x'\>^2)^{1/2}},
\end{split}
\end{equation}
where $d=d_{\mathbb{S}^p}(x,x')=\arccos(\<x,x'\>)$ is the geodesic distance. % As a consequence, $D_VZ$ 
Note that \Cref{eqn:D_VZ_Sp} is not a function of $d$, and therefore $D_VZ$ is no longer isotropic. 
\end{theorem}

We have resonance with analogous results for process smoothness in the Euclidean space. The kernel of the derivative process derived from a GP with an isotropic kernel is not necessarily isotropic~\citep[see, e.g.,][discussion after eq. (5') and eq. (12)]{banerjee2003directional}. The covariance of the derivative process for $x,x'\in\RR^2$ is $K_V(x,x') = \frac{K'(\|\delta\|)}{\|\delta\|} + \left(K''(\|\delta\|)-\frac{K'(\|\delta\|)}{\|\delta\|}\right)\frac{\delta\delta\T}{\|\delta\|}$, where $\delta=x-x'$ and $||\cdot||$ is the Euclidean distance. We observe that as $x'\to x$, $d\to 0$ and $\left(K''(d) - K'(d)\frac{\<x,x'\>}{(1 - \<x,x'\>^2)^{1/2}}\right) \to 0$, which serves as the rationale behind the representation in \Cref{eqn:D_VZ_Sp}, \Cref{thm:derivative_Sp}. This is crucial for the next result when considering ${\rm Var}(D_VZ(x))=K_V(x,x)$. 
\begin{corollary}\label{cly:K_V_Sp}
    $K_V(x,x)\coloneqq \lim_{t\to0} K_V(x,\exp_x(tu))=-\<V_x,V_x\>K''(0)$ for any $u\in T_x \mathbb{S}^p$, where $T_x\mathbb{S}^p$ is the tangent space at $x$ on $\mathbb{S}^p$. 
\end{corollary}

\begin{theorem}\label{thm:ZD_VZ_Sp}
Let $d=d_{\mathbb{S}^p}(x,x')=\arccos(\<x,x'\>)$, then
\begin{equation}\label{eq:ZD_VZ_Sp}
    \cov(Z(x),D_VZ(x'))= -\frac{K'(d)\<x,V_{x'}\>}{(1-\<x,x'\>^2)^{1/2}}.
\end{equation}
\end{theorem}
\Cref{thm:ZD_VZ_Sp} is required to carry out statistical inference for the joint process %\mathcal{L}Z=
$\left(Z, D_VZ\right)\T$. If joint evaluation of the process and its derivative is desired at the same location $x\in \mathbb{S}^p$, then a special case of \Cref{thm:ZD_VZ_Sp} arises requiring $\cov(Z(x),D_VZ(x))$.
\begin{corollary}\label{cly:ZD_VZ_Sp}
    $\cov(Z(x),D_VZ(x)) = \lim_{t\to0} \cov(Z(\exp_x(tu)),D_VZ(x))=0$ for any $u\in T_x \mathbb{S}^p$. 
\end{corollary}

% The following result establishes the necessary and sufficient conditions for $Z$ to be 2-MSD. We maintain the representation in \Cref{thm:derivative_Sp} observing that $x'\to x$, $d\to 0$ and $\left(K''(d) - K'(d)\frac{\<x,x'\>}{(1 - \<x,x'\>^2)^{1/2}}\right) \to 0$.

The above results concretely establish the existence and validity of the process $D_VZ(x)$ on $\mathbb{S}^p$. Bayesian inference on $D_VZ(x)$ can now run its course following the details in \Cref{sec:bayes} substituting entries in the cross-covariance matrix in \Cref{eqn:C_W} with the quantities derived from the above results. % We provide further details in Section 
% \ah{Insert simulation section number here!} 
% We skip the details %ed formulae that enable inference on the curvature process
% and
% % Instead, we 
% conclude % our exploration 
% with studies of smoothness for commonly used kernels for $\mathbb{S}^1$ and $\mathbb{S}^2$. 
We turn to commonly used kernels on $\mathbb{S}^2$ provided in Table~1, \cite{guinness2016isotropic} \citep[also see,][]{lang2015isotropic}.
\begin{theorem}\label{thm:S2_kernels}
    Differentiability of process realizations that are obtained from commonly used kernels is given in the last two columns of Table \ref{tab:S2_Kernels}.
\end{theorem}

\begin{table}[htb]
\centering
\caption{List of covariance functions on $\mathbb{S}^2$ and their smoothness.}\label{tab:S2_Kernels}
\resizebox{\linewidth}{!}{
    \begin{tabular}{c|c|c|c|c}
    \hline\hline
      Name & Expression & Parameter Values & MSC&1-MSD\\ \hline
    Chordal Mat\'ern &$\left(\alpha^2\sin\left(\frac{t}{2}\right)\right)^\nu K_\nu(\alpha^2\sin(\frac{t}{2}))$ &$\alpha,\nu>0$ &Yes& $\nu>1$\\ \hline
    Circular Mat\'ern &$\sum\limits_{l=-\infty}^\infty (\alpha^2+l^2)^{-\nu-\frac{1}{2}}\exp(ilt)$&$\alpha,\nu>0$ &Yes&No \\ \hline
    Legendre-Mat\'ern &$\sum\limits_{l=0}^\infty (\alpha^2+l^2)^{-\nu-\frac{1}{2}}P_l(\cos t)$&$\alpha,\nu>0$ & $\nu> \frac{1}{2}$ & $\nu> \frac{3}{2}$\\ \hline
    Truncated Legendre-Mat\'ern &$\sum\limits_{l=0}^T (\alpha^2+l^2)^{-\nu-\frac{1}{2}}P_l(\cos t)$&$\alpha,\nu>0$ &Yes & Yes\\ \hline
    Bernoulli &$1+\alpha+\sum\limits_{l\ne 0}|l|^{-2n}\exp(ilt )$ &$\alpha>0$, $n\in\NN$&Yes&No \\ \hline
    Powered Exponential &$\exp(-(\alpha t )^\nu)$ &$\alpha>0$, $\nu\in(0,1]$ &Yes&No \\ \hline
    Generalized Cauchy & $\left(1+(\alpha t )^\nu\right)^{-\tau/\nu}$&$\alpha,\tau>0$, $\nu\in(0,1]$ &Yes&No\\ \hline
    Multiquadric &$(1-\tau)^{2\alpha}/\left(1+\tau^2-2\tau\cos  t \right)^\alpha$ &$\alpha>0$, $\tau\in(0,1)$ &Yes&Yes\\ \hline
    Sine Power &$1-\left(\sin(\frac{ t }{2})\right)^\nu$&$\nu\in(0,2)$&Yes &No\\ \hline
    Spherical &$(1+\frac{\alpha t }{2})(1-\alpha t )_+^2$&$\alpha>0$&Yes &No\\ \hline
    Askey & $(1-\alpha t )_+^\tau$&$\alpha>0$, $\tau\geq 2$&Yes &No\\ \hline
    $C^2$--Wendland &$(1+\tau\alpha t )(1-\alpha t )_+^\tau$ &$\alpha\geq \frac{1}{\pi}$, $\tau\geq 4$&Yes  &Yes\\ \hline
    $C^4$--Wendland &$\left(1+\tau\alpha t +\frac{\tau^2-1}{3}(\alpha t )^2\right)(1-\alpha t )_+^\tau$ &$\alpha\geq \frac{1}{\pi}$, $\tau\geq 6$&Yes  &Yes\\ \hline
\end{tabular}
}
\end{table}
Next, we consider the smoothness of the realizations of the process for $Z(x)\sim \GP(0,K_\nu)$, $x \in\mathbb{S}^{1}$. % we are able to check whether 
We define $u=u(x,x')=\alpha\left(|x-x'|-\frac{1}{2}\right)$ \citep[see, e.g.,][Lemma 4.2]{li2023inference}, where $x$, $x'$ are identified with $x=\exp(2\pi i\theta_x)$, $x'=\exp(2\pi i\theta_{x'})$, respectively, and $|x-x'|=|\theta_x-\theta_{x'}|$. Of interest are closed-form expressions for differing values of $s=0,1,2$, where the fractal parameter is $\nu=s+\frac{1}{2}$:
\iffalse
For $s=0$, $K_{1/2}(x,x')=\frac{\sigma^2}{\cosh\left(\frac{\alpha}{2}\right)}\cosh\left(u\right)$. 

For $s=1$, $K_{3/2}(x,x')=\frac{\sigma^2}{C_{3/2,\alpha}}\left(a_{1,0}\cosh(u)+a_{1,1}u\sinh(u)\right)$, where $a_{1,0}=\frac{2\pi^2}{\alpha^2}\left(1+\frac{\alpha}{2}\coth\left(\frac{\alpha}{2}\right)\right)$, $a_{1,1}= -\frac{2\pi^2}{\alpha^2}$ and $C_{3/2,\alpha}= a_{1,0}\cosh\left(\frac{\alpha}{2}\right)+a_{1,1}\frac{\alpha}{2}\sinh\left(\frac{\alpha}{2}\right)$. 

For $s=2$, $K_{5/2}(x,x')=\frac{\sigma^2}{C_{5/2,\alpha}}\left(a_{2,0}\cosh(u)+a_{2,1}u\sinh(u)+a_{2,2}u^2\cosh(u)\right)$, where $a_{2,0}=\frac{\pi^4}{\alpha^4}\left\{\left(6-\frac{\alpha^2}{2}\right)+3\alpha\coth\left(\frac{\alpha}{2}\right)+\alpha^2\coth^2\left(\frac{\alpha}{2}\right)\right\}$, $a_{2,1}=-\frac{2\pi^4}{\alpha^4}\left(3+\alpha\coth\left(\frac{\alpha}{2}\right)\right)$, $a_{2,2}=\frac{2\pi^4}{\alpha^4}$ and $C_{5/2}=a_{2,0}\cosh\left(\frac{\alpha}{2}\right)+a_{2,1}\frac{\alpha}{2}\sinh\left(\frac{\alpha}{2}\right)+a_{2,2}\frac{\alpha^2}{4}\cosh\left(\frac{\alpha}{2}\right)$.

\fi

\begin{align}
K_{1/2}(x,x')&=\sigma^2\left(\cosh\left(\frac{\alpha}{2}\right)\right)^{-1}\cosh\left(u\right),\label{eqn:Matern0_S1}\\
K_{3/2}(x,x')&=\sigma^2\;C_{3/2,\alpha}^{-1}\left(a_{1,0}\cosh(u)+a_{1,1}u\sinh(u)\right),\label{eqn:Matern1_S1}\\
K_{5/2}(x,x')&=\sigma^2\;C_{5/2,\alpha}^{-1}\left(a_{2,0}\cosh(u)+a_{2,1}u\sinh(u)+a_{2,2}u^2\cosh(u)\right),\label{eqn:Matern2_S1}
\end{align}
where 
\begin{align*}
a_{1,0}&=\frac{2\pi^2}{\alpha^2}\left(1+\frac{\alpha}{2}\coth\left(\frac{\alpha}{2}\right)\right),a_{1,1}= -\frac{2\pi^2}{\alpha^2},C_{3/2,\alpha}=
    a_{1,0}\cosh\left(\frac{\alpha}{2}\right)+a_{1,1}\frac{\alpha}{2}\sinh\left(\frac{\alpha}{2}\right),\\
    a_{2,0}&=\frac{\pi^4}{\alpha^4}\left\{6-\frac{\alpha^2}{2}+3\alpha\coth\left(\frac{\alpha}{2}\right)+\alpha^2\coth^2\left(\frac{\alpha}{2}\right)\right\},a_{2,1}=-\frac{2\pi^4}{\alpha^4}\left(3+\alpha\coth\left(\frac{\alpha}{2}\right)\right),\\
    a_{2,2}&=\frac{2\pi^4}{\alpha^4},C_{5/2}  =a_{2,0}\cosh\left(\frac{\alpha}{2}\right)+a_{2,1}\frac{\alpha}{2}\sinh\left(\frac{\alpha}{2}\right)+a_{2,2}\frac{\alpha^2}{4}\cosh\left(\frac{\alpha}{2}\right).
\end{align*}
% \iffalse
% \begin{enumerate}
%     \item For $s=0$, $K_{1/2}(x,x')=\sigma^2\left(\cosh\left(\frac{\alpha}{2}\right)\right)^{-1}\cosh\left(u\right),$
%     \item For $s=1$, $$K_{3/2}(x,x')=\sigma^2\;C_{3/2,\alpha}^{-1}\left(a_{1,0}\cosh(u)+a_{1,1}u\sinh(u)\right),$$
%     where $a_{1,0}=\frac{2\pi^2}{\alpha^2}\left(1+\frac{\alpha}{2}\coth\left(\frac{\alpha}{2}\right)\right)$, $a_{1,1}= -\frac{2\pi^2}{\alpha^2}$
%     and $C_{3/2,\alpha}=
%     a_{1,0}\cosh\left(\frac{\alpha}{2}\right)+a_{1,1}\frac{\alpha}{2}\sinh\left(\frac{\alpha}{2}\right)$. 
%     \item For $s=2$, $K_{5/2}(x,x')=\sigma^2\;C_{5/2,\alpha}^{-1}\left(a_{2,0}\cosh(u)+a_{2,1}u\sinh(u)+a_{2,2}u^2\cosh(u)\right)$, where $a_{2,0}=\frac{\pi^4}{\alpha^4}\left\{\left(6-\frac{\alpha^2}{2}\right)+3\alpha\coth\left(\frac{\alpha}{2}\right)+\alpha^2\coth^2\left(\frac{\alpha}{2}\right)\right\}$, $a_{2,1}=-\frac{2\pi^4}{\alpha^4}\left(3+\alpha\coth\left(\frac{\alpha}{2}\right)\right)$, $a_{2,2}=\frac{2\pi^4}{\alpha^4}$ and $C_{5/2}  =a_{2,0}\cosh\left(\frac{\alpha}{2}\right)+a_{2,1}\frac{\alpha}{2}\sinh\left(\frac{\alpha}{2}\right)+a_{2,2}\frac{\alpha^2}{4}\cosh\left(\frac{\alpha}{2}\right)$. 
% \end{enumerate}
% \fi
The process $Z$ is MSC and 1-MSD if $\nu>\frac{p-1}{2}=\frac{1}{2}$ and $\nu>\frac{p+1}{2}=\frac{3}{2}$ % and $\nu>\frac{p+3}{2}=\frac{5}{2}$, respectively. % This implies that the process is MSC, 1 and 2-MSD if $s=0,1$ and 2 respectively. 
Considering smoothness of process realizations arising from the above kernels, the following proposition provides conditions on $s$ for the validity of % 1-MSD process realizations on
$D_VZ(x)$. %, $x\in\mathbb{S}^1$.
\begin{proposition}\label{prop:S1_KV} 
For $s=0$ we have MSC in process realizations. For, $s=1$ and $2$ we have 1-MSD in process realizations. The respective processes possess valid covariance functions. Their explicit expressions are detailed in \Cref{apdx:subsec:sphere}.
\end{proposition}

{ 
Finite element analysis connects our developments in differential geometry to modern applications. We provide the necessary background details in the next Section.
}

\subsection{Surfaces in \texorpdfstring{$\RR^3$}{R3}} % $\mathcal{M}$ is a 2-d surface embedded in $\RR^3$, $M$ is a discrete approximation of $\mathcal{M}$, we calculate.........(S1+ sth from S2-S6).
Every compact, \emph{orientable}, smooth 2-dimensional manifold, referred to simply as a surface, can be smoothly embedded in $\RR^3$. In practice, they are represented as polyhedral meshes \citep[see e.g.,][]{turk1994zippered}. Such discrete representations are obtained using the finite element method (FEM). Galerkin's FEM \citep[see e.g.,][]{ciarlet2002finite,brenner2008mathematical} is one such method that is used to generate polyhedral mesh representations of such manifolds. We work with triangular mesh representations and refer to it simply as a mesh. % We will denote a mesh by $M$. It discretely approximates $\MM$. %, which denotes the compact Riemannian manifold without a boundary in the manuscript.
The mesh $\MM$ is represented by the vertices $v_k \in \RR^3$, $k=1,\dots,K$. % forming triangles in the mesh. 
% For example, \Cref{fig: bunny_res} shows the various resolutions of the SB triangulations available from the Stanford 3D scanning repository. 
{  In the experiments that follow, we use the highest resolution % (see the left panel in \Cref{fig: bunny_res}) 
available composed of 34,834 vertices% shown in the left panel of the top row in \Cref{fig: field_bunny_sphere}
. The higher resolution is well-suited to demonstrate the scalability of our modeling framework discussed in \Cref{sec:bayes} which uses low-rank projection for computational efficiency. % results in faster overall computation % proof-of-concept 
% and less cluttered visualizations while serving as a proof-of-concept. Our methods remain valid irrespective of resolution.%The methods and computation work for higher resolutions. 
An approximation of the Riemannian volume, ${\rm vol}_g(\MM)$ is the total surface area of the mesh. \Cref{fig:bunny_boundary} outlines the boundaries in red for the SB mesh, which are {\em holes} in the mesh.}

\begin{figure}[!h]
    \centering
    \includegraphics[width=0.7\linewidth]{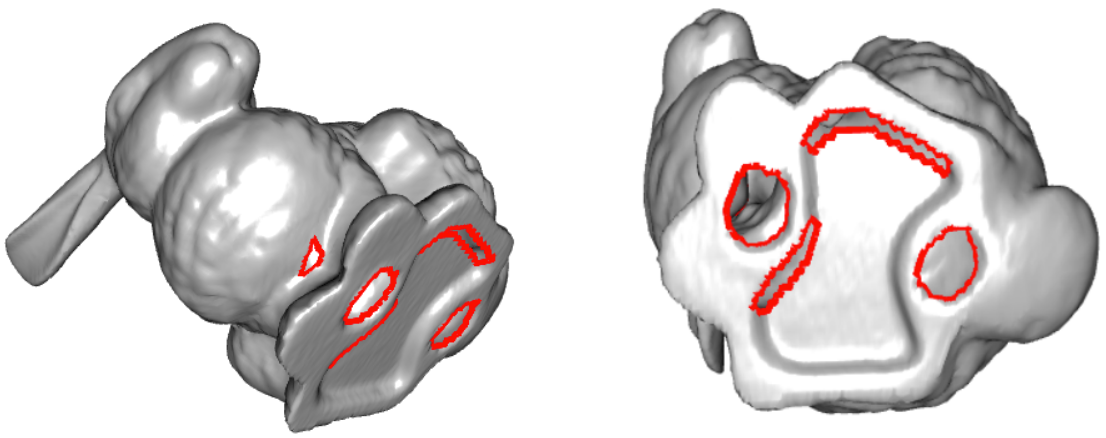}
    \caption{The (smooth) boundary generated by the five holes that sit within the SB mesh as available for download through the \href{https://graphics.stanford.edu/data/3Dscanrep/}{Stanford 3D Scanning repository}.}
    % \caption{Various triangulation resolutions (left to right: lowest to highest) for the SB that are available from the \href{https://graphics.stanford.edu/data/3Dscanrep/}{Stanford 3D Scanning repository}. The number of vertices ranges from (left to right) 453 (908 triangles), 1,887 (3,768 triangles), 8,146 (16,214 triangles) and 34,834 (69,451 triangles).}
    \label{fig:bunny_boundary}
\end{figure}

The geometry is governed by piecewise-linear basis (hat) functions $\phi_k(v_{k'})=\delta_{kk'}$, where $\delta$ is the Kronecker's delta. In this setup, the \emph{cotangent Laplacian} $\Delta$ \citep[see e.g.,][Section 3.3]{meyer2003discrete}, is the discrete approximation of the LB operator $\Delta_g$. In practice, we truncate the spectrum (see \Cref{sec:trunc}). The approximate truncated Neumann spectrum, $\{\lambda_l,f_l\}_{l=0}^{T}$, where $0=\lambda_0 \leq \lambda_1\ldots\leq\lambda_T$, emerges from the eigen-problem 
$\<\Delta f_l,\phi_k\>=\lambda_l\<f_l,\phi_k\>$, for $k=1,\dots,n_v$, where 
\begin{equation}\label{eq: cot-laplacian}
(\Delta f_l)_k=\frac{1}{2}\sum_{j\sim k}(\cot  \alpha_j +\cot \beta_j)\,(f_l(v_j)-f_l(v_k)),
\end{equation}
is the \emph{cotangent Laplacian} of $f_l$ at $v_k$, where $j\sim k$ denotes the indices of the neighbors of vertex $v_k$; $\alpha_j$ and $\beta_j$ are angles opposite to the edge connecting $(v_j, v_k)$; and $f_l(v_j)$ and $\phi_k$ are $K \times 1$ vectors. \Cref{fig: bunny_eigen} shows $f_0$ ({  constant due to Neumann condition}), $f_4, f_9$, $f_{49}$ and $f_{99}$ % which characterize geometry of 
for the SB. The eigenfunctions, $f_l$ need to be computed once prior to any Bayesian computation or model fitting is done and hence, offer no interference with the overall computational complexity of the Bayesian inference discussed in \Cref{sec:bayes}. These constitute fairly standard numerical machinery. {  Note that in the context of a mesh we have, $T\leq n_v$ hence, for scattered observations $N>n_v$, he joint covariance matrix is always positive semi-definite, which follows from \Cref{thm:truncation_pd}. The bound is useful if, $N\leq T\leq n_v$.}

\begin{figure}[b]
    \centering
    \includegraphics[width=\linewidth]{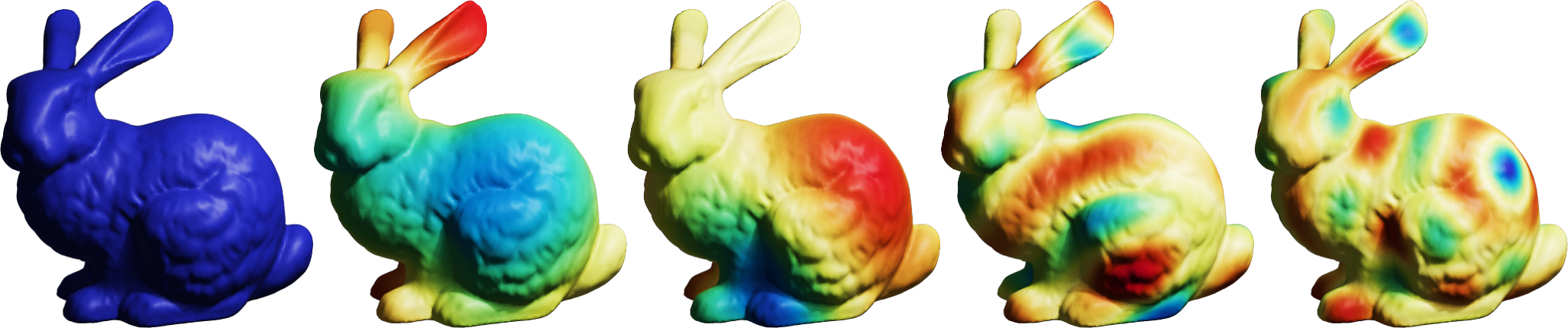}
    \caption{eigenfunctions of the cotangent Laplacian, (left to right) $f_0, f_4, f_9, f_{49}$ and $f_{99}$, for the SB.}
    \label{fig: bunny_eigen}
\end{figure}

\begin{figure}[t]
    \centering
    \includegraphics[width=0.9\linewidth]{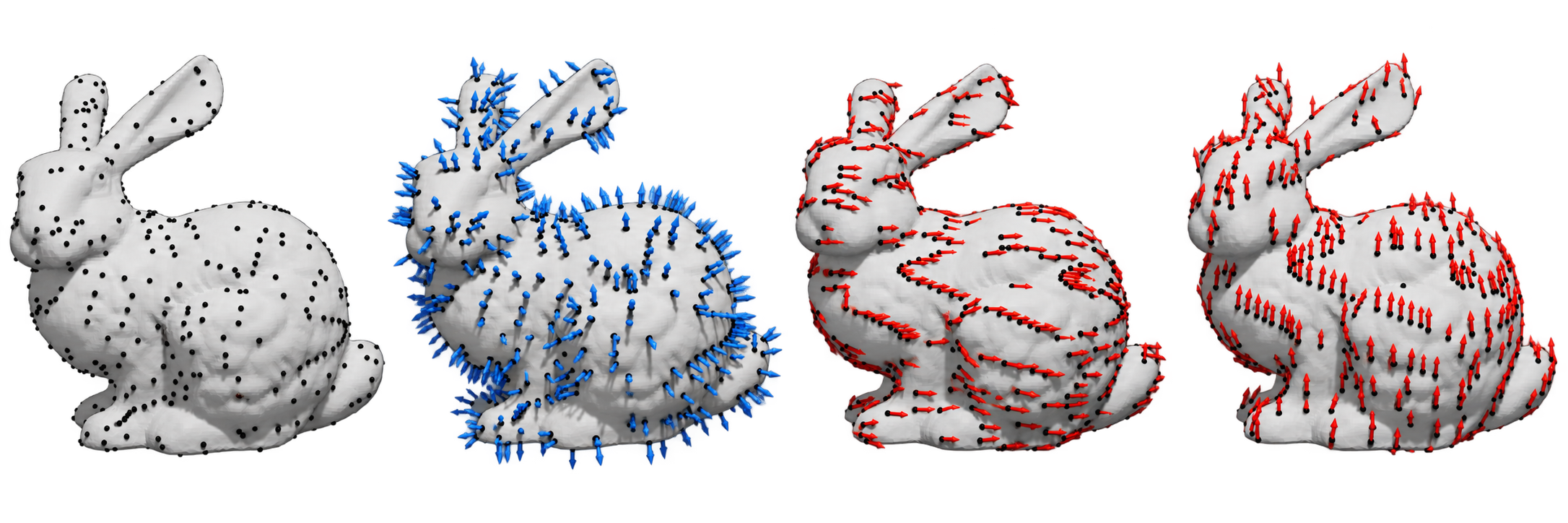}
    
    \vspace*{1.5em}
    
    \includegraphics[width=0.8\linewidth]{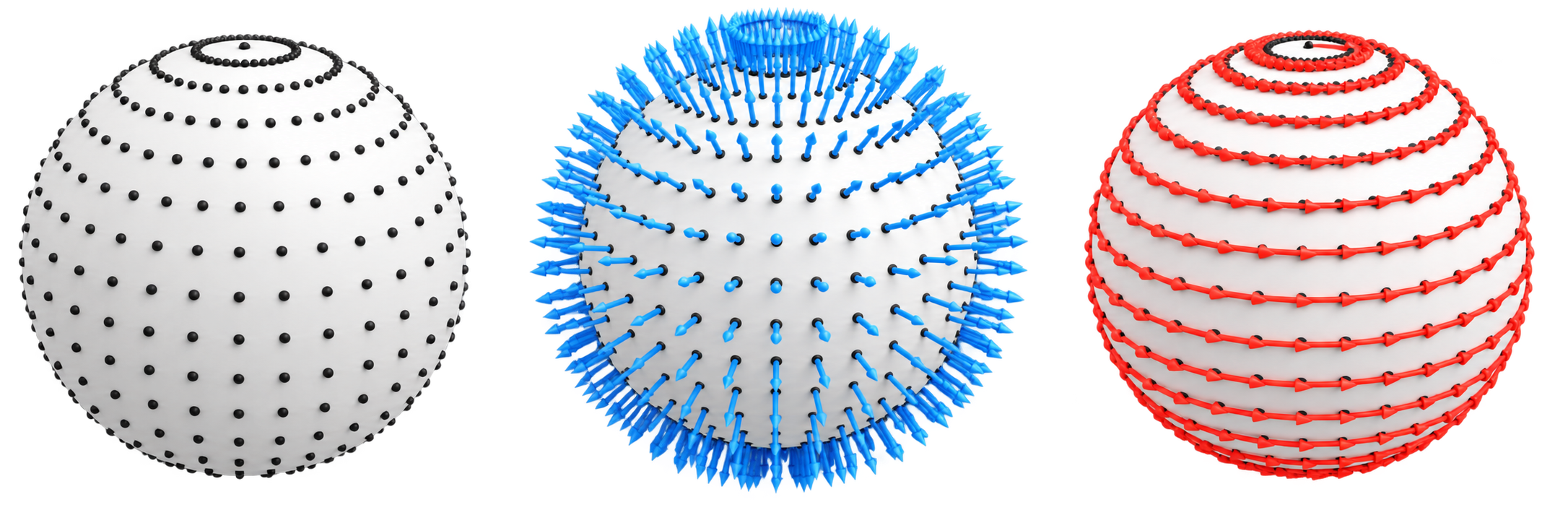}
    \caption{Plot showing (top row: left to right) the vertices, the normals, the vector fields generated using the $x$-axis and $y$-axis as reference vectors on the Stanford Bunny; (bottom row: left to right) the grid, the normals, $x=(x^1, x^2, x^3)$, and the rotational vector field, $(-x^2, x^1, 0)$ on $\mathbb{S}^2$. Vector fields at poles are consistent with the Hairy Ball Theorem.}
    \label{fig: field_bunny_sphere}
\end{figure}

 We construct a globally oriented smooth vector field by choosing a direction in the ambient space and projecting it onto the tangent plane. At each vertex $v_k$ of $\MM$, let $n_k\in\RR^3$ denote the unit normal (shown in the second plot from left in the top and bottom rows of \Cref{fig: field_bunny_sphere}). We fix $e_1=(1,0,0)\T$ and let $\widetilde{V}^k = (I_3-n_kn_k\T)e_1=e_1-(e_1\T n_k)\,n_k$. % and $V^k=\widetilde{V}^k||\widetilde{V}^k||^{-1}$. 
 Thus $V^k\in T_{v_k}\MM$ is the orthogonal projection of the ambient $x$-direction on the local tangent plane. If $\widetilde{V}^k\neq0$, a corresponding unit tangent direction may be defined by $V^k=\widetilde{V}^k/\|\widetilde{V}^k\|$. The normalized field is undefined at points where the ambient direction is normal to the surface and may vary rapidly near such points. We therefore use $\widetilde{V}$ for directional derivative inference. The resulting process $D_{\widetilde{V}}Z(x)$
quantifies the component of intrinsic process variation associated with the projected ambient $x$-direction. % is a smooth unit vector field at the vertex $v_k$. The vector field, $V_{x_0}$, for an arbitrary location $x_0\in \MM$, is obtained using barycentric interpolation as discussed in \Cref{sec: barycentric_sampling_interpolation}. The vector $e_1$ can be replaced with any vector of choice in the ambient space. Using $e_1$ lets $D_VZ$ behave as (directional) derivatives along the local $x$-axis for the triangles, thus enabling ease of interpretation. The resulting vector field over the SB is shown in \Cref{fig: field_bunny_sphere}, the top row, second plot from the right. The rightmost panel shows the same using $e_2=(0,1,0)\T$ resembling the local $y$-axis. % Using the vertex normals as vector fields, $V^k=n_k$ is not advisable. 
 Much like the computation for eigenfunctions, the vector field must be constructed once and does not affect the computational complexity of posterior inference for $D_VZ$. We use $e_1$ as the reference vector in our experiments within \Cref{sec: bunny}. The plot on the right in the bottom row of \Cref{fig: field_bunny_sphere} shows the rotational vector field, $(-x^2, x^1,0)$, which lies in the tangent space, $T_x\mathbb{S}^2$ and features in our experiments within \Cref{sec: S2}.

\section{Simulation Experiments}\label{sec:experiments}
We show the results of simulation experiments conducted for the proposed inferential framework for differential processes on $\mathbb{S}^2$ and the SB. For both scenarios, we use smooth trigonometric patterns as the truth. This demonstrates the ability of our inferential approach to learn arbitrary spatial patterns over complex domains that are not strictly from the model in \Cref{eq:hierarchical_model,eqn:posterior} and also provides a ground truth for comparing our posterior inference. We note that the kernels in \Cref{tab:S2_Kernels} and \Cref{eqn:Matern,eqn:RBF} involve infinite sums and require truncation in practice. This affects the smoothness of process realizations which is discussed in \Cref{sec:trunc}. For both scenarios, scattered data are simulated with $\tau^2=1$ in the manifold of choice, and the hierarchical model in \Cref{eq:hierarchical_model} is fitted to the simulated data. We only pursue inference on $D_VZ(x_0)\mid \Z$ following \Cref{eq:gp-grad,eq:posterior_predicitive}. The location $x_0$ should ideally lie on a grid \citep[see e.g.,][Section 5]{halder2024bayesian,halder_bayesian_2026} for optimal inference on $D_VZ$. Using a grid is appropriate for manifolds with atlases, for example, $\mathbb{S}^2$. For a general $\MM$, like SB, using point clouds are more appropriate. We place emphasis on the careful choice of an appropriate vector field $V$, as it remains crucial for the physical interpretation of $D_VZ$. In the ensuing experiments, the fractal parameter for the kernels controlling the smoothness of the process is fixed by design at $\nu=2$, hence guaranteeing valid inference on derivatives. Posterior samples for $D_{V}Z$ are obtained one for one using those for $\sigma^2$ and $\alpha$ generated from the model fit.

\subsection{Sphere \texorpdfstring{$\mathbb{S}^2$}{S2}}\label{sec: S2}
\begin{figure}[t]
    \centering
    \includegraphics[width=0.9\linewidth]{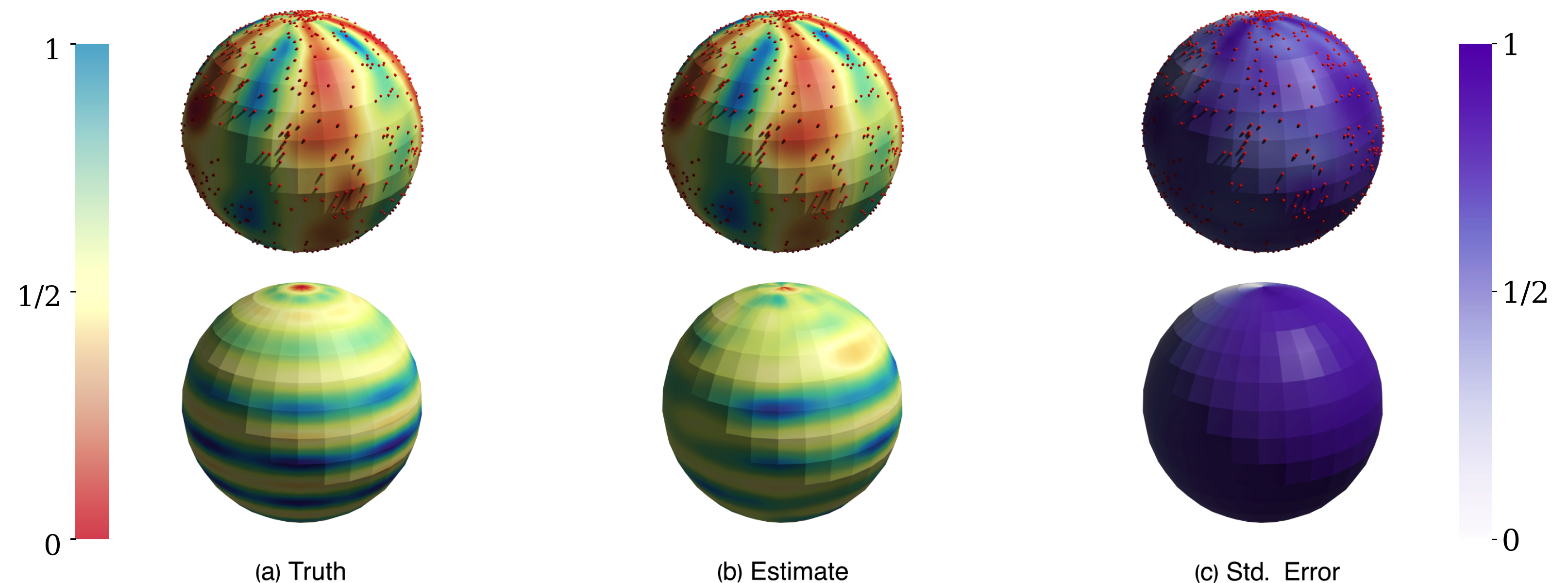}
    \caption{Plots showing interpolated surfaces for the sphere, $\mathbb{S}^2$: the true process (top a) estimated process (top b) and standard error (top c); true derivatives (bottom a)  estimated derivatives (bottom b) and standard error (bottom c).} % In the first two columns red shades indicate lower and blue shades indicate higher values. For the standard error plots in the last column deeper (lighter) shades indicate higher (lower) values.}
    \label{fig:s2-process-grad-sim}
\end{figure}
We use $Y(x(\theta, \phi))\sim \mathcal{N}[\mu(x(\theta,\phi))=2(\sin(3\pi\theta) + \cos(3\pi\phi)), \tau^2]$ as the patterned truth, where $x(\theta,\phi)=(\sin(\phi)\cos(\theta),\sin(\phi)\sin(\theta),\cos(\phi))\in \mathbb{S}^2\subset\RR^3$ and $(\theta, \phi)$ are polar and azimuthal angles, respectively, that serve as local charts. Smooth interpolated surfaces are obtained at the local chart level using multilevel B-splines \citep[see e.g.,][]{finley2024mba}. \Cref{fig:s2-process-grad-sim} (top row, plot (a)) shows the true surface for $N=10^3$ randomly simulated locations, which are marked with red dots. We use the truncated Legendre-Mat\'ern covariance kernel % (see \Cref{tab:S2_Kernels}) 
truncated at $20$ eigenpairs (terms with Legendre polynomials $P_1,\dots, P_{20}$, %) % and fix the smoothness as a part of the design by setting $\nu=2$, thus guaranteeing valid inference on gradients (
see \Cref{tab:S2_Kernels}). The estimated parameters accompanied by their 95\% credible intervals (CIs) are $\widehat{\tau}^2=1.60\;(1.44, 1.79)$, $\widehat{\sigma}^2=4.69\;(3.41,7.61)$ and $\widehat{\alpha}=13.41\;(8.93, 19.64)$. \Cref{fig:s2-process-grad-sim} (top row, plots (b) and (c)) show plots for the fitted GP and standard errors, respectively.

% We use V_x for vector fields

% Our choice of patterned truth allows us to assess the quality of posterior inference on gradients. 
We use the smooth rotational vector field, $V_x = (-x^2,x^1,0)$ on $\mathbb{S}^2$. We have a closed-form expression for the ground truth. The geodesic starting from $(\theta,\phi)$ along $Z$ is $\gamma(t)=(\theta,\phi+t)$, where $\gamma(0)=(\theta,\phi)$ and $\gamma'(0) = \partial_\phi$. 
% =Z(x) Let $\mu(x()=2(\sin(3\pi\theta)+\cos(3\pi\phi))$, so 
The derivative of the mean function is $D_V\mu(x(\theta, \phi)) = -6\pi \sin(3\pi \phi)$. Posterior inference on derivatives is sought over an equally spaced grid (constructed using local charts, i.e. spherical coordinates) spanning the sphere's surface. Further details regarding the derivation of the posterior using \Cref{eq:grad-posterior,eq:ZD_VZ_Sp,eqn:D_VZ_Sp} can be found in \Cref{sec:grad-posterior}. % Posterior samples for gradients are obtained one-for-one corresponding to those for $\sigma^2$ and  $\alpha$ from the fitted model.
Interpolated surface plots comparing the truth against estimated derivatives are shown in the bottom row of \Cref{fig:s2-process-grad-sim}. Overall, we observed satisfactory quality of posterior inference for $D_VZ$, achieving 85.44\% coverage of the truth. % in their corresponding 95\% CIs. % over a $25\times 25$ grid.

% SB & DL: I am stuck at the gradient estimation part--mostly due to coding errors (scaling issues) on my part. If you uncomment the lines below you will see that the process estimation (a Bayesian version of Borovitsky et al., 2020) works really well. I used around 200 eigen pairs for the discrete LB. 
% Update (12/22/25):: The problem was not so much with scaling and more with the choice of the vector field, V_{x_0}. I was using the normal to each trangular plane and getting worried when I saw flat gradient plots, which was a stupid idea in hindsight, I have since shifted to using a unit tangent vector field. Things look much better with the Bunny, basically as good as it gets!
\subsection{Stanford Bunny (SB)}\label{sec: bunny}
\begin{figure}[t]
    \centering
    \includegraphics[width=0.9\linewidth]{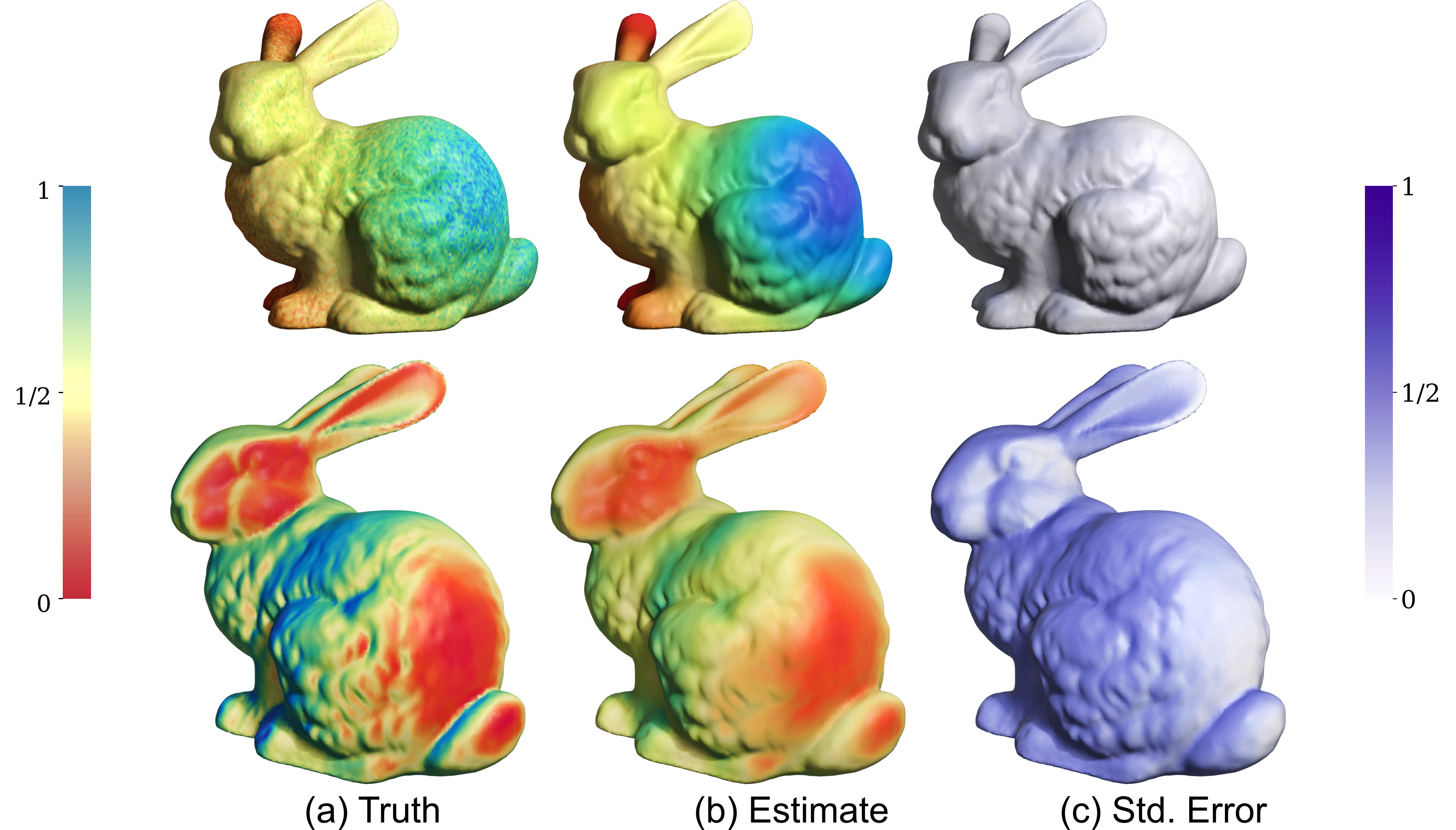}
    \caption{Plots showing interpolated surfaces for the SB: the true process (top a), estimated process (top b) and standard error (top c); true derivatives (bottom a); estimated derivatives (bottom b) and standard error (bottom c).} % In the first two columns red shades indicate lower and blue shades indicate higher values. For the standard error plots in the last column deeper (lighter) shades indicate higher (lower) values.}
    \label{fig:bunny-grad-sim}
\end{figure}

\begin{table}[b]
\centering
\caption{Comparison of quality of model fit, inference on gradients and compute time ($10^4$ MCMC iterations) for varying truncation levels.}\label{tab:trunc-sim}
    \label{tab:truncation}
\resizebox{\linewidth}{!}{
    \begin{tabular}{c|ccccc}
    \hline\hline
        $T$ (\%ge) &  Process RMSE & Process Coverage (\%ge) & Gradient RMSE & Gradient coverage (\%ge) & Time (in mins.) \\
        \hline
         300 (0.86\%) & 0.10 & 88.65 & 15.54 & 75.21 & 0.16\\
         500 (1.44\%) & 0.09 & 91.47 & 15.06 & 82.69 & 0.26\\
         1000 (2.37\%) & 0.09 & 92.25 & 14.97 & 86.38 & 1.01\\
         \hline\hline
    \end{tabular}
}
\end{table}

% In its coarsest resolution, the SB is represented as a triangular mesh composed of 453 vertices and 908 triangles. 
% For the SB, irregularly spaced sample locations in $\MM_{\rm SB}$ are generated using barycentric sampling---for $x = (x^1, x^2, x^3)\in\MM_{\rm SB}$, where $\MM_{\rm SB}$ denotes the SB manifold. 
We simulate, $Y(x)\sim \mathcal{N}[\mu(x), \tau^2]$, $\mu(x)=10(\sin(3\pi x^1) + \sin(3\pi x^2) +\sin(3\pi x^3))$ {  at the vertices of the mesh}. % Interpolated surface registration on $\MM_{\rm SB}$ is performed using surface splines.
\Cref{fig:bunny-grad-sim} (top row (a)) shows the resulting plot generated from the realizations. The cotangent Laplacian in \Cref{eq: cot-laplacian} serves as a discrete approximation for the LB operator required for the spectral Whittle-Mat\'ern kernel in \Cref{eqn:Matern}. We use 300, 500 and 1000 eigenpairs (out of 34,843, i.e. $<5\%$ of the spectrum) of the cotangent Laplacian. % to discretely represent the SB which are defined on the mesh vertices. Barycentric interpolation of the eigenfunctions approximates them at the observed locations. % and a fixed fractal parameter, $\nu=2$ that enables valid inference for $D_VZ(x)$. 
The posterior parameter estimates with 95\% CIs are $\widehat{\tau}^2=1.00\;(0.99, 1.02)$, $\widehat{\sigma}^2=21.77\;(11.17,46.13)$ and $\widehat{\alpha}=26.99\;(22.62, 32.22)$. The 95\% CI for $\tau^2$ contains the truth. \Cref{fig:bunny-grad-sim} (top row, plots (b) and (c)) shows the resulting model fit and standard error. %Further ancillary details % regarding surface registration using spline interpolation, barycentric sampling and interpolation 
%can be found in the Supplement.

%The absence of linearity in $\MM_{SB}$ poses a hurdle for generating an equally spaced grid over the SB. % over which $D_VZ(x)$ is to be estimated. 
% Working with a triangulated mesh approximation of the SB, 
\begin{figure}[!h]
    \centering
    \includegraphics[width=0.95\linewidth]{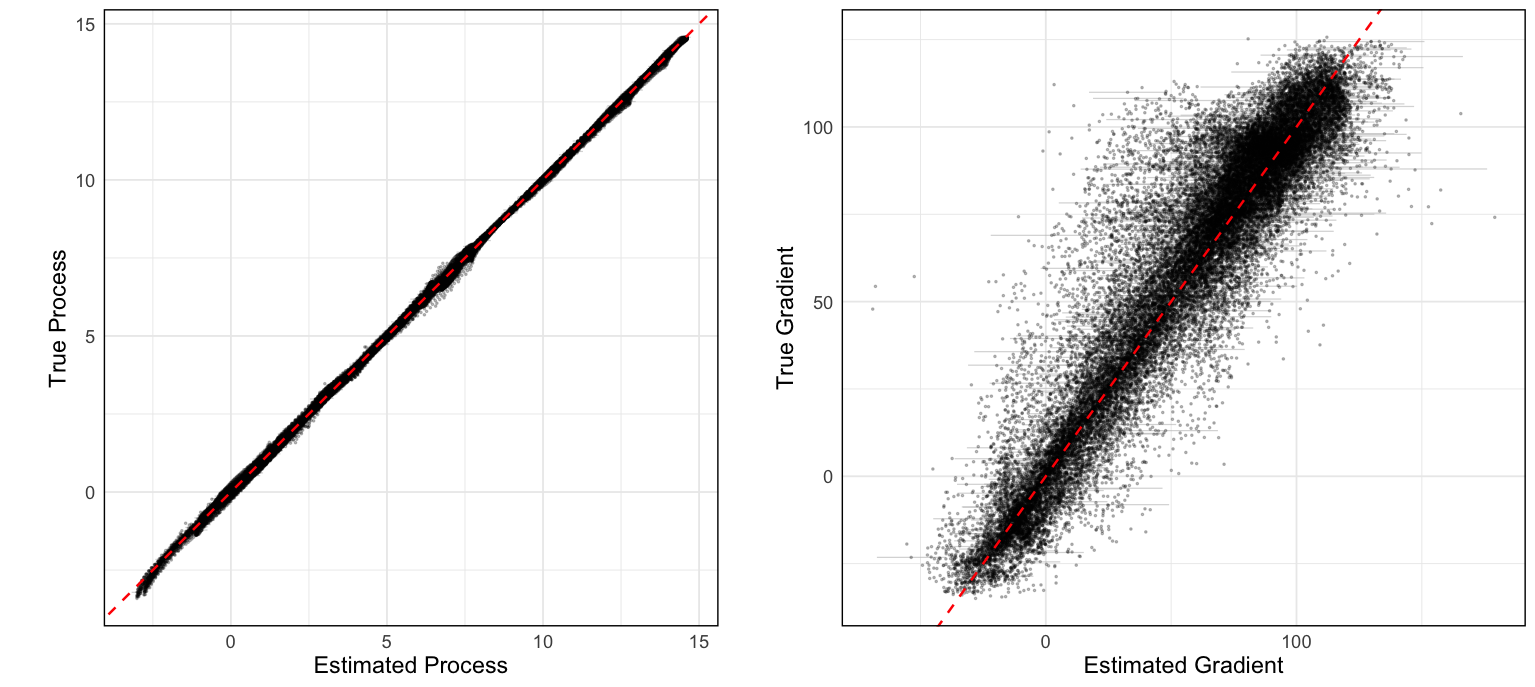}
    \caption{Scatter plots comparing posterior inference against truth for the process (left) and derivatives (right).} % In the first two columns red shades indicate lower and blue shades indicate higher values. For the standard error plots in the last column deeper (lighter) shades indicate higher (lower) values.}
    \label{fig: bunny-true-v-est}
\end{figure}

We use {  the mesh vertices as our point-cloud. Alternatively, one could use \emph{farthest point sampling} and \emph{barycentric interpolation} (see \Cref{sec: farthest_point_sampling,sec: barycentric_sampling_interpolation}) to generate a point cloud % which is shown using red dots in the bottom row of \Cref{fig:bunny-grad-sim} 
and interpolate the response and gradient surfaces at arbitrary locations}. A smooth unit vector field, $V_{v_0} \in T_{v_0}\MM_{\rm SB}$, is obtained by projecting per face mesh normals onto the tangent space using a reference direction (e.g., the $x$-axis) % $e_1=(1,0,0)\T$
at each vertex, $v_0$. % Subsequently, they are interpolated at grid locations $x_0$. 
For each triangle, we use $\gamma(t)=\exp_x(tV_x)\approx x_0+tV_{x_0}$ as the linear approximation to a geodesic starting from $x_0$ along $Z$, $\gamma(0)=x_0$ and $\gamma'(0) = V_{x_0}$. Hence, the true derivatives are $D_V\mu(x_0)=\<\partial_{x_0}\mu(x_0), V_{x_0}\>$, where $\partial_{x}$ is the coordinate-wise differential operator. The plots in the bottom row of \Cref{fig:bunny-grad-sim} compare $D_VZ(x)$ to the truth. {  \Cref{tab:truncation} compares the effect of including more high-frequency eigenvalues on the quality of posterior inference on gradients. Using $\approx2\%$ of the spectrum, we observed that 86\% of the vertices had gradient estimates whose true values were contained in the credible intervals.} \Cref{fig: bunny-true-v-est} shows the overall quality of statistical learning for the process and $D_VZ$. As expected, posterior inference for $D_VZ$ exhibits greater variability than inference for $Z$, reflecting both the increased roughness of derivative sample paths and the effects of discretization inherent to polyhedral meshes. Nevertheless, the Bayes estimate for derivatives closely follows the truth, % , with no evidence of systematic bias or sign reversals. 
the process is recovered nearly exactly, and the recovery of intrinsic (directional) derivatives is consistent with uncertainty that reflects geometry and stochastic regularity. % Further details supporting the computational methods used in this section are available in the Supplement \citep[][]{Li2026supplement}. %whose contents are briefly described below.

\begin{figure}[!h]
    \centering
    \includegraphics[width=0.8\linewidth]{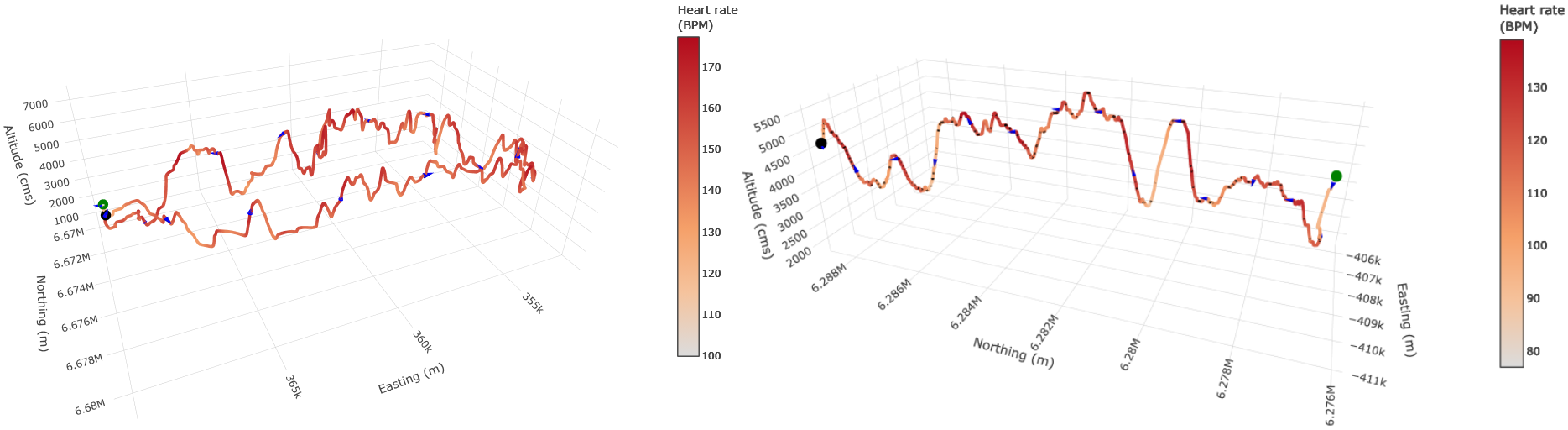}
    \caption{Workout trajectory showing overlaid heart-rate of athletes (left) male and (right) female athlete biking. The blue cones indicate direction of movement, the start is shown with a green dot and the finish with a black dot.} 
    \label{fig: trajectory_athlete}
\end{figure}
\section{Application}\label{sec:app}
{  We consider two real-world scenarios that illustrate the scientific interpretation afforded by the proposed differential processes. First, we work with movement trajectories of athletes recorded during a sport. We use the \emph{FitRec} dataset \citep[see, e.g.,][]{ni_modeling_2019}. The trajectories are obtained from wearable devices, for example, smartwatches. Intrinsically, the trajectories are 1-dimensional manifolds. We are interested in modeling the heart-rate (HR) of an athlete and subsequently obtaining gradients of the latent heart-rate as a measure of cardiovascular response to physical activity. Second, we consider modeling the cortical thickness for individuals with autism spectrum disorder (ASD) diagnosis. We use the ABIDE-I study. Literature has shown that ASD-related cortical-thickness effects are spatially heterogeneous, age-dependent, and regionally distributed, rather than a simple global shift. Some studies report regional thickening, others thinning, and large analyses show strong developmental interactions \citep[see, e.g.][]{khundrakpam_cortical_2017,bedford_large-scale_2020,ecker_interindividual_2022}. A recent large harmonized analysis found regionally thinner cortex in ASD together with curvature alterations, while earlier ABIDE work found regional thickening \citep[see, e.g.][]{mahmoudi_cortical_2026}, illustrating the heterogeneity that we demonstrate.}

\subsection{FitRec Dataset}\label{sec:fitrec}
{  Our approach differs from previous Bayesian approaches to modeling measurements over trajectories, for example, \cite{alaimo2023bayesian,wakayama_process-based_2024} consider modeling magnitude of acceleration while investigating its association with energy expenditure. We use two workouts. The trajectories we work with are intrinsically intervals. They are compact with a discrete boundary (at the starting and ending points) and highly irregular, potentially with self-intersections, for example, an athlete repeating a specific trajectory segment (loops). To avoid resulting complications, we choose an arc-length parameterization. We define, $\gamma:M\to \RR^3$, whose image $\MM=\gamma(M)$ provides the geographic realization of our 1-D manifold with boundary $\partial\MM=\{0,L\}$. Due to the possibility of self-intersections, we note that $\gamma$ need not be an embedding rather, $\gamma$ is an immersion with $\gamma'(s)\ne 0$. For time $t \in [0, T]$, we have $\widetilde{\gamma}(t) = \left(x^1(t), x^2(t), x^3(t)\right)\T$, where $x^1(t)$ is Easting, $x^2(t)$ is Northing and $x^3(t)$ is altitude, all measured in meters. The Riemannian metric, $g = \gamma^*g_{\RR^3}$ is the standard pullback metric, where $g_{\RR^3}(u,v)=\<u,v\>$ is the Euclidean metric and $g=ds^2$, where $s(t)=\displaystyle\int_{0}^{t}||\widetilde{\gamma}(u)||_{\RR^3}du$ is the arc-length. The choice of the vector field is greatly simplified: $V(s)=\frac{\partial}{\partial s}$ is the direction of movement. Under Neumann boundary conditions, the eigen-value problem is, $-\Delta_gf_l(s)=-\frac{d^2}{ds^2}f_l(s)=\lambda_lf_l(s)$, with $f'_l(0)=f'_l(L)=0$. % Our results for smoothness carry over \citep[for more details see, e.g.][]{grieser_uniform_2002,xu_gradient_2009,xu_eigenfunction_2011}. 
The eigenfunctions are $f_0(s)=L^{-\frac{1}{2}}$ and $f_l(s)=\sqrt{2}L^{-\frac{1}{2}}\cos\left(\frac{l\pi s}{L}\right)$, for $l\geq 1$, with eigenvalues, $\lambda_l=\left(\frac{l\pi}{L}\right)^2$ and the dimension of the manifold, $p=1$. The Neumann boundary conditions impose a vanishing normal derivative of the latent HR at the endpoints, in this one-dimensional setting this entails, $D_VZ(0)=D_VZ(L)=0$. These conditions do not necessarily produce a stationary covariance matrix. We select a male and a female athlete engaging in the sport of biking. Their trajectories with overlaid HRs are shown in \Cref{fig: trajectory_athlete}. The male athlete's trajectory was an {\em out-and-back} while the female athlete's route was {\em point-to-point}. The male athlete biked 53.10 km while the female athlete biked 14.34 km. The data was recorded in Helsinki, Finland.

\begin{table}[!h]
\centering
\caption{Posterior inference on model parameters. The point estimates are followed by 95\% credible intervals which are shown in brackets.}\label{tab:trajectory_model}
    \begin{tabular}{c|c|c}
    \hline\hline
    Parameters & Male & Female\\ \hline
    $\beta_0$ & 152.67 (152.39, 152.94)&  115.94 (115.85, 116.03) \\ \hline
    $\sigma^2$ & 89.17 (68.75, 116.70) & 126.25 (86.30, 189.23) \\ \hline
    $\tau^2$ & 9.92 (8.45, 11.64) & 1.07 (0.91, 1.26)\\ \hline
    $\alpha$ & 358.27 (299.18, 425.18) & 124.54 (105.93, 143.60) \\ \hline \hline
\end{tabular}
\end{table}

\begin{figure}[t]
    \centering
    \includegraphics[width=0.8\linewidth]{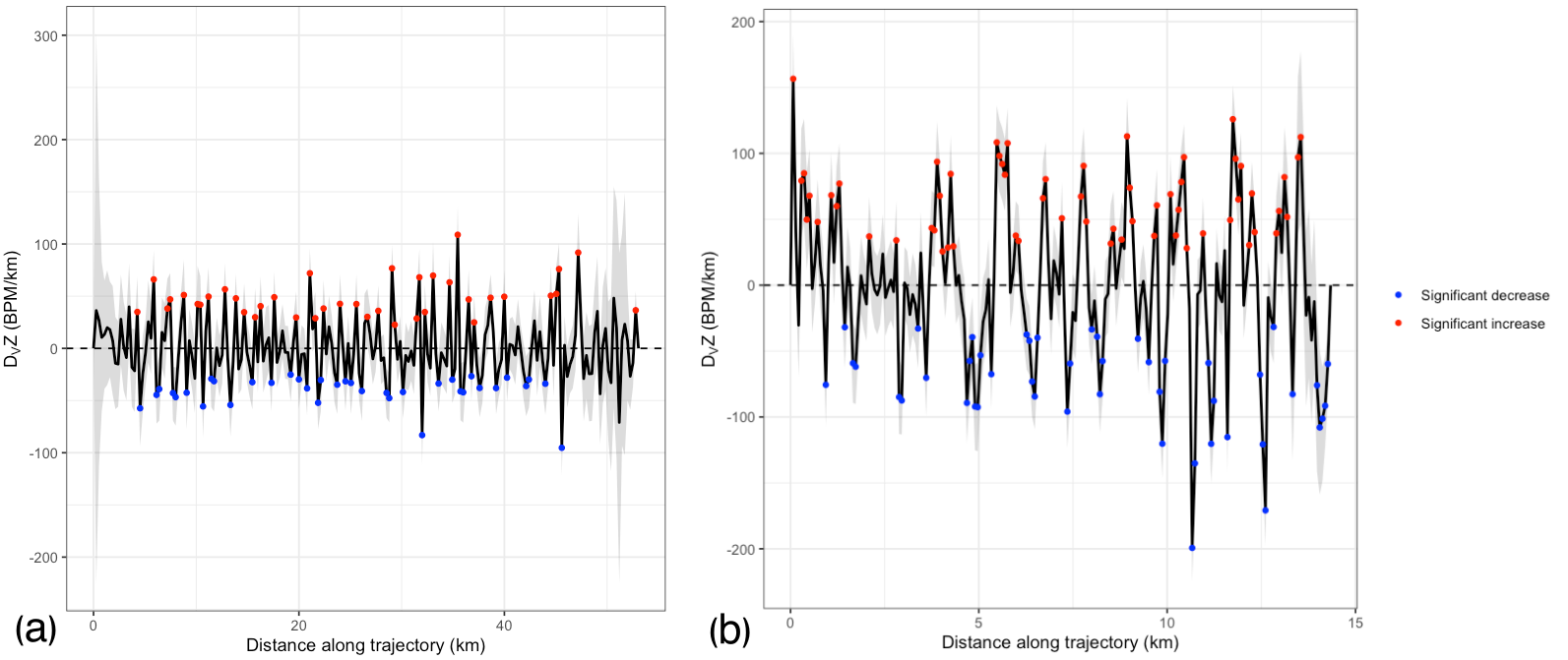}
    \caption{Comparison of variation in HR with distance for the athletes (a) male and (b) female.} 
    \label{fig: hr_v_dist_athlete}
\end{figure}

\begin{figure}[t]
    \centering
    \includegraphics[width=0.95\linewidth]{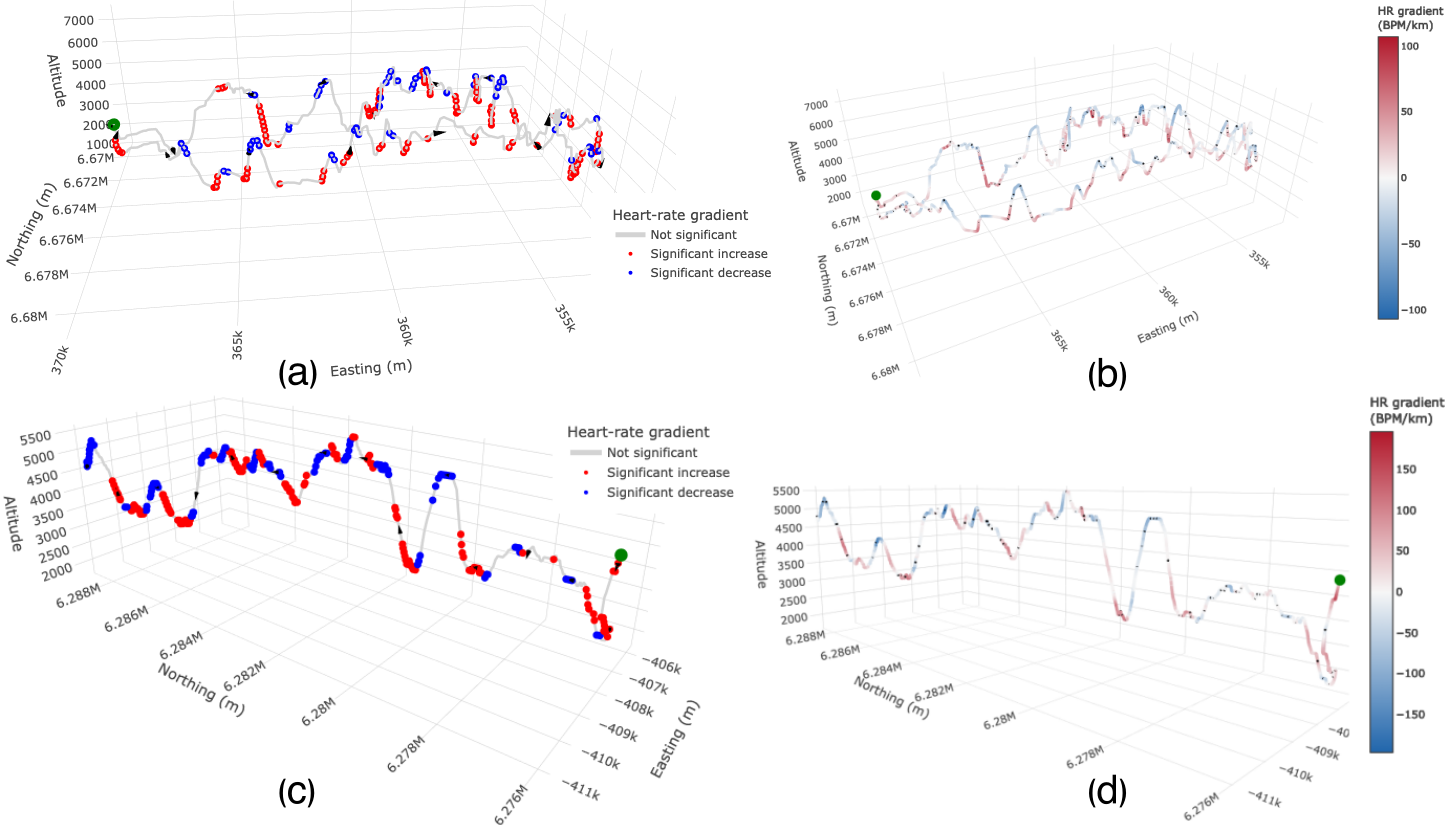}
    \caption{Probabilistic inference on gradients of the latent HR (measured in beats per minute/kilometer (BPM/km)). Plot (a) shows segments where the 95\% credible interval for the gradients for heart-rate of the male athlete excludes 0, shown in red for increase and blue for decrease in HR and, (b) shows the posterior gradient estimates. Plots (c) and (d) show the same for the female athlete.} 
    \label{fig: hr_gradient_athlete}
\end{figure}

We fit the model in \Cref{eq:hierarchical_model} with an intercept, $Y(s) = \beta_0+Z(s)+\epsilon(s)$, where $Y(s)$ is the HR measured in BPM, $Z(s)$ is the latent HR and $D_VZ(s)=\frac{d}{ds}Z(s)$ is its derivative (measured in BPM/km). Posterior inference on model parameters is shown in \Cref{tab:trajectory_model}. The variation profiles for the HR are very different for the two athletes. The male athlete averaged $\approx$152 BPM, while the female athlete averaged $\approx$116 BPM over their respective rides. The overall variance, $\sigma^2$ is 89.17 (68.75, 116.70) and $\frac{\sigma^2}{\sigma^2+\tau^2}=90.00\%$ for the male athlete while, for the female athlete $\sigma^2$ is 126.25 (86.30, 189.23) and $\frac{\sigma^2}{\sigma^2+\tau^2}=99.16\%$. The overall variation in HR for the male athlete showed less spatial variation compared to the female athlete (also evident from \Cref{fig: hr_v_dist_athlete}). The posterior estimate of $\alpha$ was substantially larger for the male athlete (358.27 vs. 124.54), indicating a marked difference in the spectral scale of variation between the two trajectories.

}

{  Probabilistic analysis of gradients of latent HR, $D_VZ(s)$ reveals more insight. The rides begin with a downward sloped warm-up. For the male athlete the HR shows no credible gradients (see, e.g., \Cref{fig: hr_v_dist_athlete} (a), \Cref{fig: hr_gradient_athlete} (a) and (b) or \Cref{fig: segment_analysis_athlete} (a)), while the female athlete has the highest HR recorded over the ride within the first two kms. (see, e.g., \Cref{fig: hr_v_dist_athlete} (b), \Cref{fig: hr_gradient_athlete} (c) and (d) or \Cref{fig: segment_analysis_athlete} (d)) indicating different levels of cardiovascular response to activity. For the first third of the course the male athlete steadily gains elevation with a consistent trend of credibly positive gradients at the start of every climb followed by credible recovery zones at the top of every climb (see, e.g., \Cref{fig: segment_analysis_athlete} (a)) while the female athlete's route has two major climbs showing a similar pattern of gradients with a flat middle section allowing for credible recovery (see \Cref{fig: segment_analysis_athlete} (d)) allowing her to settle in. The middle third section for the male athlete keeps up with the trend of sustained climbing ability (credibly high gradients at the start followed by credibly recovery at the top climbs) in the short climbs while the female athlete's latent HR shows credibly high gradients immediately followed by credibly low gradients (compare \Cref{fig: segment_analysis_athlete} (b) and (e)). In the final third of both workouts, longer segments of credibly positive HR gradients than negative gradients are observed, a pattern consistent with increasing cardiovascular demand and reduced recovery later in the workout, as evidenced from \Cref{fig: hr_v_dist_athlete} (a) and (b) and \Cref{fig: segment_analysis_athlete} (c) and (f).

\begin{figure}[!h]
    \centering
    \includegraphics[width=0.8\linewidth]{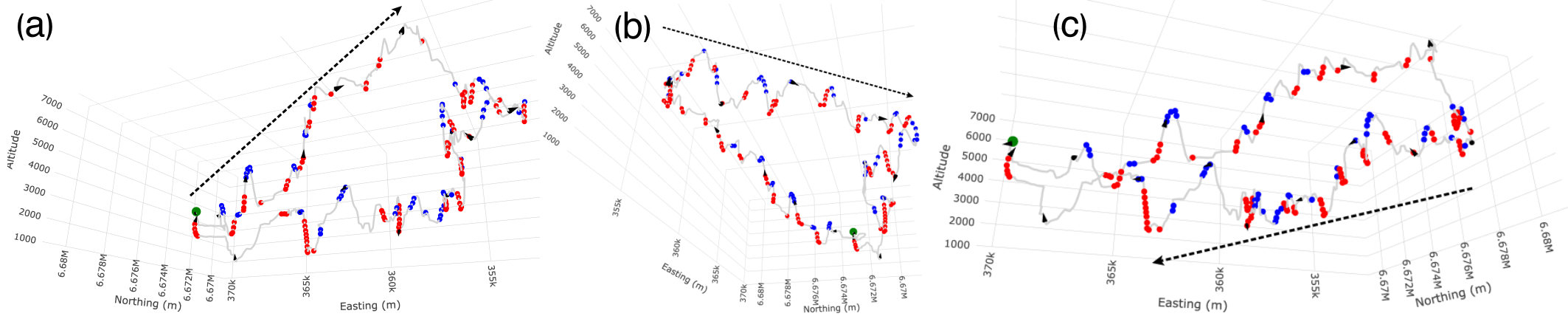}
    \includegraphics[width=0.8\linewidth]{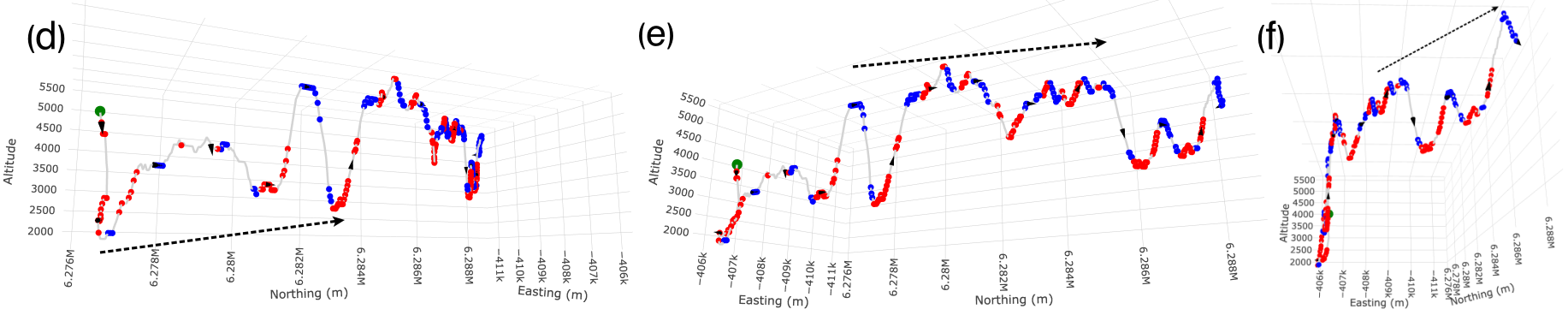}
    \caption{Segment-wise views of credibly significant cardiovascular work (and recovery) for the athletes along respective trajectories. Plot (a) top third (b) middle third (c) bottom third for the male athlete. Plots (d), (e) and (f) show the same for the female athlete.} 
    \label{fig: segment_analysis_athlete}
\end{figure}

\subsection{ABIDE-I}\label{sec:abide}
We compare intrinsic variation in cortical thickness for subjects with ASD against typically developing controls. We use cortical meshes that are pre-processed using the FreeSurfer pipeline \citep[see, e.g.,][]{fischl_measuring_2000} and contain measurements of cortical thickness, sulcal depth and curvature. The cortex has a boundary, known as the medial wall (see \Cref{fig:medial_wall}) that separates the left and the right hemisphere, abbreviated as LH and RH respectively here onwards, from the longitudinal fissure. The LH and RH are disjoint, hence we model cortical thickness for them separately. We consider 8 subjects---4 males and 4 females with 4 pairs of ASD diagnosis and controls, aged between 9--17 years at the time of scan from three different international sites.

\begin{figure}[t]
    \centering
    \includegraphics[width=0.7\linewidth]{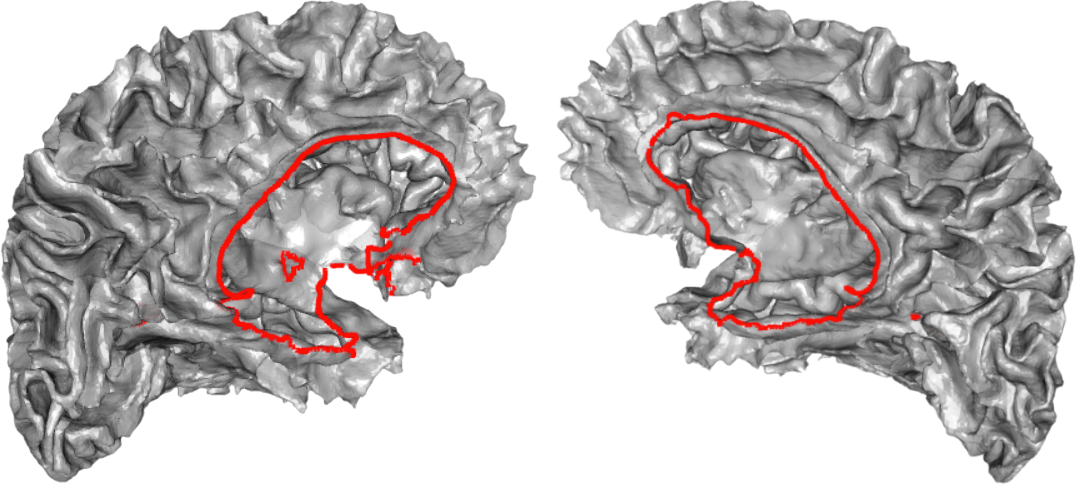}
    \caption{The combined cortex (LH and RH) for subject 50794 showing the medial wall outlined in red.}
    \label{fig:medial_wall}
\end{figure}

\begin{figure}[t]
    \centering
    \includegraphics[width=0.7\linewidth]{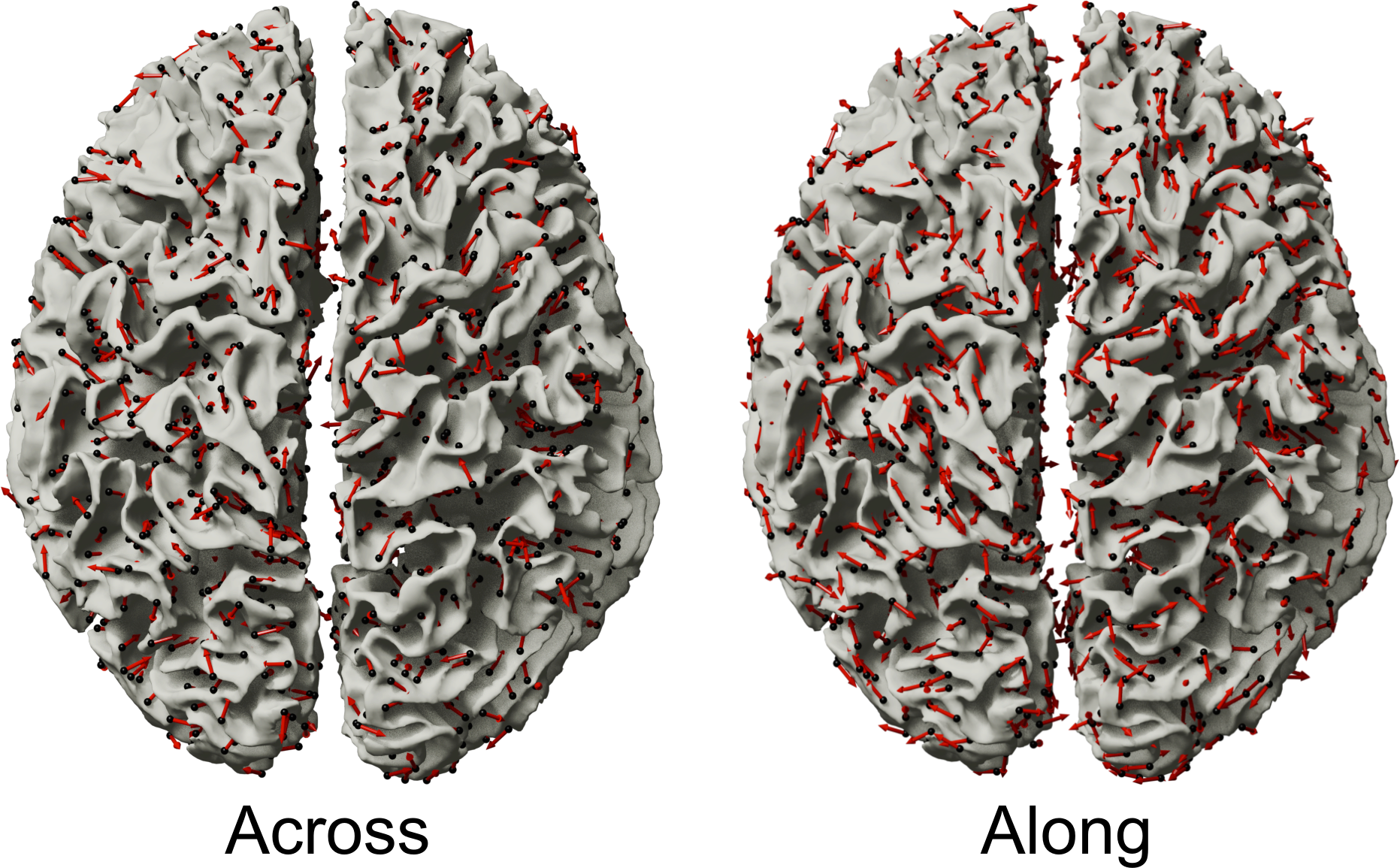}
    \caption{Anatomically informed vector fields constructed from sulcal depth: (left) the across-fold direction $V_\perp=\nabla_\MM s/\|\nabla_\MM s\|_g$, (right) the orthogonal along-fold direction $V_\parallel=n\times V_\perp$. They are marked with red arrows.}
    \label{fig:cortex_vector_fields}
\end{figure}

We construct anatomically informed smooth vector fields. To characterize directional variation in cortical thickness relative to the local folding geometry, we construct two orthogonal tangent vector fields representing directions across and along cortical folds. Let $\MM\subset\mathbb{R}^3$ denote the cortical surface and let $s:\MM\rightarrow\mathbb{R}$ denote the sulcal-depth field obtained from FreeSurfer. Since $s$ is a scalar field defined intrinsically on the cortical surface, its Riemannian gradient $\nabla_\MM s(x)\in T_x\MM$ gives the direction of greatest local change in sulcal depth. At locations where $\|\nabla_\MM s(x)\|_g>0$, we therefore define the unit \emph{across-fold} vector field by $V_{\perp}(x)=\displaystyle\frac{\nabla_\MM s(x)}{\|\nabla_\MM s(x)\|_g}$. Thus, $V_{\perp}(x)$ is tangent to the cortical surface and locally orthogonal to the level sets of sulcal depth. A corresponding \emph{along-fold} direction is obtained by a $90^\circ$ rotation within the tangent plane. Let $n(x)$ denote the oriented unit normal to the cortical surface at $x$. We define $V_{\parallel}(x)=n(x)\times V_{\perp}(x)$. Since $n(x)$ and $V_{\perp}(x)$ are orthogonal unit vectors, $V_{\parallel}(x)\in T_x\MM$ is also a unit tangent vector and satisfies $\langle V_{\perp}(x),V_{\parallel}(x)\rangle_g=0$. Consequently, $\{V_{\perp}(x),V_{\parallel}(x)\}$ forms a local orthonormal frame for $T_x\MM$ wherever $\nabla_\MM s(x)\neq0$. The field $V_{\parallel}$ is tangent to the local level sets of sulcal depth and therefore represents variation along the cortical folding pattern, whereas $V_{\perp}$ represents variation across these level sets. For the triangulated cortical surfaces, $s$ is represented by its values at the mesh vertices. Its intrinsic gradient is computed piecewise on each triangular face using the gradients of the linear FEM basis functions. The resulting facewise gradient is normalized to obtain $V_{\perp}$, and $V_{\parallel}$ is constructed using the oriented face normal. Faces for which the norm of the sulcal-depth gradient is numerically negligible are excluded because the corresponding directional frame is not well defined. \Cref{fig:cortex_vector_fields} shows the resulting vector fields.

\begin{figure}[t]
    \centering
    \includegraphics[width=0.9\linewidth]{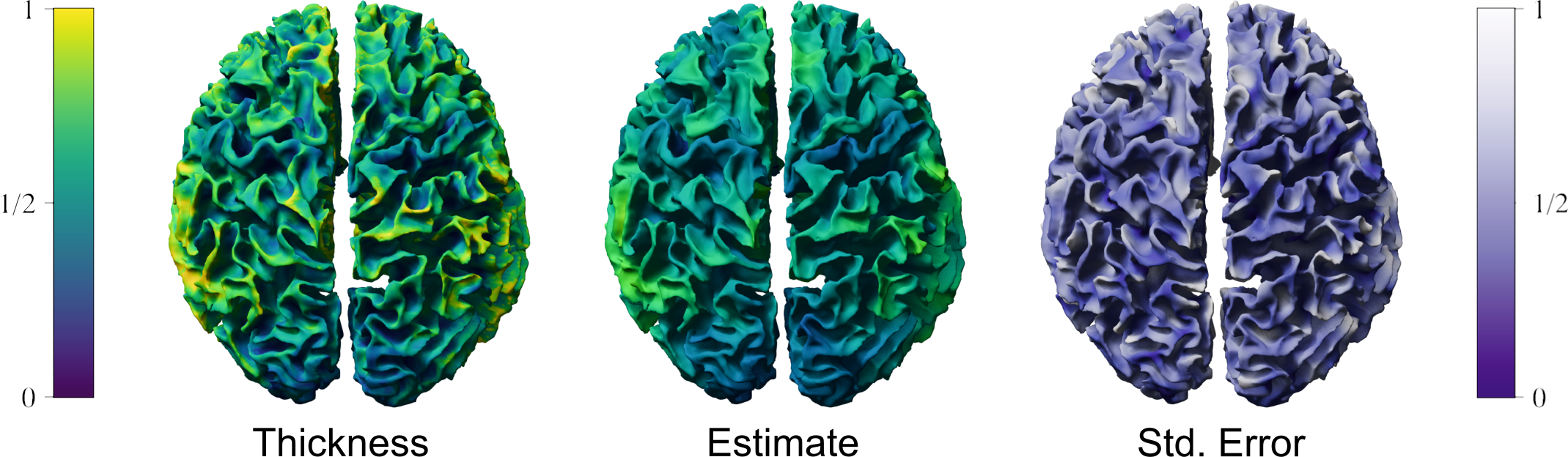}
    \caption{Observed cortical thickness (left), posterior median fitted thickness (center), and posterior standard deviation (right) for subject 50794. Left and right hemispheres are displayed in a dorsal view using a common color scale for observed and fitted thickness.}
    \label{fig:KKI_0050794}
\end{figure}

\begin{figure}[t]
    \centering
    \includegraphics[width=0.9\linewidth]{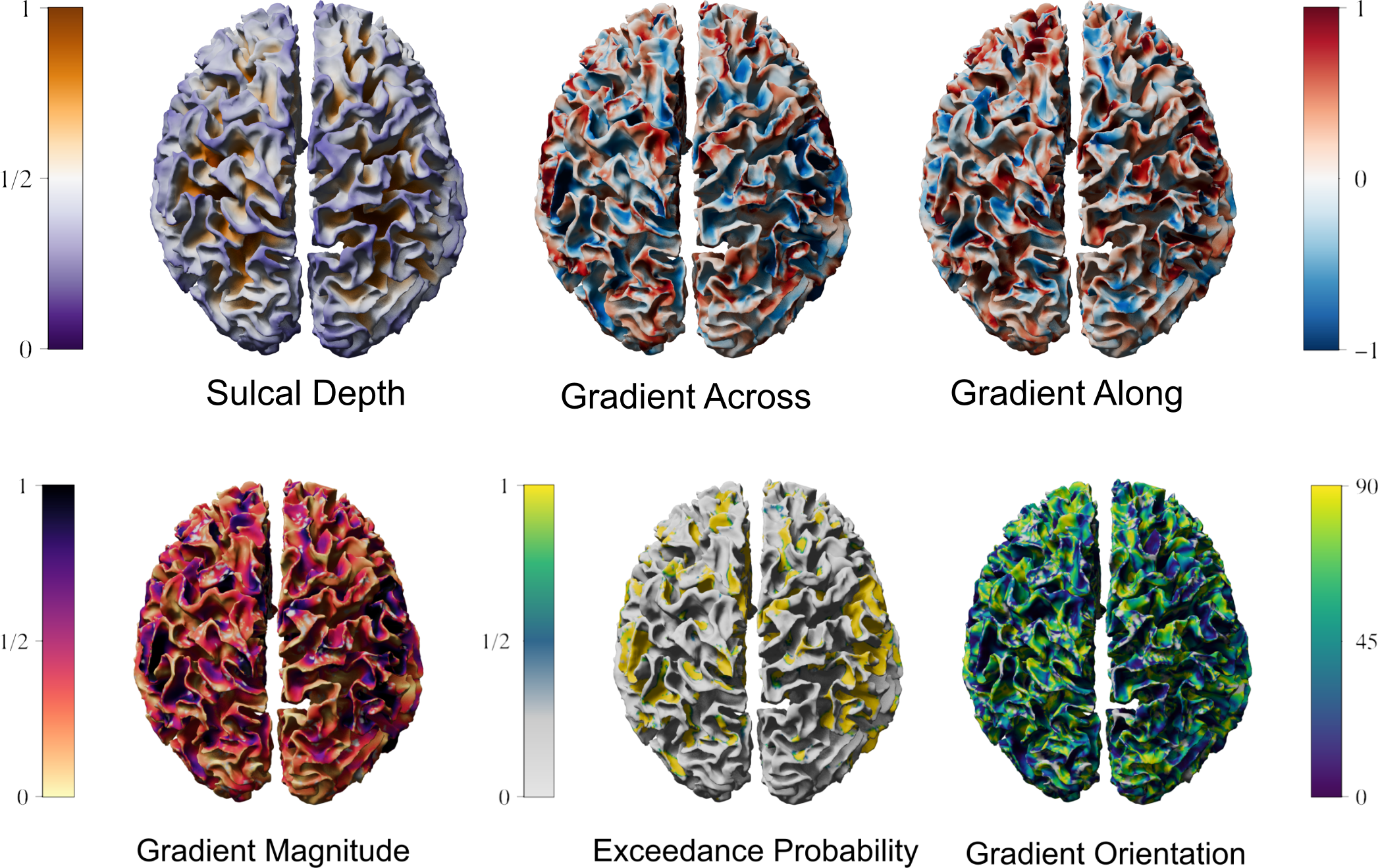}
    \caption{Posterior inference for directional changes in cortical thickness for subject 50794. Sulcal depth (top left) defines the local cortical reference frame, with $V_\perp$ denoting the across-fold direction and $V_\parallel$ the along-fold direction. The top center and right panels show the posterior median across- and along-fold directional derivatives, respectively. The bottom left panel shows the posterior median gradient magnitude, while the bottom center panel displays the posterior exceedance probability. The bottom right panel shows the orientation of the thickness gradient relative to cortical folding, $\theta$, where $0^\circ$ corresponds to variation entirely across folds and $90^\circ$ to variation entirely along folds. Across- and along-fold directional derivatives are displayed using a common zero-centered color scale, with limits robustly truncated for visualization; posterior exceedance probabilities and gradient orientations are displayed on their natural $[0,1]$ and $[0^\circ,90^\circ]$ scales, respectively.}
    \label{fig:KKI_0050794_grad_inference}
\end{figure}

Given a cortical-thickness process $Z$, we subsequently consider the directional derivative processes
$D_{\perp}Z(x)=\left\langle\nabla_\MM Z(x),V_{\perp}(x)\right\rangle_g$ and $D_{\parallel}Z(x)=\left\langle\nabla_\MM Z(x),V_{\parallel}(x)\right\rangle_g$. These respectively quantify the local rate of change in thickness across and along the cortical folding pattern. Since $V_{\perp}$ and $V_{\parallel}$ form an orthonormal basis of the tangent plane,
$\nabla_\MM Z(x)=D_{\perp}Z(x)V_{\perp}(x)+D_{\parallel}Z(x)V_{\parallel}(x)$, and hence
$\|\nabla_\MM Z(x)\|_g^2=\{D_{\perp}Z(x)\}^2+\{D_{\parallel}Z(x)\}^2$. We include two additional quantities, % $E_{\perp}(x)=\displaystyle\frac{D_\perp Z(x)^2}{\|\nabla_\MM Z(x)\|_g^2}\in [0,1]$ which measures fraction of local thickness-gradient oriented across cortical folds, and 
$\theta(x)=\displaystyle\tan^{-1}\frac{|D_\parallel Z(x)|}{|D_\perp Z(x)|}\in [0^\circ,90^\circ]$, which measures orientation of the change relative to cortical folding and the posterior exceedance probability $P\{\|\nabla_{\mathcal M}Z\|>c\given Y\}$, where $c$ is the third quartile of the posterior median gradient magnitude over the cortical surface. These provide subject-level \emph{geometric summaries} of cortical-thickness variation.
\begin{table}[t]
    \centering
    \caption{Phenotype information for the 8 subjects.}\label{tab:phenotype}
    \begin{tabular}{c|c|c|c|c|c}
    \hline
    \hline
       ID & Sampling Site & Subject &  Diagnosis & Gender & Age  \\\hline
       1 & KKI & 50794 & \multirow{4}{*}{ASD}  & Male & 9.66 \\
       2 & KKI & 50798 &   & Female & 10.17\\
       3 & LEUVEN\_2 & 50750 &  & Male & 16.80\\
       4 & UM\_1 & 50276 &  & Female & 16.80\\\hline
       5 & KKI & 50819 & \multirow{4}{*}{Control} & Male & 9.71 \\
       6 & KKI & 50788 &  & Female & 10.25\\
       7 & LEUVEN\_2 & 50742 &  & Male & 16.90\\
       8 & UM\_1 & 50340 &  & Female & 16.10\\
         \hline
         \hline
    \end{tabular}
\end{table}
We show detailed results for subject 50794. The subject is a 9.66-year-old male with an ASD diagnosis. We fit the model in \Cref{eq:hierarchical_model} separately to cortical thickness in the LH and RH. The mesh for LH and RH are comprised of 125,275 and 127,018 triangles respectively. We use the first 300 eigenpairs of the LB spectrum. The estimated model parameters are $\widehat{\sigma^2}_{\rm LH} = 3.19\,(2.38, 4.34)$, $\widehat{\tau^2}_{\rm LH} = 0.3918\,(0.3887, 0.3949)$, $\widehat{\alpha}_{\rm LH} = 0.0650\,(0.0589, 0.0717)$, $\widehat{\sigma^2}_{\rm RH} = 3.10\,(2.36, 4.23)$, $\widehat{\tau^2}_{\rm RH} = 0.4091\,(0.4059, 0.4123)$ and $\widehat{\alpha}_{\rm RH} = 0.0658\,(0.0594, 0.0719)$. The resulting model fit is shown in \Cref{fig:KKI_0050794}. We see that, $\frac{\widehat{\sigma^2}_{\rm LH}}{\widehat{\sigma^2}_{\rm LH}+\widehat{\tau^2}_{\rm LH}}\approx\frac{\widehat{\sigma^2}_{\rm RH}}{\widehat{\sigma^2}_{\rm RH}+\widehat{\tau^2}_{\rm RH}}\approx89\%$ of the variation in thickness was intrinsic. 

\Cref{fig:KKI_0050794_grad_inference} illustrates posterior inference for the cortical-thickness gradient. The estimated gradient exhibits substantial spatial structure, with both across- and along-fold directional derivatives varying over the cortical surface. The distinct spatial patterns of the two components indicate that local thickness variation is \emph{not uniformly aligned} with the cortical folding pattern. This is also reflected in the gradient-orientation map, which demonstrates considerable spatial variation in the alignment of the thickness gradient relative to the across- and along-fold directions. The gradient-magnitude map identifies localized regions of comparatively rapid thickness variation. Several of these regions are also associated with high posterior probabilities of exceeding the prespecified gradient threshold, which is set as the third quartile, indicating that the corresponding features are supported by the posterior distribution rather than arising solely from the posterior point estimate. Taken together, these results illustrate how the proposed framework characterizes not only the magnitude of local cortical-thickness variation, but also its orientation relative to cortical folding and the posterior uncertainty associated with regions of elevated gradient magnitude.

Let,  ${\rm Vol}_\MM={\rm vol}_g(\MM_{\rm LH})+{\rm vol}_g(\MM_{\rm RH})$. We construct the following, \begin{equation}\label{eq:summaries}
    \begin{split}
        A_\perp &= {\rm Vol}_\MM^{-1}\left(\int_{\MM_{\rm LH}}| D_\perp Z| \, dV_g + \int_{\MM_{\rm RH}}| D_\perp Z| \, dV_g\right),\\
        A_\nabla&={\rm Vol}_\MM^{-1}\left(\int_{\MM_{\rm LH}}\| \nabla_{\MM_{\rm LH}} Z\|_g \, dV_g + \int_{\MM_{\rm RH}}\| \nabla_{\MM_{\rm RH}} Z\|_g\, dV_g\right),\\
        R_\perp&=\frac{\int_{\MM_{\rm LH}}D_\perp Z^2 \, dV_g + \int_{\MM_{\rm RH}} D_\perp Z^2 \, dV_g}{\int_{\MM_{\rm LH}}\| \nabla_{\MM_{\rm LH}} Z\|^2_g \, dV_g + \int_{\MM_{\rm RH}}\| \nabla_{\MM_{\rm RH}} Z\|^2_g\, dV_g},
    \end{split}
\end{equation}
and $A_\parallel$ by replacing $ D_\perp Z$ with $ D_\parallel Z$ in the first quantity. We use them to compare subjects. Note that, $A_\perp, A_\parallel$ and $A_\nabla$ quantify the average magnitude of cortical-thickness variation across folds, along folds, and irrespective of direction, respectively, whereas $R_\perp$ characterizes the orientation of this variation by measuring the fraction of total gradient energy attributable to the across-fold component. Contributions from the left and right hemispheres are combined at the integral level before normalization by total cortical surface area.

\begin{figure}[t]
    \centering
    \includegraphics[width=\linewidth]{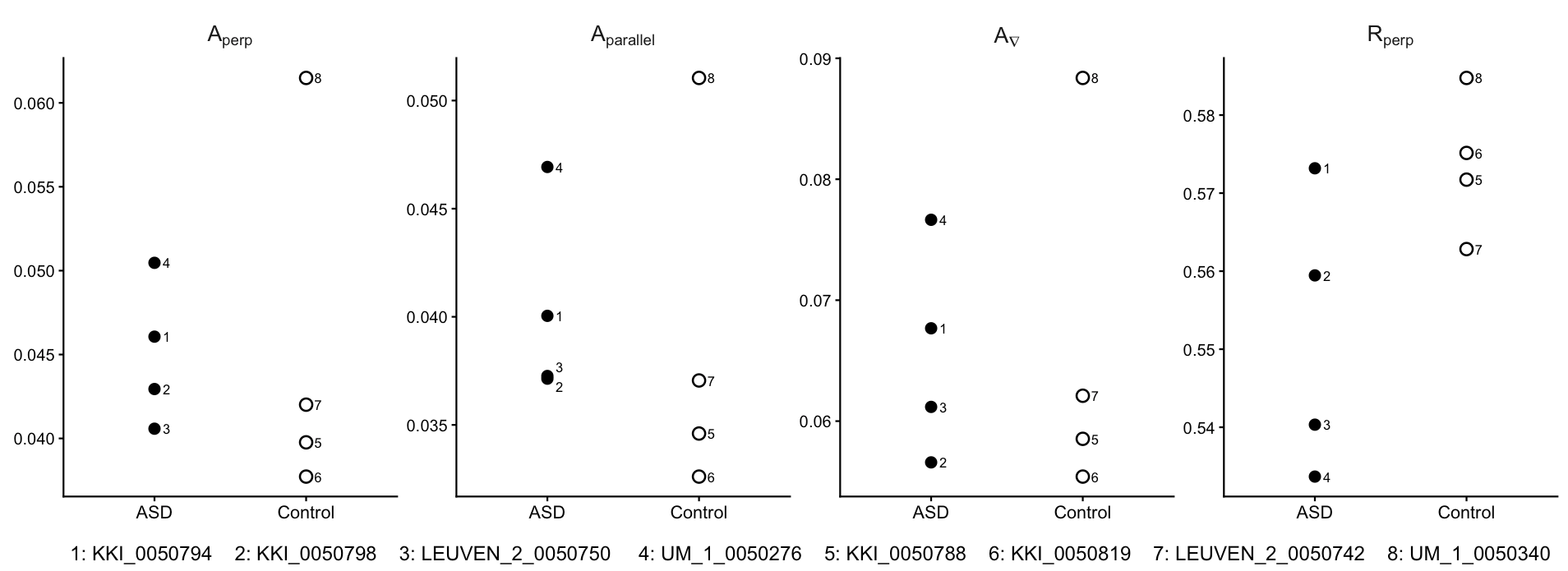}
    \caption{Subject-level summaries of cortical-thickness gradients for the eight ABIDE subjects. Filled circles denote subjects with ASD and open circles denote controls; numerical labels identify individual subjects as listed below the panels and in \Cref{tab:phenotype}.}
    \label{fig:eight_subjects}
\end{figure}

Considerable between-subject heterogeneity is evident in the magnitude-based summaries from \Cref{fig:eight_subjects}. The across- and along-fold summaries generally vary together with the total gradient magnitude, although one control subject (subject 8) exhibits substantially larger values of all three quantities than the remaining controls. Excluding this individual, the ASD subjects in this illustrative sample tend to exhibit larger values of $A_\perp$, $A_\parallel$, and $A_\nabla$ than the remaining controls. The orientation summary $R_\perp$, however, displays a different pattern: three of the four ASD subjects have lower across-fold gradient-energy fractions than all four controls, while one ASD subject lies within the control range. Thus, differences in the overall magnitude of cortical-thickness variation need not coincide with differences in how that variation is oriented relative to cortical folding.

These comparisons are intended to demonstrate the subject-level summaries afforded by the proposed derivative framework rather than to establish diagnostic-group differences. In particular, the small number of subjects and the visible between-subject heterogeneity preclude population-level conclusions regarding ASD. Nevertheless, the results illustrate how the proposed approach separates the overall rate of cortical-thickness variation from its directional organization relative to cortical folding, providing quantities that could be investigated in a larger population-level analysis.
}

\section{Discussion}\label{sec: discussion}

We developed a framework for statistical inference on intrinsic rates of change of Gaussian processes indexed by compact Riemannian manifolds, possibly with boundary. The central distinction from conventional spatial gradient analysis is geometric: directions of change are elements of the tangent space rather than directions in an ambient Euclidean coordinate system. By defining differential processes along smooth vector fields, the framework separates the stochastic model for the underlying field from the scientifically meaningful directions along which its variation is interrogated. We established mean-square regularity conditions for spectral Whittle-Mat\'ern Gaussian processes and derived the joint covariance structure of the parent and differential processes, providing a coherent probabilistic construction for intrinsic derivative inference. We further characterized the effect of finite spectral approximation by deriving uniform truncation-error bounds for the parent covariance and its differential cross-covariances. These results quantify the progressively slower convergence of covariance components involving higher-order differentiation and provide sufficient conditions under which spectral truncation preserves strict positive definiteness of the finite covariance matrix. Together with the resulting low-rank Bayesian formulation, these developments connect the regularity of the underlying infinite-dimensional process, the accuracy and stability of its finite spectral representation, and model-based uncertainty quantification for intrinsic rates of change.

%The simulations highlight an important interaction between differential inference and spectral approximation. Finite truncation produces smooth process realizations and enables substantial computational savings, but derivatives are more sensitive than the parent process to omission of high-frequency eigenfunctions. This is reflected in the Stanford Bunny experiment, where process reconstruction changes little with increasing truncation while derivative coverage improves appreciably. % Thus, truncation levels adequate for prediction of a spatial field need not be adequate for inference on its local rates of change. 
%Developing approximation-error bounds specifically for differential processes and principled criteria for selecting the spectral truncation level are important directions for future work.

The applications demonstrate the role of scientifically chosen vector fields. For wearable-device trajectories, intrinsic differentiation with respect to arc length identifies changes in latent heart rate independently of the geographic representation of the route. For cortical thickness, sulcal depth supplies an anatomically informed tangent frame that decomposes variation into across- and along-fold components. This permits posterior characterization of the magnitude and orientation of local thickness variation relative to cortical folding and produces subject-level summaries that are distinct from conventional analyses of thickness alone. The 8-subject comparison is illustrative rather than inferential; cohort-level investigation of these summaries in larger neuroimaging cohorts is a natural next step.

Several theoretical extensions remain. Our development focuses primarily on pointwise differential processes, while classical wombling often seeks inference for rates of change aggregated along curves or boundaries. On a Riemannian manifold, this suggests defining wombling functionals over smooth curves. A second direction concerns higher-order differential processes. While the present work considers repeated differentiation along a single vector field, derivatives involving distinct vector fields require an explicit treatment of the affine connection. Formulating such processes through covariant derivatives would permit inference for mixed directional derivatives. %and more general components of the intrinsic Hessian, providing a systematic foundation for higher-order derivative inference on manifolds. 
{ 
Beyond scalar-valued fields, an especially interesting direction is to consider stochastic processes whose response spaces are themselves manifolds, as in recent work on manifold-valued curves and geometric functional data \citep[see, e.g.][]{bharath_analysis_2020,wang_probabilistic_2024,preston_rolled_2026}. Allowing both the index and response spaces to be manifolds would lead to random maps $Z:\mathcal M_1\rightarrow\mathcal M_2$, for which local rates of change are naturally characterized by the differential $dZ_x:T_x\mathcal M_1\rightarrow T_{Z(x)}\mathcal M_2$, providing a geometric extension of the scalar differential processes considered here. 
}

\section*{Declaration of the use of generative AI and AI-assisted technologies}
During the preparation of this work, the authors supervised ChatGPT GPT-5.6 Sol and Claude Sonnet 4.6 in order to improve the exposition of proofs and the general writing and organization of the manuscript. The authors independently derived all technical results, reviewed and edited the content as necessary, and are fully responsible for the content of the manuscript.

\section*{Acknowledgments and Funding}
AH acknowledges funding from the National Science Foundation, Division of Mathematical Sciences (DMS 2515897), the National Institutes of Health, National Institute of Environmental Health Sciences (R56 ES036250, R21 CA305033), and institutional funds at the Dornsife School of Public Health and the Urban Health Collaborative at Drexel University. DL acknowledges funding from the National Science Foundation, Division of Mathematical Sciences (DMS 2515899), and the National Institutes of Health (P42 ES031007, P50 CA058223, P30 ES010126, R01 LM014407). SB acknowledges funding from the National Science Foundation, Division of Mathematical Sciences (DMS 2515898), and the National Institute of Health, National Institute of General Medical Sciences (R01 GM148761)
% %%%%%%%%%%%%%%%%%%%%%%%%%%%%%%%%%%%%%%%%%%%%%%
% %% Example with single Appendix:            %%
% %%%%%%%%%%%%%%%%%%%%%%%%%%%%%%%%%%%%%%%%%%%%%%
% \begin{appendix}
% \section*{Title}\label{appn} %% if no title is needed, leave empty \section*{}.
% Appendices should be provided in \verb|{appendix}| environment,
% before Acknowledgements.

% If there is only one appendix,
% then please refer to it in text as \ldots\ in the \hyperref[appn]{Appendix}.
% \end{appendix}
%%%%%%%%%%%%%%%%%%%%%%%%%%%%%%%%%%%%%%%%%%%%%%
%% Example with multiple Appendixes:        %%
%%%%%%%%%%%%%%%%%%%%%%%%%%%%%%%%%%%%%%%%%%%%%%
\begin{appendix}
\section{Proofs for Section \ref{sec:cont}}\label{apdx:proof_sec:cont}
In this section, we present the proofs of \Cref{thm:MSC} and \Cref{prop:MSC_iso}.
\begin{proof}[Proof of \Cref{thm:MSC}]
Let {  $x\in\MM^\circ$ and $\gamma:(-\delta,\delta)\to \MM$ be a geodesic with $\gamma(0)=x$, while for $x\in \partial\MM$ and $\gamma:[0,\delta)\to \MM$ be an admissible smooth curve with $\gamma(0)=x$. In what follows, limits $t\to0$ are understood as two-sided for $x\in M^\circ$ and as $t\downarrow0$ for $x\in\partial M$.} We have,
\begin{align*}
&\EE(Z(\gamma(t))-Z(x))^2 = \EE \left(Z(\gamma(t))Z(\gamma(t))-2Z(\gamma(t))Z(\gamma(0))+Z(\gamma(0))Z(\gamma(0))\right)\\
& = K(\gamma(t),\gamma(t))-2K(\gamma(t),\gamma(0))+K(\gamma(0),\gamma(0))\\
& = \frac{\sigma^2}{C_{\nu,\alpha}}\sum_{l=0}^\infty (\alpha^2+\lambda_l)^{-\nu-\frac{p}{2}}\left(f_l(\gamma(t))f_l(\gamma(t))-2f_l(\gamma(t))f_l(\gamma(0))+f_l(\gamma(0))f_l(\gamma(0))\right)\\
& = \frac{\sigma^2}{C_{\nu,\alpha}}\sum_{l=0}^\infty (\alpha^2+\lambda_l)^{-\nu-\frac{p}{2}}\left(f_l(\gamma(t))-f_l(\gamma(0))\right)^2 \eqqcolon \frac{\sigma^2}{C_{\nu,\alpha}}\sum_{l=0}^\infty\xi_l(t). 
\end{align*}
To switch the limit $\lim\limits_{t\to0}$ {  (or $\lim\limits_{t\downarrow 0}$)} with $\sum_{l=0}^\infty$, we need the uniform convergence of the series $\sum_{l=0}^\infty\xi_l(t)$. Observe that  $\|f_l\|_\infty \leq C\lambda_l^{\frac{p-1}{4}}\|f_l\|_2=C\lambda_l^{\frac{p-1}{4}}$, where $C>0$ is a constant {  for $\partial\MM=\phi$ \citep[see][]{donnelly2001bounds} and $\partial\MM\ne \phi$, with Neumann conditions, \citep[see][Theorem 1]{grieser_uniform_2002}}. Then let $a_l \coloneqq4C^2(\alpha^2+\lambda_l)^{-\nu-\frac{p}{2}}\lambda_l^{\frac{p-1}{2}}$, so $|\xi_l(t)|\leq a_l$. Then by Weistrass M-test, it suffices to show $\sum_{l=0}^\infty a_l$ converges. By Weyl's law {  \citep[see, e.g.][]{seeley_estimate_1980}}, $\lambda_l 	\asymp l^{2/p}$. So, when $\nu>0$
$$\sum_{l=0}^\infty a_l\asymp \sum_{l=0}^\infty (\alpha^2+l^{2/p})^{-\nu-p/2}l^{\frac{p-1}{p}}\asymp \sum_{l=0}^\infty l^{-\frac{2\nu+1}{p}}<\infty.$$ 
As a result, when $\nu>\frac{p-1}{2}$,
\begin{align*}
\lim\limits_{t \to 0}\EE(Z(\gamma(t))-Z(x))^2 & = \sum_{l=0}^\infty (\alpha^2+l^{2/p})^{-\nu-p/2}l^{\frac{p-1}{p}}\lim\limits_{t\to0}\left(f_l(\gamma(t))-f_l(\gamma(0))\right)^2= 0,
\end{align*}
by the continuity of $f_l$ and $\gamma$. The proof for RBF is obtained by replacing $(\alpha^2+\lambda_l)^{-\nu-\frac{p}{2}}$ with $e^{-\frac{\lambda_l}{2\alpha^2}}$.
\end{proof}

\begin{proof}[Proof of \Cref{prop:MSC_iso}]
    Let $\gamma(t) = \exp_x(tv)$ where $v\in T_x\MM$ and $t\in[0,\delta)$ for some $\delta>0$ such that $\exp$ is a local diffeomorphism, then 
    \begin{align*}
    \EE(Z(\gamma(t))-Z(x))^2 &= \EE(Z(\exp_x(tv))^2-2Z(\exp_x(tv))Z(x)+Z(x)^2)\\
    & = 2\left(K(0)-K(t)\right),
    \end{align*}
    since $d_\MM(\exp_x(tv),x) = t$. As a result, $\lim_{t\to0}\EE(Z(\gamma(t))-Z(x))^2$ exists if and only if $K$ is continuous at $0$. 
\end{proof}

\section{Proofs for Section \ref{sec:deriv}}
In this section, we present proofs of \Cref{thm:1MSD}, \Cref{lem:Taylor_1}, \Cref{lem:D_VZ_cov}, \Cref{lem:ZD_VZ}, \Cref{cly:ZD_VZMat} and \Cref{prop:MSD_iso}.

\begin{proof}[Proof of \Cref{thm:1MSD}]
Let {  $x\in\MM^\circ$ and $\gamma:(-\delta,\delta)\to \MM$ be a geodesic with $\gamma(0)=x$, while for $x\in \partial\MM$ and $\gamma:[0,\delta)\to \MM$ be an admissible smooth curve with $\gamma(0)=x$. In what follows, limits $t\to0$ are understood as two-sided for $x\in M^\circ$ and as $t\downarrow0$ for $x\in\partial M$.} Similar to the previous proof, we have 
\begin{align*}
\EE\left(\frac{Z(\gamma(t))-Z(x)}{t}\right)^2 & = \frac{\sigma^2}{C_{\nu,\alpha}}\sum_{l=0}^\infty (\alpha^2+\lambda_l)^{-\nu-\frac{p}{2}}\frac{\left(f_l(\gamma(t))-f_l(\gamma(0))\right)^2}{t^2}\\
&\eqqcolon \frac{\sigma^2}{C_{\nu,\alpha}}\sum_{l=0}^\infty\eta_l(t).
\end{align*}
To switch the limit $\lim\limits_{t\to0}$ and the sum $\sum_{l=0}^\infty$, we need the uniform convergence of the series $\sum_{l=0}^\infty\eta_l(t)$. Observe that $\lim\limits_{t\to0} \frac{|f_l(\gamma(t))-f_l(\gamma(0))|}{t^2}=|(f_l\circ\gamma)'(0)|\leq C\|\nabla f_l\|_\infty$ where $C>0$ is a constant. {  Using} \cite{shi2010gradient} {  for $\partial\MM=\phi$ and} \cite{hu_gradient_2015,arnaudon2020gradient} {  for $\partial\MM\ne\phi$}, we have
$$\|\nabla f_l\|_\infty \leq C\sqrt{\lambda_l}\|f_l\|_\infty \leq C \sqrt{\lambda_l}\lambda_l^{\frac{p-1}{4}}\|f_l\|_2=C\lambda_l^{\frac{p+1}{4}}.$$

Let $b_l \coloneqq(\alpha^2+\lambda_l)^{-\nu-\frac{p}{2}}\lambda_l^{\frac{p+1}{2}}$. Then $|\eta_l(t)|\leq 4C^2b_l$. Similarly, when $\nu>\frac{p+1}{2}$
$$\sum_{l=0}^\infty b_l\asymp \sum_{l=0}^\infty (\alpha^2+l^{2/p})^{-\nu-p/2}l^{\frac{p+1}{p}}\asymp \sum_{l=0}^\infty l^{-\frac{2\nu-1}{p}}<\infty.$$ 
As a result, when $\nu>\frac{p+1}{2}$,
\begin{align*}
\lim\limits_{t \to 0}\EE\left(\frac{Z(\gamma(t))-Z(x)}{t}\right)^2 & = \sum_{l=0}^\infty (\alpha^2+l^{2/p})^{-\nu-p/2}l^{\frac{p-1}{p}}\lim\limits_{t\to0}\frac{\left(f_l(\gamma(t))-f_l(\gamma(0))\right)^2}{t^2}\\
& = \sum_{l=0}^\infty (\alpha^2+l^{2/p})^{-\nu-p/2}l^{\frac{p-1}{p}}(f_l\circ\gamma)'(0)^2
\end{align*}
by the smoothness of of $f_l$ and $\gamma$. The proof for RBF is obtained by replacing $(\alpha^2+\lambda_l)^{-\nu-\frac{p}{2}}$ by $e^{-\frac{\lambda_l}{2\alpha^2}}$.
\end{proof}

{ 
\begin{proof}[Proof of \Cref{prop:MSD-boundary}]
    Using the spectral representation, for $x\in\partial\MM$, $D_\eta Z(x)=\sum\limits_{l=0}^{\infty}\varphi_lD_{\eta_x} f_l(x)$, where $\eta_x$ is the inward pointing co-normal at $x$. Since the Neumann condition requires $D_\eta f_l(x)=0$ for every $x\in \partial\MM$, $D_\eta Z(x)=0$. Using the orthogonal decomposition of $V_x=V_x^{\rm T}+\<V_x,\eta_x\>_g\eta_x$ and linearity of the directional derivative operator we obtain, $D_VZ(x)=D_{V^{\star}}Z(x)+\<V_x,\eta_x\>_gD_{\eta_x} Z(x)=D_{V^{\star}}Z(x)$.
\end{proof}
}

\begin{proof}[Proof of \Cref{lem:Taylor_1}]
Note that
\begin{align*}
&\lim\limits_{t\to 0}\EE\left(\frac{Z(\exp_x(tv)) - Z(x) - t \<v,\nabla Z(x)\>}{t}\right)^2 \\
& = \lim\limits_{t\to 0}\EE\left(\frac{Z(\exp_x(tv)) - Z(x) - t\lim\limits_{s\to 0} \frac{Z(\gamma(s))-Z(x)}{s}}{t}\right)^2\\
& = \lim\limits_{t\to 0} \sum_{l=0}^\infty t^2 (\alpha^2+\lambda_l)^{-\nu-p/2}\left[f_l(\gamma(t))f_l(\gamma(t))-2f_l(\gamma(t))f_l(x)+f_l(x)f_l(x)\right.\\&\pushright{+\left.t^2(f_l\circ\gamma)'(0)(f_l\circ\gamma)'(0)-2tf_l(\gamma(t))(f_l\circ\gamma)'(0)+2tf_l(x)(f_l\circ\gamma)'(0)\right]}\\
& = \sum_{l=0}^\infty (\alpha^2+\lambda_l)^{-\nu-p/2}\left((f_l\circ \gamma)'(0)+(f_l\circ \gamma)'(0)-2(f_l\circ \gamma)'(0)\right) =0.
\end{align*}
\end{proof}
\begin{proof}[Proof of \Cref{lem:D_VZ_cov}]
By definition, the mean function of $D_VZ$ is
\begin{align*}
\EE(D_VZ(x))&=\EE\left(\lim_{t\to0} \frac{Z(\exp_x(tV_x))-Z(x)}{t}\right)\\
& = \lim_{t\to0} \frac{\mu(\exp_x(tV_x))-\mu(x)}{t}= \<\nabla \mu(x),V_x\>.
\end{align*}

For the covariance of $D_VZ$, by definition,
\begin{align*}
&K_V(x,x')=\cov(D_VZ(x),D_VZ(x'))\\
& = \cov\left(\lim_{t\to 0}\frac{Z(\exp_x(tV_x))-Z(x)}{t},\lim_{s\to 0}\frac{Z(\exp_{x'}(sV_{x'}))-Z(x')}{s}\right)\\
& = \lim_{t,s\to0} \frac{1}{ts}[K(\exp_x(tV_x),\exp_{x'}(sV_{x'}))-K(x,\exp_{x'}(sV_{x'})){-K(\exp_x(tV_x),x')+K(x,x')]}\\
%& = \lim_{s\to 0}\left[\<\nabla_1 K(x,\exp_{x'}(sV_{x'})),V_{x}\>-\<\nabla_1 K(x,x'),V_x\>\right] = (\nabla_{12} K(x,x'))(V_x,V_{x'}).
& = \lim_{s\to 0}\frac{1}{s}\left[\<\nabla_1 K(x,\exp_{x'}(sV_{x'})),V_{x}\>-\<\nabla_1 K(x,x'),V_x\>\right]\\ & = (\nabla_{12} K(x,x'))(V_x,V_{x'}).
\end{align*}
\end{proof}
\begin{proof}[Proof of \Cref{lem:ZD_VZ}]
By definition,
\begin{align*}
\cov(Z(x),D_VZ(x'))& = \cov\left(Z(x),\lim_{s\to 0}\frac{Z(\exp_{x'}(sV_{x'}))-Z(x')}{s}\right)\\
& = \lim_{s\to0} \frac{1}{s}\left[K(x,\exp_{x'}(sV_{x'}))-K(x,x')\right] = \nabla_2 K(x,x')(V_{x'}).
\end{align*}
Proceeding in a similar fashion we have, $\cov(D_VZ(x),Z(x'))= \nabla_2 K(x,x')(V_x)$.
\end{proof}
\begin{proof}[Proof of \Cref{cly:ZD_VZMat}]
 We focus on the spectral Whittle-Mat\'ern kernel since the proof for RBF is similar. For $K_VZ$, we plug in the series representation of $K$ into the third equation in the proof of Theorem \ref{lem:D_VZ_cov}:
\begin{align*}
    K_V(x,x')&=\lim_{t,s\to0} \frac{1}{ts}\left[K(\exp_x(tV_x),\exp_{x'}(sV_{x'}))-K(\exp_x(tV_x),x')-K(x,\exp_{x'}(sV_{x'}))+K(x,x')\right]\\
    & = \frac{\sigma^2}{C_{\nu,\alpha}}\lim_{t,s\to0} \frac{1}{ts}\sum_{l=0}^\infty (\alpha^2+\lambda_l)^{-\nu-\frac{p}{2}}\left[f_l(\exp_x(tV_x))f_l(\exp_{x'}(sV_{x'}))-f_l(\exp_x(tV_x))f_l(x')\right.\\&\hspace{6cm}\left.-f_l(x)f_l(\exp_{x'}(sV_{x'}))+f_l(x)f_l(x')\right]\\
    & = \frac{\sigma^2}{C_{\nu,\alpha}}\lim_{t\to0} \frac{1}{t}\sum_{l=0}^\infty (\alpha^2+\lambda_l)^{-\nu-\frac{p}{2}}\left[f_l(\exp_x(tV_x))\nabla f_l(V_{x'})-f_{l}(x)\nabla f_l(V-{x'})\right]\\
    & = \frac{\sigma^2}{C_{\nu,\alpha}}\sum_{l=0}^\infty (\alpha^2+\lambda_l)^{-\nu-\frac{p}{2}}\nabla f_l(V_x)\nabla f_l(V_{x'}).
\end{align*}
For $\cov(Z(x),D_VZ(x'))$, we plug in the series representation of $K$ into \Cref{eqn:ZD_VZ}:
\begin{align*}
    \cov(Z(x),D_VZ(x'))&=\lim_{s\to0} \frac{1}{s}\left[K(x,\exp_{x'}(sV_{x'}))-K(x,x')\right]\\
    & = \frac{\sigma^2}{C_{\nu,\alpha}}\lim_{s\to0} \frac{1}{s}\sum_{l=0}^\infty (\alpha^2+\lambda_l)^{-\nu-\frac{p}{2}}f_l(x)\left[f_l(\exp_{x'}(sV_{x'}))-f_l(x')\right]\\
    & =  \frac{\sigma^2}{C_{\nu,\alpha}}\sum_{l=0}^\infty (\alpha^2+\lambda_l)^{-\nu-\frac{p}{2}}f_l(x)\nabla f_l(V_{x'}).
\end{align*}
\end{proof}
\begin{proof}[Proof of \Cref{prop:MSD_iso}]
 Let $\gamma(t) = \exp_x(tv)$ where $v\in T_x\MM$ and $t\in[0,\delta)$ for some $\delta>0$ such that $\exp$ is local diffeomorphism, then similar to the proof of \Cref{prop:MSC_iso}, we have
    \begin{align*}
    \EE(Z(\gamma(t))-Z(x))^2=2K(0)-2K(t),
    \end{align*}
    As a result, $\lim_{t\to0}\frac{\EE(Z(\gamma(t))-Z(x))^2}{t^2}$ exists if and only if $K(t)=K(0)+O(t^2)$, or equivalently, $K'(0)=0$ and $K''(0)<\infty$.
\end{proof}

\section{Proofs for Section \ref{sec:trunc}}\label{apdx:sec:trunc}

\begin{proof}[Proof of \Cref{thm:truncation}]
Let $x\in\MM$ and $\gamma:(-\delta,\delta)\to \MM$ be a geodesic  with $\gamma(0)=x$. We show the truncated process $Z^T$ is MSC first. 
\begin{align*}
&\EE(Z^T(\gamma(t))-Z^T(x))^2  \\
&= \EE \left(Z^T(\gamma(t))Z^T(\gamma(t))-2Z^T(\gamma(t))Z^T(\gamma(0))+Z^T(\gamma(0))Z^T(\gamma(0))\right)\\
& = \widetilde{K}(\gamma(t),\gamma(t))-2\widetilde{K}(\gamma(t),\gamma(0))+\widetilde{K}(\gamma(0),\gamma(0))\\
& = \sum_{l=0}^{T-1}a_l\left(f_l(\gamma(t))f_l(\gamma(t))-2f_l(\gamma(t))f_l(\gamma(0))+f_l(\gamma(0))f_l(\gamma(0))\right)\\
& = \sum_{l=0}^{T-1}a_l\left(f_l(\gamma(t))-f_l(\gamma(0))\right)^2. 
\end{align*}
Since the sum is finite, we can exchange the limit and sum, so by the continuity of the harmonic functions $f_l$, we have
\begin{equation*}
    \begin{split}
        \lim_{t\to0} \EE(Z^T(\gamma(t))-Z^T(x))^2 &=\lim_{t\to0}\sum_{l=0}^{T-1}a_l\left(f_l(\gamma(t))-f_l(\gamma(0))\right)^2,\\
        &=\sum_{l=0}^{T-1}a_l\lim_{t\to0}\left(f_l(\gamma(t))-f_l(\gamma(0))\right)^2=0.
    \end{split}
\end{equation*}
By the same argument on the finite sum and exchangeability of the limit and the sum, we can show 1-MSD:
\begin{align*}
\lim_{t\to0}\EE\left(\frac{Z^T(\gamma(t))-Z^T(x)}{t}\right)^2 & = \lim_{t\to0}\sum_{l=0}^{T-1} a_l\frac{\left(f_l(\gamma(t))-f_l(\gamma(0))\right)^2}{t^2},\\
&=\sum_{l=0}^{T-1} a_l\lim_{t\to0}\frac{\left(f_l(\gamma(t))-f_l(\gamma(0))\right)^2}{t^2}<\infty,
\end{align*}
by the smoothness of $f_l$. For 2-MSD,
\begin{align*}
&\lim_{t\to0}\EE\left(\frac{D_VZ^T(\gamma(t))-D_VZ^T(x)}{t}\right)^2 \\
&= \lim_{t\to0}\frac{1}{t^2} \left[K^{T}_V(\gamma(t),\gamma(t))-2\widetilde{K}_V(\gamma(t),x)+\widetilde{K}_V(x,x)\right]\\
& = \lim_{t\to0}\sum_{l=0}^{T-1}a_l\frac{\left(\nabla f_l(V_{\gamma(t)})-\nabla f_l(V_x)\right)^2}{t^2}= \sum_{l=0}^{T-1}a_l\lim_{t\to0}\frac{\left(\nabla f_l(V_{\gamma(t)})-\nabla f_l(V_x)\right)^2}{t^2}<\infty
\end{align*}
by the smoothness of $f_l$ and $V$. 
\end{proof}
{ 
\begin{proof}[Proof of \Cref{thm:truncation_error}]
    Assume $V$ has unit norm. We show this for the spectral Whittle-Mat\'ern case, the proof for RBF is similar. Note that using \cite{donnelly2001bounds,grieser_uniform_2002,shi2010gradient,hu_gradient_2015,arnaudon2020gradient}, for $r,s\leq2$ we have
    \begin{align*}
        |C_{(r,s)}-C^T_{(r,s)}|&\lesssim \frac{\sigma^2}{C_{\alpha,\nu}}\sum\limits_{l=T}^{\infty}\lambda_l^{-\nu-\frac{p}{2}}\lambda_l^{\frac{p-1+2r}{4}}\lambda_l^{\frac{p-1+2s}{4}}=\frac{\sigma^2}{C_{\alpha,\nu}}\sum\limits_{l=T}^{\infty}\lambda_l^{-\frac{2\nu-r-s+1}{2}}\\
        &\lesssim\frac{\sigma^2}{C_{\alpha,\nu}}\sum\limits_{l=T}^{\infty}l^{-\frac{2\nu-r-s+1}{p}}\\
        &\lesssim \frac{\sigma^2}{C_{\alpha,\nu}}\int_{T}^\infty t^{-\frac{2\nu-r-s+1}{p}} {\rm d}t=\frac{\sigma^2}{C_{\alpha,\nu}}\frac{p}{2\nu-p-r-s+1}T^{-\frac{2\nu-p-r-s+1}{p}}.
    \end{align*}
    The second inequality is from Weyl's law---for sufficiently large $l$, $\lambda_l\geq c_\MM l^{\frac{2}{p}}$, where $c_\MM$ is the Weyl's constant.
    Hence, $\|\widetilde{R}_{(r,s)}\|_\infty\lesssim\frac{\sigma^2}{C_{\alpha,\nu}}\frac{p}{2\nu-p-r-s+1}T^{-\frac{2\nu-p-r-s+1}{p}}$, provided $\nu>\frac{p+r+s-1}{2}$.
\end{proof}

\begin{proof}[Proof of \Cref{thm:truncation_pd}]
    Define, $R_{W^*}^T=\Sigma_{W^*}-\widetilde{\Sigma}_{W^*}$, then from \Cref{thm:truncation_pd}, $\|R_{W^*}^T\|_2\leq\|R_{W^*}^T\|_\infty\leq (n+m)C_WT^{-\frac{2\nu-p-1}{p}}$, where $C_W =\max_{0\leq r,s\leq 1}C_{(r,s)}$, is a constant that is independent of $T$, $n$ and $m$. Using Weyl's perturbation inequality \citep[see, e.g.,][]{bhatia_matrix_1997},
    \begin{equation*}
        \lambda_{\rm min}(\Sigma^T_{W^*})\geq  \lambda_{\rm min}(\Sigma_{W^*}) - \|R_{W^*}^T\|_2 \geq \lambda_{\rm min}(\Sigma_{W^*}) - (n+m)C_WT^{-\frac{2\nu-p-1}{p}}.
    \end{equation*}
    Define, $A^T=\begin{pmatrix}
        F^T(X_n)\\D_VF^T(X_m^*)%\\D^2_VF^T(X_m^*)
    \end{pmatrix}$, which is a matrix of order $(n+m)\times T$ then, $\Sigma^T_{W^*}=A^TQ^T{A^T}\T$, where $Q^T={\rm diag}\{a_0,\ldots,a_{T-1}\}$. Note that if $T < n+m$ then, $\Sigma^T_{W^*}$ is rank deficient. Equivalently,
    \begin{equation*}
        T >\max\left\{n+m-1,\left[\frac{(n+m)C_W}{\lambda_{\rm min}(\Sigma_{W^*})}\right]^{\frac{p}{2\nu-p-1}}\right\},
    \end{equation*}
    is sufficient to ensure strict positive definiteness of $\Sigma^T_{W^*}$.
\end{proof}
}
\section{Proofs for Section \ref{subsec:sphere}}\label{apdx:subsec:sphere}
In this section, we present proofs of \Cref{thm:MSC1-MSD_iso}, \Cref{thm:derivative_Sp}, \Cref{cly:K_V_Sp}, \Cref{thm:ZD_VZ_Sp}, \Cref{cly:ZD_VZ_Sp}, \Cref{thm:D^2_Sp}, \Cref{thm:S2_kernels}, \Cref{prop:S1_KV}.

\begin{proof}[Proof of \Cref{thm:MSC1-MSD_iso}]
    Let $\gamma(t)=\exp_x(tv)$ where $v\in T_x{\mathbb{S}^p}$, then by the definition of MSC, we have
$\lim_{t\to0}\EE(Z(\gamma(t))-Z(x))^2=\lim_{t\to0}2K(0)-2K(t).$
As a result, $Z$ is MSC if and only if $K(t)\to K(0)$, as $t\to0$, that is, $K$ is continuous at $0$. 
    By the definition of 1-MSD, it suffices to check whether the following limit exists or not:
    \begin{align*}
    &\lim_{t\to0} \frac{\EE\left(Z(\exp_{x_0}(tv))-Z(x_0)\right)^2} {t^2} \\
    &={\lim_{t\to0}\frac{1}{t^2} \EE\left(Z(\exp_{x_0}(tv))Z(\exp_{x_0}(tv))-2Z(\exp_{x_0}(tv))Z(x_0)+Z(x_0)Z(x_0)\right),}\\
    & =  \lim_{t\to0}\frac{1}{t^2} \left(K(\exp_{x_0}(tv),\exp_{x_0}(tv))-2K(\exp_{x_0}(tv),x_0)+K(x_0,x_0)\right),\\
    & = \lim_{t\to0}\frac{1}{t^2} \left(2K(0)-2K( t )\right),\\
    & = \lim_{t\to0}\frac{2}{t^2} \left(K(0)-K(0)-K'(0)t-K''(0)\frac{t^2}{2} - O(t^3)\right),\\
    & = -\lim_{t\to0}\frac{2}{t}K'(0) -K''(0).\\
    \end{align*}
    Consequently, the limit exists if and only if $K(0)-K(t)=O(t^2)$, that is, $K'(0)=0$ and $|K''(0)|<\infty$. 
\end{proof}

\begin{proof}[Proof of \Cref{thm:derivative_Sp}]

First observe that the conclusion of \Cref{thm:derivative_Sp} is invariant under rescaling of $V$. Then observe that  holds when either $V_x=0$ or $V_{x'}=0$ as both sides of the equation are zero. As a result, we assume $\|V_x\|\neq0$ and $\|V_{x'}\|=0$ without loss of generality.

%DL: TBC.
First let, 
    \begin{align*}
        A(t)&=\exp_x(tV_x)=\cos(t\|V_x\|)x+\sin(t\|V_x\|)\frac{V_x}{\|V_x\|},\\
        D(s)&=\exp_{x'}(sV_{x'})=\cos(s)x'+\sin(s)V_{x'},\\
        \xi(s)&= \arccos(\<A(t),D(s)\>).\\
    \end{align*}
    then
    \begin{align*}
        &K_V(x,x')=\lim_{t,s\to0} \frac{1}{ts}\left\{K(\exp_x(tV_x),\exp_{x'}(sV_{x'}))-K(\exp_x(tV_x),x')-K(x,\exp_{x'}(sV_{x'}))+K(x,x')\right\}.\\
        & = \lim_{t,s\to0} \frac{1}{ts}\left\{K(\arccos(A(t),D(s)\>))-K(\arccos(\<A(t),x'\>))-K(\arccos(\<x,D(s)\>))+K(\arccos(\<x,x'\>))\right\}.
    \end{align*}
      %\textcolor{blue}{Aritra: is this $D(s)$?} 
    Then by chain rule, %we have
    \begin{align*}
    &\lim_{s\to0}\frac{1}{s}\left\{K(\arccos(A(t),D(s)\>))-K(\arccos(A(t),x'\>))\right\}\\
   % &\lim_{s\to0}\frac{1}{s}\left\{K(\arccos(\<\cos(t)x+\sin(t)V_x,\cos(s)x'+\sin(s)V_{x'}\>))\right.\\
    %&\pushright{\left.-K(\arccos(\<\cos(t)x+\sin(t)V_x,x'\>))\right\}}\\
    &=\lim_{s\to0}\frac{1}{s}\left(K(\xi(s))-K(\xi(0))\right)= K'(\xi(0))\xi'(0)\\
    & = K'(\xi(0))\frac{-\<A(t),D'(0)\>}{(1-\<A(t),D(0)\>^2)^{1/2}} = -\frac{K'(\arccos(\<A(t),x'\>))\<A(t),V_{x'}\>}{(1-\<A(t),x'\>^2)^{1/2}}.
    \end{align*}
    Similarly, we have 
    \begin{align*}
    &\lim_{s\to0}\frac{1}{s}K(\arccos(\<x,D(s)\>-K(\arccos\<x,x'\>)
 %   &=\lim_{s\to0}\frac{1}{s}K(\arccos(\<x,\cos(s)x'+\sin(s)V_{x'}\>))-K(\arccos(\<x,x'\>) \\
 %   & = K'(\arccos(\<x,x'\>))\frac{-1}{\sqrt{1-\<x,x'\>^2}}\<x,V_{x'}\>\\
     = -\frac{ K'(\arccos(\<x,x'\>))\<x,V_{x'}\>}{(1-\<x,x'\>^2)^{1/2}}.
    \end{align*}

Then we need to analyze
\begin{align*}
&-\lim_{t\to0}\frac{1}{t}\left(\frac{ K'(\arccos(\<A(t),x'\>))\<A(t),V_{x'}\>}{(1-\<A(t),x'\>^2)^{1/2}}-\frac{ K'(\arccos(\<A(0),x'\>))\<A(0),V_{x'}\>}{(1-\<A(0),x'\>^2)^{1/2}}\right)\\
&\coloneqq  -F'(0),
\end{align*}
where $F(t) = \frac{ K'(\arccos(\<A(t),x'\>))\<A(t),V_{x'}\>}{(1-\<A(t),x'\>^2)^{1/2}}$. To simplify the notation, let $B(t) = \<A(t),x'\>$ and $C(t) = \<A(t),V_{x'}\>$. We know that $A'(0) = V_x,~ B'(0) = \<V_x, x'\>,~C'(0) = \<V_x, V_{x'}\>$. Hence,
\begin{align*}
F'(t)&=\frac{dF}{dt}\bigg|_{t=0} =\frac{d}{dt}\frac{K'(\arccos(B(t)))C(t)}{(1 - B(t)^2)^{1/2}}\\
%& = \frac{K''(\arccos(B(t))) B'(t)  C(t) + K'(\arccos(B(t))) C'(t)} {\sqrt{1 - B(t)^2}} +\frac{ K'(\arccos(B(t)))  C(t)B'(t) }{\sqrt{1 - B(t)^2}^3},\\
%& = -\frac{K''(\arccos(B(t)))B'(t)C(t)}{1-B(t^2)}+\frac{K'(\arccos(B(t)))C'(t)\sqrt{1-B(t)^2}}{1-B(t^2)}\\&\hspace{4cm}+\frac{K'(\arccos(B(t)))C(t)B(t)B'(t)/\sqrt{1-B(t)^2}}{1-B(t^2)},
& = -\frac{K''(\arccos(B(t)))B'(t)C(t)}{1-B(t^2)}+\frac{K'(\arccos(B(t)))C'(t)}{(1-B(t^2))^{1/2}}{+\frac{K'(\arccos(B(t)))C(t)B(t)B'(t)}{(1-B(t^2))^{3/2}}.}
%& = \frac{[C(t) K''(\arccos(B(t))) (-1/ \sqrt{1 - B(t)^2}) B'(t) + K'(\arccos(B(t))) C'(t)] / \sqrt{1 - B(t)^2} - [C(t) *K'(\arccos(B(t)))  B'(t)  [2B(t)  B'(t) / (2  (1 - B(t)^2))]}{1 - B(t)^2}
\end{align*}
Hence, 
% \begin{align*}
% F'(0)&= \frac{K''(B(0)) B'(0)  C(0) + K'(B(0)) C'(0)} {\sqrt{1 - B(0)^2}} +\frac{ K'(B(0))  C(0)B'(0) }{\sqrt{1 - B(0)^2}^3}\\
% & =  \frac{K''(\<x,x'\>) \<V_x,x'\> \<x,V_{x'}\> + K'(\<x,x'\>) \<V_{x},V_{x'}\>} {\sqrt{1 - \<x,x'\>^2}} +\frac{ K'(\<x,x'\>)  \<x,V_{x'}\>\<V_x,x'\> }{\sqrt{1 - \<x,x'\>^2}^3},
% \end{align*}
% and the theorem follows. 
\begin{align*}
F'(0)&= -\frac{-K''(\arccos(B(0))) B'(0)  C(0)} {1 - B(0)^2} +\frac{K'(\arccos(B(0)))C'(0)}{(1-B(0)^2)^{1/2}}{+\frac{ K'(\arccos(B(0)))  C(0)B(0)B'(0) }{(1 - B(0)^2)^{3/2}}}\\
& =  -\frac{K''(\arccos(\<x,x'\>)) \<V_x,x'\> \<x,V_{x'}\> } {{1 - \<x,x'\>^2}} +\frac{K'(\arccos(\<x,x'\>))\<V_x,V_{x'}\>}{(1 - \<x,x'\>^2)^{1/2}}\\
&\hspace{5cm}+\frac{ K'(\arccos(\<x,x'\>))  \<x,V_{x'}\>\<x,x'\>\<V_x,x'\> }{(1 - \<x,x'\>^2)^{3/2}},
\end{align*}
and the theorem follows, $K_V(x,x')=-F'(0)$. 

\end{proof}

\begin{proof}[Proof of \Cref{cly:K_V_Sp}]
Let $A(t)=\exp_x(tu)$, where $u\in T_x\mathbb{S}^p$ with $\|u\|=1$, so we have
\begin{align*}
\<x,A(t)\>&=\<x,\cos(t)x+\sin(t)u\>=\cos(t)\\
\<V_x,A(t)\>&=\<V_x,\cos(t)x+\sin(t)u\>=\sin(t)\<V_x,u\>\\
\<x,V_{A(t)}\>& = \<x,V_x+t\nabla_uV+O(t^2)\> = t\<x,\nabla_uV\>+O(t^2),\\
\<V_x,V_{A(t)}\> &=\<V_x,V_x+t\nabla_uV+O(t^2)\>=\<V_x,V_x\>+t\<V_x,\nabla_uV\>+O(t^2)\\
&=\<V_x,V_x\>+t\<V_x,\nabla_uV\>+O(t^2).
\end{align*}
From \Cref{eqn:D_VZ_Sp} we have
\begin{align*}
    &K_V(x,x)=\lim_{t\to0} K_V(x,\exp_x(tu))=\lim_{t\to0} K_V(x,A(t))\\
    & = \lim_{t\to0}\left\{- \frac{K'(\arccos(\<x,A(t)\>))\<V_x,V_{A(t)}\>}{(1 - \<x,A(t)\>^2)^{1/2}}  \right.\\
    &\pushright{\left.~~~~+\left(K''(\arccos(\<x,A(t)\>)) - \frac{ K'(\arccos(\<x,A(t)\>)) \<x,A(t)\>}{(1 - \<x,A(t)\>^2)^{1/2}}\right)\frac{\<V_x,A(t)\> \<x,V_{A(t)}\> } {{1 - \<x,A(t)\>^2}}\right\}},
\end{align*}
which can be simplified using the previous observations as
    \begin{align*}
    % &K_V(x,x)=\lim_{t\to0} K_V(x,\exp_x(tu))=\lim_{t\to0} K_V(x,A(t))\\
    % & = \lim_{t\to0}\left\{- \frac{K'(\arccos(\<x,A(t)\>))\<V_x,V_{A(t)}\>}{(1 - \<x,A(t)\>^2)^{1/2}} + \left(K''(\arccos(\<x,A(t)\>))  -\frac{ K'(\arccos(\<x,A(t)\>)) \<x,A(t)\>}{(1 - \<x,A(t)\>^2)^{1/2}}\right)\times \right.\\&\hspace{5cm}\left.\frac{\<V_x,A(t)\> \<x,V_{A(t)}\> } {{1 - \<x,A(t)\>^2}}\right\}\\
    & \lim_{t\to0}\left\{-\frac{K'(t)(\<V_x,V_x\>+t\<V_x,\nabla_uV\>+O(t^2))}{(1 - \cos^2(t))^{1/2}}\right.\\
    &\pushright{\left.+\left(K''(t) -
    \frac{ K'(t) \cos(t) }{(1 -\cos^2(t))^{1/2}}\right)\frac{\sin(t)\<V_x,u\>(t\<x,\nabla_uV\>+O(t^2))} {{1 -\cos^2(t)}}\right\}}\\
    & = -\<V_x,V_x\>\lim_{t\to0}\frac{K'(t)}{t}\frac{t}{\sin(t)}+\<V_x,\nabla_uV\>\lim_{t\to0}t\;\frac{t}{\sin(t)}\;\frac{K'(t)}{t}\\
    &\pushright{~~~~+\lim_{t\to0}\left(K''(t)-\frac{K'(t)}{t}\frac{t}{\sin(t)}\right)\<V_x,u\>\lim_{t\to0}\frac{\sin(t)(t\<x,\nabla_uV\>+O(t^2))}{\sin^2(t)}}\\
    % & = -\lim_{t\to0}\frac{K'(t)}{\sin(t)}+K''(0)\<V_x,u\>\<x,\nabla_uV\>-\<V_x,u\>\<x,\nabla_uV\>\lim_{t\to0}\frac{K'(t)}{\sin(t)}\\
    % & = K''(0)\<V_x,u\>\<x,\nabla_uV\>-K''(0)-\<V_x,u\>\<x,\nabla_uV\>K''(0)\\
    & =  -\<V_x,V_x\>K''(0),%\\
    % &K_V(x,x)=\lim_{t\to0} K_V(x,\exp_x(tu))=\lim_{t\to0} K_V(x,A(t))\\
    % & = \lim_{t\to0}\left(\frac{K''(\arccos(\<x,A(t)\>))\<V_x,A(t)\> \<x,V_{A(t)}\> } {{1 - \<x,A(t)\>^2}} - \frac{K'(\arccos(\<x,A(t)\>))\<V_x,V_{A(t)}\>}{(1 - \<x,A(t)\>^2)^{1/2}} -\right.\\&\hspace{6cm}\left.\frac{ K'(\arccos(\<x,A(t)\>))  \<x,V_{A(t)}\>\<x,A(t)\>\<V_x,A(t)\> }{\sqrt{1 - \<x,A(t)\>^2}^3}\right)\\
    % & = \lim_{t\to0}\left(\frac{K''(t) \sin(t)\<V_x,u\>(t\<x,\nabla_uV\>+O(t^2))} {{1 -\cos(t)^2}} -\frac{K'(t)\<V_x,V_{A(t)}\>}{\sqrt{1 - \cos(t)^2}}-\right.\\&\hspace{6cm}\left.\frac{ K'(t)  (t\<x,\nabla_uV\>+O(t^2))\cos(t)\sin(t)\<V_x,u\> }{\sqrt{1 -\cos(t)^2}^3}\right)\\
    % & = K''(0)\<V_x,u\>\lim_{t\to0}\frac{\sin(t)(t\<x,\nabla_uV\>+O(t^2))}{\sin^2(t)}-\lim_{t\to0}\frac{K'(t)}{\sin(t)}-\<V_x,u\>\lim_{t\to0}\frac{K'(t)(t\<x,\nabla_uV\>+O(t^2))\sin(t)}{\sin^3(t)}\\
    % & = K''(0)\<V_x,u\>\<x,\nabla_uV\>-\lim_{t\to0}\frac{K'(t)}{\sin(t)}-\<V_x,u\>\<x,\nabla_uV\>\lim_{t\to0}\frac{K'(t)}{\sin(t)}\\
    % & = K''(0)\<V_x,u\>\<x,\nabla_uV\>-K''(0)-\<V_x,u\>\<x,\nabla_uV\>K''(0)\\
    % & = -K''(0),
    \end{align*}
    where the second equation from the bottom comes from the Taylor expansion of $K$ at 0: $K(t)=K(0)+K''(0)t^2/2+O(t^3)$ and the differentiability of $Z$. Note that the above calculation does not depend on the choice of $u\in T_x \mathbb{S}^p$ . 
\end{proof}

\begin{proof}[Proof of \Cref{thm:ZD_VZ_Sp}]
First observe that for any $\alpha\in\RR$,
\begin{align*}
\cov(Z(x),D_{\alpha V}Z(x'))=\alpha\cov(Z(x),D_{V}Z(x')),~\frac{K'(d)\<x,\alpha V_{x'}\>}{(1-\<x,x'\>^2)^{1/2}}=\alpha\frac{K'(d)\<x, V_{x'}\>}{(1-\<x,x'\>^2)^{1/2}},    
\end{align*}
%$\cov(Z(x),D_{\alpha V}Z(x'))=\alpha\cov(Z(x),D_{V}Z(x')),~\frac{K'(d)\<x,\alpha V_{x'}\>}{(1-\<x,x'\>^2)^{1/2}}=\alpha\frac{K'(d)\<x, V_{x'}\>}{(1-\<x,x'\>^2)^{1/2}},$
so the conclusion of \Cref{thm:ZD_VZ_Sp} is invariant under rescaling $V$. Next, observe that \Cref{thm:ZD_VZ_Sp} holds when $V_{x'}=0$. When $V_{x'}\geq 0$, we assume $\|V_{x'}\|=1$ without loss of generality and let $A(s) = \exp_{x'}(sV_{x'})$ and $d(s) = d_{\mathbb{S}^p}(x,A(s))$, where $d_{\mathbb{S}^p}(x,A(s))=\arccos(\<x,\cos(s)x'+\sin(s)V_{x'}\>)$, then
    \begin{align*}
&\cov(Z(x),D_VZ(x')) = \cov\left(Z(x),\lim_{s\to 0}\frac{Z(\exp_{x'}(sV_{x'}))-Z(x')}{s}\right)\\
& = \lim_{s\to0} \frac{1}{s}\left[K(x,\exp_{x'}(sV_{x'}))-K(x,x')\right]\\
& = \lim_{s\to0} \frac{1}{s}\left[K(d(s))-K(d(0))\right]\\
& = K'(d(0))\frac{-\<x,A'(s)\>}{(1-\<x,A(s)\>^2)^{1/2}}\bigg|_{s=0}\\
& = -\frac{K'(d)\<x,V_{x'}\>}{(1-\<x,x'\>^2)^{1/2}}. 
\end{align*}
Similarly, $\cov(D_VZ(x),Z(x'))=-\frac{K'(d)\<x',V_{x}\>}{(1-\<x',x\>^2)^{1/2}}\ne\cov(Z(x),D_VZ(x'))$.
\end{proof}

\begin{proof}[Proof of \Cref{cly:ZD_VZ_Sp}]
Let $A(t)=\exp_x(tu)$, where $u\in T_x\mathbb{S}^p$ with $\|u\|=1$, so we have
\begin{align*}
\<A(t),x\>&=\<\cos(t)x+\sin(t)u,x\>=\cos(t)\\
\<A(t),V_x\>&=\<\cos(t)x+\sin(t)u,V_x\>=\sin(t)\<u,V_x\>.
\end{align*}
Now we can simplify $K(x,A(t))$ as:
    \begin{align*}
    &K(x,A(t))=\lim_{t\to0}\cov(Z(\exp_x(tu)),D_VZ(x))= -\lim_{t\to0}\frac{K'(t)\<A(t),V_{x}\>}{(1-\<A(t),x\>^2)^{1/2}}\\
    & = -\lim_{t\to0}\frac{K'(t)\sin(t)\<u,V_x\>}{(1-\cos^2(t))^{1/2}}=-K'(0)\<u,V_x\>=0.
    \end{align*}
    where the second equation from the bottom comes from the Taylor expansion of $K$ at 0: $K(t)=K(0)+K''(0)t^2/2+O(t^3)$ and the differentiability of $Z$. Note that the above result does not depend on the choice of $u\in T_x \mathbb{S}^p$. 
\end{proof}

{  
\section{Higher-order Directional Derivative Processes}\label{sec:curv}
Turning to higher order differentiability, we require multiple vector fields and affine connections which demands fairly complex treatment that is beyond the scope of this manuscript (for more details, see \cite{petersen_riemannian_2016}, Section 2.2.2.3, p. 61 or, \cite{lee_introduction_2018}, Chapter 4, Example 4.22). We consider the important special case, where $V$ is covariantly constant along its integral curves (so that $\nabla_VV\coloneqq0$). The vector field in the trajectory application in \Cref{sec:fitrec} satisfies this. The repeated directional derivative then coincides with the Riemannian Hessian evaluated along $V$. We term this as the directional curvature process. This restriction is imposed only for the second-order theoretical development; the first-order differential process $D_VZ$ developed in the manuscript permits general smooth admissible vector fields and is unaffected by this assumption.
\begin{definition}
The directional curvature process along $V$, denoted by $D^2_VZ$ at $x\in\MM$ is defined as
\begin{equation*}
    D^2_VZ(x)\coloneqq\begin{cases}
        \lim_{t\to 0} \frac{D_VZ(\exp_x(tV(x)))-D_VZ(x)}{t}, & x\in \MM^\circ, V\in \mathfrak{X}(\MM),\\
        \lim_{t\downarrow 0} \frac{D_VZ(\gamma_{x,V}(t))-D_VZ(x)}{t}, & x\in \partial\MM, V\in \mathfrak{X}^{+}(\MM).
    \end{cases}
\end{equation*} 
\end{definition}

\begin{definition}\label{def:2-MSD}
The process $Z(x)$ is 2-MSD along $V$ at $x\in\MM^\circ$ if $\lim\limits_{t \to 0}\EE\left(D^2_VZ(\gamma(t))-D^2_VZ(x)\right)^2=0$, for any geodesic $\gamma$ with $\gamma(0)=x$ and, at $x\in\partial\MM$, if $\lim\limits_{t \downarrow 0}\EE\left(D^2_VZ(\gamma(t))-D^2_VZ(x)\right)^2=0$ for any admissible curve $\gamma$ with $\gamma(0)=x$. $Z$ is said to be 2-MSD if it is 2-MSD at any $x\in\MM$.   
\end{definition}
}
While we leverage the derivative process, $D_VZ$, to define a 2-MSD process, we could also define it using the parent process $Z$. Proposition~\ref{prop:2-MSD} at the end of this subsection discusses the consequences of adopting this route. Turning to our kernels, the next result is an extension of \autoref{thm:1MSD} showing the relationship between the smoothness parameter and the dimension of the manifold when $Z$ is 2-MSD.

%The following theorem studies the relation between the smoothness $\nu$ in Matern and 2-MSD. 
\begin{theorem}\label{thm:2MSD}
$Z$ is 2-MSD if $\nu>\frac{p+3}{2}$; RBF is 2-MSD. 
\end{theorem}

The curvature process captures the rates of change in the derivative process over $\MM$. In an Euclidean setting they manifest as Hessians \citep[see, e.g.][]{halder2024bayesian}. Directional curvature is of interest when monitoring such a change along a direction. While in the Euclidean setting such a course is offered through familiar bi-linear forms involving the Hessian and direction vectors, an analogous formulation for $\MM$ is nuanced. The following results develop the required machinery. We first define the directional curvature process.

Note that $D^2_VZ$ is well-defined with regard to Definition \ref{def:2-MSD} and the geodesic remains $\gamma(t) = \exp_x(tV(x))$. 
%To simplify the notation, we denote $Z_{t,U,V}(x)\coloneqq \frac{D_VZ(\exp_x(tU_x))-D_VZ(x)}{t}$. 
The following lemmas derive the covariance function for the curvature process, $D^2_VZ$. These results eventually enable inference for the joint process, $(Z(x),W(x)\T)\T$, where $W(x)=(D_VZ(x), D^2_VZ(x))\T$. We begin with covariance $K_{V,V}(x,x')=\cov(D^2_VZ(x), D^2_VZ(x'))$ in the next lemma.
\begin{lemma} \label{lem:D_UVZ_cov}
 If $Z$ is 2-MSD, then $D^2_VZ$ is a valid GP on $\MM$ with the mean function $\nabla^2\mu(x)(V_x,V_x)$ and the covariance function
 \begin{equation}\label{eqn:D_VZD_VZ}
     K_{V,V}(x,x') = (\nabla_{12} K_V(x,x'))(V_x,V_{x'})=\nabla_{1212}K(x,x')(V_x,V_{x'},V_{x},V_{x'}),
 \end{equation}
  where $\nabla_{1212}$ represents the partial gradient of $K$, a function on the product manifold $\MM\times\MM$, with respect to the first and second coordinates twice. The term $(V_x,V_{x'},V_x,V_{x'})\in \mathcal{T}^2_0(\MM\times\MM)$ results from a product of two 2-0 tensors on $\MM$. %In case, $x=x'$... \ah{Can we make a comment here about $K_{U,V}(x,x)?$ Does the $\nabla_{1212}$ notation simplify then to $\nabla^4$, when $x=x'$? Since, that is what we really need this result for.}
\end{lemma}
We extend $\left(Z(x), D_V Z(x)\right)\T$ to include the curvature process and elaborate on the consequences of $Z(x)\sim GP(0,K)$. The developments % are analogous to those found in 
resemble \Cref{eq:gp-grad}. %We assume $Z$ is 1-MSD.
\begin{lemma}\label{lem:Z_D_UVZ;D_VZD_UVZ}
The joint distributions are given by
\begin{align}
\cov(Z(x),D^2_VZ(x')) &= \nabla_{22} K(x,x')(V_{x'},V_{x'}),\label{eqn:ZD_UVZ}\\
\cov(D_VZ(x),D^2_VZ(x'))&= \nabla_{122} K(x,x')(V_{x},V_{x'},V_{x'})\label{eqn:D_VZD_UVZ}.
\end{align}
\end{lemma}

It is possible, $\{\cov(Z,D^2_VZ)\}\ne\{\cov(D^2_VZ,Z)\}$ and $\{\cov(D_VZ,D^2_VZ)\}\ne \{\cov(D^2_VZ,D_VZ)\}$. Extending the discussion following \Cref{lem:ZD_VZ}, the above lemmas now apply to the joint differential process $W(x)\coloneqq(D_VZ(x),D^2_VZ(x))\T$ to provide
\begin{equation}\label{eq:gp-grad-curv}
\begin{split}    
&\qquad\qquad \begin{pmatrix}Z(x)\\W(x)\end{pmatrix}\sim GP\left(\mu_W, \Sigma_W\right),\qquad
    \mu_W = \left[\begin{array}{c}
        \mu(x)  \\
        \<\nabla \mu(x),V_{x}\>\\
        \nabla^2\mu(x)(V_{x},V_x)
    \end{array}\right],\\
    &\Sigma_W = \left[\begin{array}{ccc}
        K(x,x') & \nabla_2K(x,x')(V_{x'}) & \nabla_{22} K(x,x')(V_{x'},V_{x'})\\
        \nabla_2K(x,x')(V_{x'}) & K_V(x,x') & \nabla_{122} K(x,x')(V_{x},V_{x'},V_{x'})\\
        \nabla_{22} K(x,x')(V_{x'},V_{x'})& \nabla_{122} K(x,x')(V_{x},V_{x'},V_{x'}) & K_{V,V}(x,x')
    \end{array}\right].
\end{split}
\end{equation}
%The cross-covariance matrix, $\Sigma_W$ lacks symmetry. 
Similarly, $P(W(x)\mid Z(x))$, is obtained following the discussion below \Cref{eq:gp-grad}. Thus, the process $(Z(x),W(x)\T)\T$ comprising the parent and differential processes has the matrix-valued cross-covariance function 
\begin{equation}\label{eqn:C_W}
   C_W(x,x') = \begin{bmatrix}
      C_{Z, Z}(x,x') & C_{Z, D_VZ}(x,x') & C_{Z, D^2_VZ}(x,x') \\
      C_{D_VZ, Z}(x,x') & C_{D_VZ, D_VZ}(x,x') & C_{D_VZ, D^2_VZ}(x,x') \\
      C_{D^2_VZ, Z}(x,x') & C_{D^2_VZ, D_VZ}(x,x') & C_{D^2_VZ, D^2_VZ}(x,x') 
    \end{bmatrix}\;,
\end{equation}
where $C_{F_1,F_2}(x,x') = \mbox{cov}(F_1(x), F_2(x'))$ for random variables $F_1(x)$ and $F_2(x')$. If the parent process is a GP, then the joint process above is also a valid GP. The above results, when applied to our choices for kernels, drive the following expressions. % housed within the next corollary:
\begin{corollary}\label{cly:D_UVZMat}
When the covariance function of $Z$ is spectral Whittle-Mat\'ern we have
\begin{align}
    K_{V,V}(x,x')&=\frac{\sigma^2}{C_{\nu,\alpha}}\sum_{l=0}^\infty (\alpha^2+\lambda_l)^{-\nu-\frac{p}{2}}\nabla f_l^2(V_x,V_x)\nabla^2 f_l(V_{x'},V_{x'}),\\
    \cov(Z(x),D^2_VZ(x'))&= \frac{\sigma^2}{C_{\nu,\alpha}}\sum_{l=0}^\infty (\alpha^2+\lambda_l)^{-\nu-\frac{p}{2}}f_l(x)\nabla^2 f_l(V_{x'},V_{x'}),\\
    \cov(D_VZ(x),D^2_VZ(x'))&= \frac{\sigma^2}{C_{\nu,\alpha}}\sum_{l=0}^\infty (\alpha^2+\lambda_l)^{-\nu-\frac{p}{2}}\nabla f_l(V_x)\nabla^2 f_l(V_{x'},V_{x'}).
\end{align}
In case of the RBF,
\begin{align}
    K_{V,V}(x,x')&=\frac{\sigma^2}{C_{\nu,\alpha}}\sum_{l=0}^\infty e^{-\frac{\lambda_l}{2\alpha^2}}\nabla f_l^2(V_x,V_x)\nabla^2 f_l(V_{x'},V_{x'}),\\
    \cov(Z(x),D^2_VZ(x'))&= \frac{\sigma^2}{C_{\nu,\alpha}}\sum_{l=0}^\infty e^{-\frac{\lambda_l}{2\alpha^2}}f_l(x)\nabla^2 f_l(V_{x'},V_{x'}),\\
    \cov(D_VZ(x),D^2_VZ(x'))&= \frac{\sigma^2}{C_{\nu,\alpha}}\sum_{l=0}^\infty e^{-\frac{\lambda_l}{2\alpha^2}}\nabla f_l(V_x)\nabla^2 f_l(V_{x'},V_{x'}).
\end{align}
\end{corollary}

\begin{theorem}\label{thm:joint}
In \Cref{eqn:C_W}, $C_{Z, Z}(x,x')=K(x,x')$, $C_{Z, D_VZ}(x,x')$ is given by \Cref{lem:ZD_VZ}, $C_{D_VZ, D_VZ}(x,x')$ by \Cref{lem:D_VZ_cov}, $C_{Z, D^2_VZ}(x,x')$ and $C_{D_VZ, D^2_VZ}(x,x')$ by \Cref{lem:Z_D_UVZ;D_VZD_UVZ} and \Cref{lem:D_UVZ_cov} gives, $C_{D^2_VZ, D^2_VZ}(x,x')$. In particular, for the spectral Whittle-Mat\'ern and RBF kernels the matrix-valued cross-covariance functions are $\frac{\sigma^2}{C_{\nu,\alpha}}\sum_{l=0}^\infty(\alpha^2+\lambda_l)^{-\nu-\frac{p}{2}} H(x,x')$ and $\frac{\sigma^2}{C_{\nu,\alpha}}\sum_{l=0}^\infty e^{-\frac{\lambda_l}{2\alpha^2}} H(x,x')$, respectively, where 
\begin{equation*}
    H(x,x') = \begin{bmatrix}f_l(x)f_l(x') & f_l(x)\nabla f_l(V_{x'})&f_l(x)\nabla^2 f_l(V_{x'},V_{x'})\\
    \nabla f_l(V_{x'})f_l(x)&\nabla f_l(V_x)\nabla f_l(V_{x'})&\nabla f_l(V_x)\nabla^2 f_l(V_{x'},V_{x'})\\
    \nabla^2 f_l(V_{x'},V_{x'})f_l(x)&\nabla^2 f_l(V_{x'},V_{x'})\nabla f_l(V_x)&\nabla^2 f_l(V_{x},V_x)\nabla^2 f_l(V_{x'},V_{x'})
    \end{bmatrix}.
\end{equation*}
\end{theorem}

We conclude with an analogue of \Cref{prop:MSD_iso} for isotropic kernels. % for the curvature process. % To connect 2-MSD with the derivative of $K$ at zero, we need more assumptions.
\begin{proposition}\label{prop:2-MSD}
The process is $Z$ is 2-MSD if and only if $K'(0)=K^{(3)}(0)=0$ and $K^{(4)}<\infty$. 
\end{proposition}

\section{Higher-order derivatives on the Sphere}
\begin{theorem}\label{thm:D^2_Sp}
If $Z$ is isotropic in $\mathbb{S}^p$, then $Z$ is 2-MSD if and only if $K(t)=K(0) + K''(0)t^2/2 + K^{(4)}(0)t^4/4!+O(t^5)$, that is, $K'(0)=K^{(3)}(0)=0$, $|K''(0)|<\infty$, and $|K^{(4)}(0)|<\infty$.
\end{theorem}

\Cref{thm:D^2_Sp} provides the necessary and sufficient conditions for $Z$ to be 2-MSD. Considering inference for $W(x)$ and following the strategy in \Cref{thm:derivative_Sp} while observing that $x'\to x$, $d\to 0$ and $\left(K''(d) - K'(d)\frac{\<x,x'\>}{(1 - \<x,x'\>^2)^{1/2}}\right) \to 0$, the analytic form of the cross-covariance involving the curvature process is theoretically computable, but tedious. \Cref{tab:S2_Kernels-2} considers commonly used kernels considered by \cite{guinness2016isotropic} for 2-MSD.

\begin{table}[htb]
\centering
\caption{List of covariance functions on $\mathbb{S}^2$ and 2-MSD.}\label{tab:S2_Kernels-2}
\resizebox{\linewidth}{!}{
    \begin{tabular}{c|c|c|c|c|c}
    \hline\hline
      Name & Expression & Parameter Values & MSC&1-MSD & 2-MSD\\ \hline
    Chordal Mat\'ern &$\left(\alpha^2\sin\left(\frac{t}{2}\right)\right)^\nu K_\nu(\alpha^2\sin(\frac{t}{2}))$ &$\alpha,\nu>0$ &Yes& $\nu>1$ &$\nu>2$\\ \hline
    Circular Mat\'ern &$\sum\limits_{l=-\infty}^\infty (\alpha^2+l^2)^{-\nu-\frac{1}{2}}\exp(ilt)$&$\alpha,\nu>0$ &Yes&No &No\\ \hline
    Legendre-Mat\'ern &$\sum\limits_{l=0}^\infty (\alpha^2+l^2)^{-\nu-\frac{1}{2}}P_l(\cos t)$&$\alpha,\nu>0$ & $\nu> \frac{1}{2}$ & $\nu> \frac{3}{2}$&$\nu>\frac{5}{2}$\\ \hline
    Truncated Legendre-Mat\'ern &$\sum\limits_{l=0}^T (\alpha^2+l^2)^{-\nu-\frac{1}{2}}P_l(\cos t)$&$\alpha,\nu>0$ &Yes & Yes&Yes\\ \hline
    Bernoulli &$1+\alpha+\sum\limits_{l\ne 0}|l|^{-2n}\exp(ilt )$ &$\alpha>0$, $n\in\NN$&Yes&No &No\\ \hline
    Powered Exponential &$\exp(-(\alpha t )^\nu)$ &$\alpha>0$, $\nu\in(0,1]$ &Yes&No &No\\ \hline
    Generalized Cauchy & $\left(1+(\alpha t )^\nu\right)^{-\tau/\nu}$&$\alpha,\tau>0$, $\nu\in(0,1]$ &Yes&No &No\\ \hline
    Multiquadric &$(1-\tau)^{2\alpha}/\left(1+\tau^2-2\tau\cos  t \right)^\alpha$ &$\alpha>0$, $\tau\in(0,1)$ &Yes&Yes &Yes\\ \hline
    Sine Power &$1-\left(\sin(\frac{ t }{2})\right)^\nu$&$\nu\in(0,2)$&Yes &No &No\\ \hline
    Spherical &$(1+\frac{\alpha t }{2})(1-\alpha t )_+^2$&$\alpha>0$&Yes &No &No\\ \hline
    Askey & $(1-\alpha t )_+^\tau$&$\alpha>0$, $\tau\geq 2$&Yes &No &No\\ \hline
    $C^2$--Wendland &$(1+\tau\alpha t )(1-\alpha t )_+^\tau$ &$\alpha\geq \frac{1}{\pi}$, $\tau\geq 4$&Yes  &Yes & No\\ \hline
    $C^4$--Wendland &$\left(1+\tau\alpha t +\frac{\tau^2-1}{3}(\alpha t )^2\right)(1-\alpha t )_+^\tau$ &$\alpha\geq \frac{1}{\pi}$, $\tau\geq 6$&Yes  &Yes &Yes\\ \hline
\end{tabular}
}
\end{table}

\section{Proofs for Appendix \ref{sec:curv}}\label{apdx:sec:curv}
In this section, we present proofs of \Cref{thm:2MSD}, \Cref{lem:D_UVZ_cov}, \Cref{lem:Z_D_UVZ;D_VZD_UVZ}, \Cref{cly:D_UVZMat}, \Cref{thm:joint}, and \Cref{prop:2-MSD}.

\begin{proof}[Proof of \Cref{thm:2MSD}]
Let {  $x\in\MM^\circ$ and $\gamma:(-\delta,\delta)\to \MM$ be a geodesic with $\gamma(0)=x$, while for $x\in \partial\MM$ and $\gamma:[0,\delta)\to \MM$ be an admissible smooth curve with $\gamma(0)=x$. In what follows, limits $t\to0$ are understood as two-sided for $x\in M^\circ$ and as $t\downarrow0$ for $x\in\partial M$.} Following proofs for \Cref{thm:MSC,thm:1MSD} we have,
\begin{align*}
&\EE\left(\frac{D_VZ(\gamma(t))-D_VZ(x)}{t}\right)^2   = \frac{1}{t^2} \left[K_V(\gamma(t),\gamma(t))-2K_V(\gamma(t),x)+K_V(x,x)\right]\\
& = \frac{\sigma^2}{C_{\nu,\alpha}}\sum_{l=0}^\infty (\alpha^2+\lambda_l)^{-\nu-\frac{p}{2}}\frac{\left(\nabla f_l(V_{\gamma(t)})-\nabla f_l(V_x)\right)^2}{t^2}\eqqcolon \frac{\sigma^2}{C_{\nu,\alpha}}\sum_{l=0}^\infty\zeta_l(t).
\end{align*}
To switch the limit, $\lim\limits_{t\to0}$ and the sum, $\sum_{l=0}^\infty$, we need the series, $\sum_{l=0}^\infty\zeta_l(t)$ to be uniformly convergent. Observe that, $\lim\limits_{t\to0} \frac{|\nabla f_l(V_{\gamma(t)})-\nabla f_l(V_{\gamma(0)})|}{t}\leq C\|\nabla^2 f_l\|_\infty$, where $C>0$ is a constant. By \cite{cheng2024hessian}, we have
$$\|\nabla^2 f_l\|_\infty \leq C \lambda_l\|f_l\|_\infty \leq C \lambda_l\lambda_l^{\frac{p-1}{4}}\|f_l\|_2=C\lambda_l^{\frac{p+3}{4}}.$$

Let $c_l \coloneqq(\alpha^2+\lambda_l)^{-\nu-\frac{p}{2}}\lambda_l^{\frac{p+3}{2}}$. Then $|\zeta_l(t)|\leq 4C^2c_l$. Similarly, when $\nu>\frac{p+3}{2}$
$$\sum_{l=0}^\infty a_l\asymp \sum_{l=0}^\infty (\alpha^2+l^{2/p})^{-\nu-p/2}l^{\frac{p+3}{p}}\asymp \sum_{l=0}^\infty l^{-\frac{2\nu-3}{p}}<\infty.$$ 
As a result, when $\nu>\frac{p+3}{2}$,
\begin{align*}
&\lim\limits_{t \to 0}\EE\left(\frac{D_VZ(\gamma(t))-D_VZ(x)}{t}\right)^2 \\
& = \sum_{l=0}^\infty (\alpha^2+\lambda_l)^{-\nu-p/2}\lim\limits_{t\to0}\frac{\left(\nabla f_l(V_{\gamma(t)})-\nabla f_l(V_{\gamma(0)})\right)^2}{t^2}\\
& = \sum_{l=0}^\infty (\alpha^2+\lambda_l)^{-\nu-p/2}\nabla^2 f_l(V_x,\gamma'(x)). 
\end{align*}

The proof for RBF is obtained by replacing $(\alpha^2+\lambda_l)^{-\nu-\frac{p}{2}}$ by $e^{-\frac{\lambda_l}{2\alpha^2}}$.
\end{proof}

\begin{proof}[Proof of \Cref{lem:D_UVZ_cov}]
For the mean function, we have
\begin{align*}
\EE(D^2_VZ(x)) & =\EE\left(\lim_{t\to0}\frac{D_VZ(\exp_x(tV(x)))-D_VZ(x)}{t}\right) \\
& = \lim_{t\to0} \frac{\<\nabla \mu(\exp_x(tV(x))),V_{\exp_x(tV(x))}\>-\<\nabla \mu(x),V_x\>}{t}\\
& = \nabla^2 \mu(x)(V_x,V_x).
\end{align*}
For the covariance of $D^2_VZ$, by definition,
\begin{align*}
&K_{V,V}(x,x')=\cov(D^2_VZ(x),D^2_VZ(x'))\\
& = \cov\left(\lim_{t\to 0}\frac{D_V(\exp_x(tV_x)-D_VZ(x)}{t},\lim_{s\to 0}\frac{D_VZ(\exp_{x'}(sV_{x'}))-D_VZ(x')}{s}\right)\\
& = \lim_{t,s\to0} \frac{1}{ts}[K_V(\exp_x(tV_x),\exp_{x'}(sV_{x'}))-K_V(\exp_x(tV_x),x')-K_V(x,\exp_{x'}(sV_{x'}))+K_V(x,x')]\\
& = \lim_{t\to 0}\frac{1}{t}\left[\<\nabla_2 K_V(\exp_x(tV_x),x'),V_{x'}\>)-\<\nabla_2 K_V(x,x'),V_{x'}\>\right]\\
&= (\nabla_{12} K_V(x,x'))(V_x,V_{x'}).
%& = \lim_{s\to 0}\left[\<\nabla_2 K_V(x,\exp_{x'}(sU_{x'})),V_x\>)-\<\nabla_2 K_V(x,x'),V_x\>\right]= (\nabla_{12} K_V(x,x'))(V_x,V_{x'}).
\end{align*}
\end{proof}

\begin{proof}[Proof of \Cref{lem:Z_D_UVZ;D_VZD_UVZ}]
We first calculate $\cov(Z(x),D^2_VZ(x')$. By definition,
\begin{align*}
&\cov(Z(x),D^2_VZ(x')) = \cov\left(Z(x),\lim_{s\to 0}\frac{D_VZ(\exp_{x'}(sV_{x'}))-D_VZ(x')}{s}\right)\\
& = \lim_{s\to0} \frac{1}{s}\left[\<\nabla_2 K(x,\exp_{x'}(sU_{x'})),V_{\exp_{x'}(sU_{x'}))}\>-\<\nabla_2K(x,x'),V_{x'}\>\right]\\
& = \nabla_{22} K(x,x')(V_{x'},V_{x'}).
\end{align*}
Then we calculate $\cov(D_VZ(x),D^2_VZ(x'))$.
\begin{align*}
&\cov(D_VZ(x),D^2_VZ(x')) = \cov\left(D_VZ(x),\lim_{s\to 0}\frac{D_VZ(\exp_{x'}(sV_{x'}))-D_VZ(x')}{s}\right)\\
& = \lim_{s\to0} \frac{1}{s}\left[(\nabla_{12} K(x,\exp_{x'}(sV_{x'}))(V_x,V_{\exp_{x'}(sV_{x'})})-(\nabla_{12}K(x,x'))(V_x,V_{x'})\right]\\
& = \nabla_{122} K(x,x')(V_{x},V_{x'},V_{x'}).
\end{align*}
\end{proof}

\begin{proof}[Proof of \Cref{cly:D_UVZMat}]
We consider the spectral Whittle-Mat\'ern kernel. The proof for RBF follows by replacing $(\alpha^2+\lambda_l)^{-\nu-\frac{p}{2}}$ by $e^{-\frac{\lambda_l^2}{2\alpha^2}}$.
    For $K_{V,V}$, we plug in the series representation of $K$ into \Cref{eqn:D_VZD_VZ}:
    \begin{align*}
        &K_{V,V}(x,x')=\lim_{t,s\to0} \frac{1}{ts}\left[K_V(\exp_x(tV_x),\exp_{x'}(sV_{x'}))-K_V(\exp_x(tV_x),x')-K_V(x,\exp_{x'}(sV_{x'}))+K_V(x,x')\right]\\
& = \frac{\sigma^2}{C_{\nu,\alpha}}\lim_{t,s\to0} \frac{1}{ts}\sum_{l=0}^\infty (\alpha^2+\lambda_l)^{-\nu-\frac{p}{2}}\left[\nabla f_l(V_{\exp_x(tV_x)})\nabla f_l(V_{\exp_{x'}(sV_{x'})})\right.\\
&\pushright{\left.-\nabla f_l(V_{\exp_x(tV_x)})\nabla f_l(V_{x'})-\nabla f_l(V_x)\nabla f_l(V_{\exp_{x'}(sV_{x'})})+\nabla f_l(V_x)\nabla f_l(V_{x'})\right]}\\
& = \frac{\sigma^2}{C_{\nu,\alpha}}\lim_{t\to0} \frac{1}{t}\sum_{l=0}^\infty (\alpha^2+\lambda_l)^{-\nu-\frac{p}{2}}\left[\nabla f_l(V_{\exp_x(tV_x)})\nabla^2 f_l(V_{x'},V_{x'})-\nabla f_{l}(V_x)\nabla^2 f_l(V_{x'},V_{x'})\right]\\
& = \frac{\sigma^2}{C_{\nu,\alpha}}\sum_{l=0}^\infty (\alpha^2+\lambda_l)^{-\nu-\frac{p}{2}}\nabla f_l^2(V_x,V_x)\nabla^2 f_l(V_{x'},V_{x'}).
    \end{align*}

    For $\cov(Z,D_VZ)$, we plug in the series representation of $K$ into the \Cref{eqn:ZD_UVZ}:
    \begin{align*}
        &\cov(Z(x),D^2_VZ(x'))=\lim_{s\to0} \frac{1}{s}\left[\cov(Z(x),D_VZ(\exp_{x'}(sV_{x'})))-\cov(Z(x),D_VZ(x'))\right]\\
        & = \frac{\sigma^2}{C_{\nu,\alpha}}\lim_{s\to0} \frac{1}{s}\sum_{l=0}^\infty (\alpha^2+\lambda_l)^{-\nu-\frac{p}{2}}f_l(x)\left[\nabla f_l(V_{\exp_{x'}(sV_{x'})})-\nabla f_l(V_{x'})\right]\\
        & =  \frac{\sigma^2}{C_{\nu,\alpha}}\sum_{l=0}^\infty (\alpha^2+\lambda_l)^{-\nu-\frac{p}{2}}f_l(x)\nabla^2 f_l(V_{x'},V_{x'}).
    \end{align*}

        For $\cov(D_VZ,D^2_VZ)$, we plug in the series representation of $K$ into \Cref{eqn:D_VZD_UVZ}:
    \begin{align*}
        &\cov(D_VZ(x),D^2_VZ(x'))=\lim_{s\to0} \frac{1}{s}\left[\cov(D_VZ(x),D_VZ(\exp_{x'}(sV_{x'})))-\cov(D_VZ(x),D_VZ(x'))\right]\\
        & = \frac{\sigma^2}{C_{\nu,\alpha}}\lim_{s\to0} \frac{1}{s}\sum_{l=0}^\infty (\alpha^2+\lambda_l)^{-\nu-\frac{p}{2}}\nabla f_l(V_x)\left[\nabla f_l(V_{\exp_{x'}(sV_{x'})})-\nabla f_l(V_{x'})\right]\\
        & =  \frac{\sigma^2}{C_{\nu,\alpha}}\sum_{l=0}^\infty (\alpha^2+\lambda_l)^{-\nu-\frac{p}{2}}\nabla f_l(x)\nabla^2 f_l(V_{x'},V_{x'}).
    \end{align*}
\end{proof}

\begin{proof}[Proof of \Cref{thm:joint}]
The proof combines the proofs for \Cref{lem:D_VZ_cov}, \Cref{lem:ZD_VZ}, \Cref{lem:D_UVZ_cov}, \Cref{cly:ZD_VZMat}, and \Cref{cly:D_UVZMat}.
\end{proof}

\begin{proof}[Proof of \Cref{prop:2-MSD}]
We prove by the definition of 2-MSD. 
\begin{align*}
&\lim_{t\to0} \frac{D_VZ(tV_x)-D_VZ(x)}{t}\\
&=\lim_{t\to0} \frac{1}{t}\left(\lim_{s\to0} \frac{Z(\exp_{\exp_x(tV_x)}(sV_{\exp_x(tV_x)}))-Z(\exp_x(tV_x))}{s}-\lim_{s\to0} \frac{Z(\exp_{x}(sV_x))-Z(x)}{s}\right)\\
& = \lim_{t,s\to0}\frac{1}{ts}\left(Z(\exp_{\exp_x(tV_x)}(sV_{\exp_x(tV_x)}))-Z(\exp_x(tV_x))-Z(\exp_{x}(sV_x))+Z(x)\right)\\
& \eqqcolon \lim_{t,s\to0}\frac{1}{ts}(a-b-c+d).
\end{align*}
Observe that
\begin{align*}
&\EE\left(a-b-c+d\right)^2\\
& = 4K(0)-2K(a,b)-2K(a,c)+2K(a,d)+2K(b,c)-2K(b,d)-2K(c,d)\\
& = 4K(0) -2K(s)-2K(a,c)+2K(a,d)+2K(b,c)-2K(t)-2K(s).
\end{align*}

We have
\begin{align*}
&\lim_{t,s\to0} \EE\left(\frac{a-b-c+d}{ts}\right)^2\\
& = \lim_{t,s\to0}\frac{1}{t^2s^2}\left( 4K(0)-2K(a,b)-2K(a,c)+2K(a,d)+2K(b,c)-2K(b,d)-2K(c,d)\right)\\
& = \lim_{t,s\to0}\frac{1}{t^2s^2}\left( 4K(0) -2K(s)-2K(a,c)+2K(a,d)+2K(b,c)-2K(t)-2K(s)\right)\\
& =  \lim_{t}\frac{1}{t^2s^2}\left( 4K(0) -2K(t)-2K(t)+2K(2t)+2K(|t-s|)-2K(t)-2K(t)\right)\\
&=\lim_{t}\frac{1}{t^4}\left(6K(0) -8K(t)+2K(2t)\right)\\
& = \lim_{t}\frac{2}{t^4}\left(3K(0) -4(K(0)+\frac{1}{2}t^2K''(0)+\frac{1}{6}t^3K^{(3)}(0)+O(t^4))+K(0)\right.\\&\pushright{\left.+\frac{1}{2}(2t)^2K''(0)+\frac{1}{6}(2t)^3K^{(3)}(0)+O(t^4)\right)}\\
& = \lim_{t}\frac{2}{t^4}\left(-2t^2K''(0)+2t^2K''(0)-\frac{2}{3}t^3K^{(3)}(0)+\frac{4}{3}t^3K^{(3)}(0)+O(t^4)\right),
\end{align*}
which is finite if and only if
$$K'(0)=0,K''(0)<\infty, K^{(3)}(0)=0, K^{(4)}(0)<\infty.$$
\end{proof}

\begin{proof}[Proof of \Cref{thm:D^2_Sp}]
Let $A(t)=\exp_x(tu)$, where $u\in T_x\mathbb{S}^p$ with $\|u\|=1$, then when $t\approx 0$,

\begin{align*}
K_V(x,A(t))&=-\frac{K'(t)\<V_x,V_{A(t)}\>}{(1 - \<x,A(t)\>^2)^{1/2}}{+\left(K''(t) -\frac{ K'(t)  \<x,A(t)\> }{(1 - \<x,A(t)\>^2)^{1/2}}\right) \frac{\<V_x,A(t)\> \<x,V_{A(t)}\> } {{1 - \<x,A(t)\>^2}}}\\
& = -\frac{K'(t)\<V_x,V_{A(t)}\>}{\sin(t)}{+\left(K''(t) -\frac{ K'(t)  \<x,A(t)\> }{(1 - \<x,A(t)\>^2)^{1/2}}\right)\<V_x,u\>\frac{\sin(t)(t\<x,\nabla_uV\>+O(t^2))}{\sin^2(t)}}.\\
% &K_V(x,A(t))=\frac{K''(t) \<V_x,A(t)\> \<x,V_{A(t)}\> } {{1 - \<x,A(t)\>^2}} -\frac{K'(t)\<V_x,V_{A(t)}\>}{\sqrt{1 - \<x,A(t)\>^2}}-\frac{ K'(t)  \<x,V_{A(t)}\>\<x,A(t)\>\<V_x,A(t)\> }{\sqrt{1 - \<x,A(t)\>^2}^3}\\
% & = K''(t)\<V_x,u\>\frac{\sin(t)(t\<x,\nabla_uV\>+O(t^2))}{\sin^2(t)}-\frac{K'(t)\<V_x,V_{A(t)}\>}{\sin(t)}-\<V_x,u\>\<x,A(t)\>\frac{K'(t)(t\<x,\nabla_uV\>+O(t^2))\sin(t)}{\sin^3(t)}
\end{align*}
Observe that similar to previous calculations,
\begin{equation*}
    \begin{split}
        &\<V_x,V_{A(t)}\>=\<V_x,V_x\>+t\<V_x,\nabla_uV\>+O(t^2),\\
        &\<V_{A(t)},V_{A(t)}\>=\<V_x,V_x\>+2t\<V_x,\nabla_uV\>+t^2\<\nabla_uV,\nabla_uV\>+O(t^2).
    \end{split}
\end{equation*}
Using \Cref{cly:K_V_Sp} we get the following, $K_V(A(t),A(t))=-\<V_{A(t)},V_{A(t)}\>K''(0)=-(\<V_x,V_x\>+2t\<\nabla_uV,V_x\>+t^2\<\nabla_uV,\nabla_uV\>+O(t^2))K''(0)$, and $K_V(x,x)=-\<V_x,V_x\>K''(0)$. Hence,
\begin{align*}
&\lim_{t\to0}\EE\left(\frac{D_VZ(A(t))-D_VZ(x)}{t}\right)^2 = \frac{1}{t^2} \left[K_V(A(t),A(t))-2K_V(x,A(t))+K_V(x,x)\right]\\
& =\lim_{t\to0}\frac{1}{t^2} \bigg[-(2\<V_x,V_x\>+2t\<\nabla_uV,V_x\>+t^2\<\nabla_uV,\nabla_uV\>+O(t^2))K''(0)\\
&~~~\left.-2\left\{-\frac{K'(t)\<V_x,V_{A(t)}\>}{\sin(t)}+\left(K''(t)-\frac{K'(t)\<x,A(t)\>}{\sin(t)}\right)\<V_x,u\>\frac{\sin(t)(t\<x,\nabla_uV\>+O(t^2))}{\sin^2(t)}\right\}\right].
% & =\lim_{t\to0}\frac{1}{t^2} \left\{-2K''(0)-2\left(K''(t)\<V_x,u\>\frac{\sin(t)(t\<x,\nabla_uV\>+O(t^2))}{\sin^2(t)}-\right.\right.\\&\hspace{5cm}\left.\left.\frac{K'(t)\<V_x,V_{A(t)}\>}{\sin(t)}-\<V_x,u\>\<x,A()\>\frac{K'(t)(t\<x,\nabla_uV\>+O(t^2))\sin(t)}{\sin^3(t)}\right)\right\}\\
% & = \lim_{t\to0}-\frac{2}{t^2}\left(K''(0)+K''(0)\<V_x,u\>\frac{t\<x,\nabla_uV\>+O(t^2)}{\sin(t)}-\frac{K'(t)}{\sin(t)}-\<V_x,u\>\frac{K'(t)(t\<x,\nabla_uV\>+O(t^2))}{\sin^2(t)}\right)\\
% & = \lim_{t\to0}-\frac{2}{t^2}\left\{K''(0)+K''(0)\<V_x,u\>\<x,\nabla_uV\>-\frac{1}{\sin(t)}\left(K''(0)t+\frac{1}{2}K^{(3)}(0)t^2+\frac{K^{(4)}(0)}{6}t^3+o(t^3)\right)-\right.\\&\hspace{5cm}\left.\<V_x,u\>\frac{1}{\sin(t)}\left(K''(0)t+\frac{1}{2}K^{(3)}(0)t^2+\frac{K^{(4)}(0)}{6}t^3+o(t^3)\right)\<x,\nabla_uV\>\right\}\\
% & = \lim_{t\to0}-\frac{2}{t^2}\left\{K''(0)+K''(0)\<V_x,u\>\<x,\nabla_uV\>-K''(0)-\frac{1}{2}K^{(3)}(0)t-\frac{K^{(4)}(0)}{6}t^2-o(t^2)-\right.\\&\hspace{5cm}\left.\<V_x,u\>\<x,\nabla_uV\>\left(K''(0)+\frac{1}{2}K^{(3)}(0)t+\frac{K^{(4)}(0)}{6}t^2+o(t^2)\right)\right\}\\
% %&=\lim_{t\to0}\frac{-2}{t^2}\left(-\frac{1}{2}K^{(3)}(0)t-\frac{1}{2}K^{(3)}(0)t+(1-\<V_x,u\>\<x,\nabla_uV\>)\frac{K^{(4)}(0)}{6}t^2+o(t^2)\right),\\
% &=\lim_{t\to0}-\frac{2}{t^2}(1+\<V_x,u\>\<x,\nabla_uV\>)\left(\frac{1}{2}K^{(3)}(0)t+\frac{K^{(4)}(0)}{6}t^2+o(t^2)\right).
\end{align*}
We deal with the terms inside the limit individually. The first term can be re-written as
\begin{equation*}
   -K''(0)\<\nabla_uV,\nabla_uV\> - K''(0) - K''(0)\frac{2}{t}\<\nabla_uV,V_x\>- K''(0)\frac{2}{t^2}\<V_x,V_x\>.%(1+O(t)),
\end{equation*}
Leveraging the Taylor's expansion, $K(t) = K(0) + K'(0)t + K''(0)t^2/2! + K^{(3)}(0)t^3/3! + K^{(4)}(0)t^4/4!+O(t^4)$, the second term is,

\begin{align*}
&\frac{2}{t^2}\frac{K'(t)}{\sin(t)}\left(\<V_x,V_x\>+t\<V_x,\nabla_uV\>+O(t^2)\right)\\
&=\frac{2}{t^3}K'(t)\left(\<V_x,V_x\>+t\<V_x,\nabla_uV\>+O(t^2)\right)\frac{t}{\sin(t)}\\
&=\frac{2}{t^3}\left(K'(0)+K''(0)t+K^{(3)}(0)\frac{t^2}{2}+K^{(4)}(0)\frac{t^3}{3!}+O(t^4)\right){\left(\<V_x,V_x\>+t\<V_x,\nabla_uV\>+O(t^2)\right)\frac{t}{\sin(t)}}\\
&=\frac{t}{\sin(t)}\left\{\left(K'(0)\frac{2}{t^3}+K''(0)\frac{2}{t^2}+K^{(3)}(0)\frac{1}{t}+\frac{1}{3}K^{(4)}(0)\right)\<V_x,V_x\>+O(t)
\right.\\
&\hspace{2cm}\left.+\left(K'(0)\frac{2}{t^2}+K''(0)\frac{2}{t}+K^{(3)}(0)+\frac{t}{3}K^{(4)}(0)\right)\<V_x,\nabla_uV\>+O(t^2)\right.\\
&\hspace{2cm}+\left.K'(0)\frac{2}{t}+2K''(0)+K^{(3)}(0)t+K^{(4)}(0)\frac{t^2}{3}+O(t^3)\right\}.
\end{align*}
The third term, as a multiple of $\<V_x,u\>$, is
\begin{equation*}
    \begin{split}
        &-\frac{2}{t^2}\left(K''(t)-\frac{K'(t)\cos(t)}{\sin(t)}\right)\frac{t\<x,\nabla_uV\>+O(t^2)}{\sin(t)}\\
        &=-\frac{2}{t^2}\left(K''(t)-\frac{K'(t)\cos(t)}{\sin(t)}\right)\left(\<x,\nabla_uV\>\frac{t}{\sin(t)}+\frac{t}{\sin(t)}O(t)\right).
    \end{split}
\end{equation*}
Individually they are
\begin{equation*}
    \begin{split}
        &-\left(K''(0)\frac{2}{t^2}+K^{(3)}(0)\frac{2}{t}+K^{(4)}(0)+O(t)\right)\frac{t}{\sin(t)}\<x,\nabla_uV\>\\
        &\pushright{-\left(K''(0)\frac{2}{t}+2K^{(3)}(0)+K^{(4)}(0)t+O(t^2)\right)\frac{t}{\sin(t)},}
    \end{split}
\end{equation*}
and,
\begin{equation*}
\begin{split}
    &\left(K'(0)\frac{2}{t^3}+K''(0)\frac{2}{t^2}+K^{(3)}(0)t+K^{(4)}(0)\frac{1}{3}+O(t)\right)\cos(t)\frac{t^2}{\sin^2(t)}\<x,\nabla_uV\>\\
    &\pushright{ +\left(K'(0)\frac{2}{t^2}+K''(0)\frac{2}{t}+K^{(3)}(0)+K^{(4)}(0)\frac{t}{3}+O(t^2)\right)\cos(t)\frac{t^2}{\sin^2(t)}.}
\end{split}
\end{equation*}
Now evaluating the limit after collecting expressions for the terms above we get,
\begin{equation*}
\begin{split}
    &\lim_{t\to0}2\left(\frac{1}{t^3}\<V_x,V_x\>+\frac{1}{t^2}\<V_x,\nabla_uV\>+\frac{1}{t}+\frac{1}{t^3}\<V_x,u\>\<x,\nabla_uV\>+\frac{1}{t^2}\<V_x,u\>\right)K'(0)\\
    &+(1-\<\nabla_uV,\nabla_uV\>)K''(0)\\
    &+\lim_{t\to0}\left(\frac{1}{t}\<V_x,V_x\>+\<V_x,\nabla_uV\>-\frac{2}{t}\<V_x,u\>\<x,\nabla_uV\>-\<V_x,u\>\right)K^{(3)}(0)\\
    &+\left(\frac{1}{3}\<V_x,V_x\>-\frac{2}{3}\<V_x,u\>\<x,\nabla_uV\>\right)K^{(4)}(0).
\end{split}
\end{equation*}
Since, $V_x$ is a non-zero vector field, $Z$ is 2-MSD if and only if $K'(0)=0$, $K^{(3)}(0)=0$, $|K''(0)|<\infty$ and $|K^{(4)}(0)|< \infty$. 
\end{proof}

\begin{proof}[Proof of \Cref{thm:S2_kernels}]
   Recall that the exponential function at $x$ admits a simple form: $\exp_{x_0}(tv) = \cos(t)x+\sin(t)v$ for $\|v\|=1$, and the geodesic distance (great circle distance) between $x_0$ and $\exp_{x_0}(tv)$ is given by
    $\theta\coloneqq d(x_0,\exp_{x_0}(tv))=t$ while the Euclidean distance or, chordal distance, is given by $\|x_0-\exp_{x_0}(tv)\|=2\sin\left(\frac{\theta}{2}\right)=2\sin\left(\frac{t}{2}\right)$.

By Theorem \ref{thm:MSC1-MSD_iso} and \ref{thm:D^2_Sp}, it suffices to check the leading term of $K(0)-K(t)$ where $t\sim 0$, denoted by $O(t^\eta)$: The GP is 1-MSD if and only if the $\eta\geq 2$. Now we calculate the leading term of the 12 kernels one by one.
\begin{enumerate}[leftmargin=*, align=left]
\item[\textbf{1. Chordal Mat\'ern.}] Observe that $\frac{t}{2\sin(t/2)}=1$, we can replace $2\sin(t/2)$ by $t$ in all limits that involve $t\to0$. As a result, the limiting behavior of $K$ around $0$ is the same as the Mat\'ern kernel in Euclidean space: $K(t)=\alpha^\nu K_\nu(\alpha t)$. As a result, $Z$ is MSC if $\nu>0$, 1-MSD if $\nu>1$ and 2-MSD if $\nu>2$ from existing literature~\cite{stein1999interpolation}. 
\item[\textbf{2. Circular Mat\'ern.}] Note that $|\exp(ilt)|=1$, so $|\sum_{l=-\infty}^\infty(\alpha^2+l^2)^{-\nu-1/2}\exp(ilt)|\leq \sum_{l=-\infty}^\infty(\alpha^2+l^2)^{-\nu-1/2}<\infty$ for any $\nu>0$. As a result, we can exchange limit, sum, and derivative. First, observe that 
\begin{align*}
    \lim_{t\to0}K(t)&=\lim_{t\to0}\sum_{l=-\infty}^\infty(\alpha^2+l^2)^{-\nu-1/2}\exp(ilt)=\sum_{l=-\infty}^\infty(\alpha^2+l^2)^{-\nu-1/2}\lim_{t\to0}\exp(ilt)\\
    & =\sum_{l=-\infty}^\infty(\alpha^2+l^2)^{-\nu-1/2} = K(0).
\end{align*}
As a result, $Z$ is MSC. Second, observe that
\begin{align*}
    K'(0)&=\frac{d}{dt}\sum_{l=-\infty}^\infty(\alpha^2+l^2)^{-\nu-1/2}\exp(ilt)\bigg{|}_{t=0}
    = \sum_{l=-\infty}^\infty(\alpha^2+l^2)^{-\nu-1/2}\frac{d\exp(ilt)}{dt}\bigg{|}_{t=0}\\
    & = \sum_{l=-\infty}^\infty(\alpha^2+l^2)^{-\nu-1/2}il\exp(ilt)\bigg{|}_{t=0}=i\sum_{l=-\infty}^\infty(\alpha^2+l^2)^{-\nu-1/2}l\neq 0.
\end{align*}
so $Z$ is not 1-MSD.
\item [\textbf{3. Legendre-Mat\'ern.}]
By \cite{borovitskiy2020matern}, the Legendre-Mat\'ern can be expressed as $K(x,x')= \sum_{l=0}^\infty(\alpha^2+l^2)^{-\nu-1/2}f_l(x)f_l(x')$ where $f_l$ is the spherical harmonics. As a result, it coincide with the Mat\'ern defined in Equation \ref{eqn:Matern} so the condition for MSC, 1-MSD and 2-MSD follow Theorem \ref{thm:MSC},\ref{thm:1MSD} and \ref{thm:2MSD}. 
\item [\textbf{4. Truncated Legendre-Mat\'ern.}] 
The truncated Legendre-Mat\'ern can be expressed as $K(x,x')\propto \sum_{l=1}^T(\alpha^2+l^2)^{-\nu-1/2}f_l(x)f_l(x')$ where $f_l$ is the spherical harmonics \citep[see e.g.,][]{borovitskiy2020matern}. As a direct consequence of Theorem \ref{thm:truncation}, it is MSC, 1-MSD and 2-MSD due to the finite truncation. 

\item[\textbf{4. Bernoulli Mat\'ern.}] Similar to the Circular Mat\'ern case, we have
\begin{align*}
    K(t)&=1+\alpha+\sum_{l\neq 0}|l|^{-2n}\exp(ilt)
    = 1+\alpha+\sum_{l\neq 0}|l|^{-2n}(1+ilt+O(t^2))\\
    &= K(0)+ O(t).
\end{align*}
As a result, $Z$ is MSC but not 1-MSD.
\item[\textbf{5. Powered exponential.}] Observe that 
\begin{align*}
    K(t)&=\exp(-(\alpha t)^\nu)
    = (1-(\alpha t)^\nu+O(t^{2\nu})) = K(0)+ O(t^{\nu}).
\end{align*}
Since $\nu\in(0,1]$, $Z$ is MSC but not 1-MSD.
\item[\textbf{6. Generalized Cauchy.}] Observe that 
\begin{align*}
    K(t)&=\left(1+(\alpha t )^\nu\right)^{-\tau/\nu}
    = 1-\frac{\tau}{\nu}(\alpha t)^\nu+O(t^{2\nu})= K(0)+ O(t^{\nu}).
\end{align*}
Since $\nu\in(0,1]$, $Z$ is MSC but not 1-MSD.
\item[\textbf{7. Multiquadric.}] Since the numerator won't affect the smoothness, we set it to be one for simplicity.
\begin{align*}
    K(t)&=\frac{1}{\left(1+\tau^2-2\tau\cos  t \right)^\alpha}
    = \frac{1}{\left(1+\tau^2-2\tau(1-O(t^2))\right)^\alpha}\\
    & = \frac{1}{((1-\tau)^2+O(t^2))^\alpha}=\frac{1}{(1-\tau)^{2\alpha}}\frac{1}{1+O(t^2)}\\
    &=\frac{1}{(1-\tau)^{2\alpha}}(1+O(t^2)) = K(0)+ O(t^{2}).
\end{align*}
As a result, $Z$ is MSC and 1-MSD. To check 2-MSD, we need to analyze the higher order terms:
\begin{align*}
    K(t)&=\frac{1}{\left(1+\tau^2-2\tau\cos  t \right)^\alpha}
    = \frac{1}{\left(1+\tau^2-2\tau(1-t^2/2+O(t^4))\right)^\alpha}\\
    & = \frac{1}{((1-\tau)^2+\tau t^2+O(t^4))^\alpha}=\frac{1}{(1-\tau)^{2\alpha}}\frac{1}{(1+\frac{\tau}{(1-\tau)^{2}}t^{2}+O(t^4))^\alpha}\\
    &=\frac{1}{(1-\tau)^{2\alpha}}\frac{1}{1+\frac{\tau\alpha}{(1-\tau)^{2}}t^{2}+O(t^4))}\\
    &= \frac{1}{(1-\tau)^{2\alpha}}\left(1+\frac{\tau\alpha}{(1-\tau)^{2}}t^{2}+O(t^4)\right)\\
    &=K(0)+ \frac{1}{2}K''(0)t^{2}+O(t^4),
\end{align*}
so we conclude that $Z$ is 2-MSD.

\item[\textbf{8. Sine Power.}] Similar to Chordal Mat\'ern, we replace $\sin\left(\frac{t}{2}\right)$ by $\frac{t}{2}$:
\begin{align*}
    K(t)&=1-\left(\frac{ t }{2}\right)^\nu
    = K(0)+ O(t^{\nu}).
\end{align*}
Since $\nu\in(0,2)$, $Z$ is MSC but not 1-MSD.
\item [\textbf{9. Spherical.}] When $t\to0$, we can assume $t<\frac{1}{\alpha}$ so that $(1-\alpha t)_+ =1-\alpha t$, then observe that
\begin{align*}
    K(t)&=\left(1+\frac{\alpha t }{2}\right)(1-\alpha t )_+^2
    =\left(1+\frac{\alpha t }{2}\right)(1-\alpha t )^2\\
    & = 1+O(t) = K(0)+ O(t).
\end{align*}
As a result, $Z$ is MSC but not 1-MSD.
\item[\textbf{10. Askey.}] Similar to Spherical kernel, 
observe that
\begin{align*}
    K(t)&=(1-\alpha t )_+^\tau
    =(1-\alpha t )^\tau = 1+O(t) = K(0)+ O(t).
\end{align*}
As a result, $Z$ is MSC but not 1-MSD.
\item [\textbf{11. $C^2$-Wendland.}] Similar to Spherical kernel, 
observe that
\begin{align*}
    K(t)&=(1+\tau\alpha t )(1-\alpha t )_+^\tau
    =(1+\tau\alpha t )(1-\alpha t )^\tau\\
    &= (1+\tau\alpha t)(1-\tau\alpha t+O(t^2))=1-\tau\alpha t+\tau\alpha t +O(t^2)\\
    &= K(0)+ O(t^2).
\end{align*}
As a result, $Z$ is MSC and 1-MSD. To check 2-MSD, we need to analyze the higher order terms:
\begin{align*}
    K(t)&=(1+\tau\alpha t )(1-\alpha t )_+^\tau
    =(1+\tau\alpha t )(1-\alpha t )^\tau\\
    &= (1+\tau\alpha t)\left(1-\tau\alpha t+\frac{\tau(\tau-1)}{2}\alpha^2t^2-\frac{\tau(\tau-1)(\tau-2)}{6}\alpha^3t^3+O(t^4)\right)\\
    &=1-\tau\alpha t+\tau\alpha t +\left(\frac{\tau(\tau-1)\alpha^2}{2}-\tau^2\alpha^2\right)t^2+{\left(\frac{\tau^2(\tau-1)\alpha^3}{2}-\frac{\tau(\tau-1)(\tau-2)\alpha^3}{6}\right)t^3+O(t^4)}\\
    &= K(0)+ \frac{1}{2}K''(0)t^2+\frac{\tau(\tau-1)(\tau+1)\alpha^3}{3}t^3+O(t^4).
\end{align*}
So we conclude that $Z$ is 1-MSD but not 2-MSD. 
\item[\textbf{12. $C^4$-Wendland.}]Similar to $C^2$-Wendland, 
observe that
\begin{align*}
    K(t)&=\left(1+\tau\alpha t +\frac{\tau^2-1}{3}(\alpha t )^2\right)(1-\alpha t )_+^\tau
    =\left(1+\tau\alpha t +\frac{\tau^2-1}{3}(\alpha t )^2\right)(1-\alpha t )^\tau\\
    &= \left(1+\tau\alpha t +O(t^2)\right)(1-\tau\alpha t+O(t^2))=1-\tau\alpha t+\tau\alpha t +O(t^2)\\
    &= K(0)+ O(t^2).
\end{align*}
As a result, $Z$ is MSC and 1-MSD. To check 2-MSD, we need to analyze the higher order terms:
\begin{align*}
    K(t)&=\left(1+\tau\alpha t +\frac{\tau^2-1}{3}(\alpha t )^2\right)(1-\alpha t )_+^\tau
    =\left(1+\tau\alpha t +\frac{\tau^2-1}{3}(\alpha t )^2\right)(1-\alpha t )^\tau\\
    &= \left(1+\tau\alpha t +\frac{\tau^2-1}{3}\alpha^2 t^2\right){\left(1-\tau\alpha t+\frac{\tau(\tau-1)}{2}\alpha^2t^2-\frac{\tau(\tau-1)(\tau-2)}{6}\alpha^3t^3+O(t^4)\right)}\\
    &= 1-\tau\alpha t+\tau\alpha t +\left(\frac{(\tau^2-1)\alpha^2}{3}+\frac{\tau(\tau-1)\alpha^2}{2}-\tau^2\alpha^2\right)t^2\\
    &\pushright{+\left(\frac{\tau^2(\tau-1)\alpha^3}{2}-\frac{\tau(\tau-1)(\tau-2)\alpha^3}{6}-\frac{(\tau^2-1)\tau\alpha^3}{3}\right)t^3+O(t^4)}\\
    &= K(0)+ \frac{1}{2}K''(0)t^2+O(t^4).
\end{align*}
Therefore, we conclude that $Z$ is 2-MSD. 
\end{enumerate}

\end{proof}

\begin{proof}[Proof of \Cref{prop:S1_KV}]
Note that the distance $|x-x'|\in[0,1]$ is not the geodesic distance we used above, but they coincide up to a rescaling by $2\pi$. Equivalently, we can re-scale the tangent vectors by $2\pi$ in all exponential maps. As a result, we still keep the same notation but the distance different by $2\pi$. 
% Note that 
% $K_{1/2}'(s)=-\frac{\sigma^2}{\cosh(\frac{\alpha}{2})}\alpha\sinh(u),~K_{1/2}''(s)=-\frac{\sigma^2}{\cosh(\frac{\alpha}{2})}\alpha^2\cosh(u).$
%     \begin{align*}
%     &K_{V,1/2}(x,x')=\frac{K_{1/2}''(s) \<V_x,x'\> \<x,V_{x'}\> } {{1 - \<x,x'\>^2}} -\frac{K_{1/2}'(s)\<V_x,V_{x'}\>}{\sqrt{1 - \<x,x'\>^2}}-\frac{ K_{1/2}'(s))  \<x,V_{x'}\>\<x,x'\>\<V_x,x'\> }{\sqrt{1 - \<x,x'\>^2}^3}\\
%     &=-\frac{\sigma^2}{\cosh(\frac{\alpha}{2})}\left( \frac{\alpha^2 \cosh(u)\<V_x,x'\> \<x,V_{x'}\> } {{1 - \<x,x'\>^2}} -\frac{\alpha \sinh(u)\<V_x,V_{x'}\>}{\sqrt{1 - \<x,x'\>^2}}-\frac{ \alpha \sinh(u)  \<x,V_{x'}\>\<x,x'\>\<V_x,x'\> }{\sqrt{1 - \<x,x'\>^2}^3}\right).
%     \end{align*}

For $s=1$, 
\begin{align*}
    K_{3/2}'(u) &=\frac{\sigma^2}{C_{\nu,3/2}}\left[-a_{1,0}\alpha\sinh(u)+a_{1,1}(\sinh(u)+\alpha u \cosh(u))\right] \\
    K_{3/2}''(u)&=\frac{\sigma^2}{C_{\nu,3/2}}\left[-a_{1,0}\alpha^2\cosh(u)+a_{1,1}\left(\alpha\cosh(u)+\alpha\cosh(u)-\alpha^2u\sinh(u)\right)\right]. 
\end{align*}
%\ah{
Then we have
 \begin{align*}
    &K_{V,3/2}(x,x')=-\frac{K_{3/2}'(s)\<V_x,V_{x'}\>}{\sqrt{1 - \<x,x'\>^2}}\\
    &\pushright{+\left(K_{3/2}''(s) - K_{3/2}'(u)\frac{\<x,x'\>}{(1 - \<x,x'\>^2)^{1/2}}\right)\frac{\<x,V_{x'}\>}{(1 - \<x,x'\>^2)^{1/2}}\frac{\<V_x,x'\>}{(1 - \<x,x'\>^2)^{1/2}}}\\
    &=-\frac{\sigma^2}{C_{\nu,3/2}}\left[\left\{-a_{1,0}\alpha\sinh(u)+a_{1,1}(\sinh(u)+\alpha u \cosh(u))\right\}\frac{\<V_x,V_{x'}\>}{\sqrt{1 - \<x,x'\>^2}}  \right.\\&\left.\hspace{2cm}+\left(\vphantom{\frac{\<x,x'\>}{(1 - \<x,x'\>^2)^{1/2}}}-a_{1,0}\alpha^2\cosh(u)+a_{1,1}\left(\alpha\cosh(u)+\alpha\cosh(u)-\alpha^2u\sinh(u)\right)-\right.\right.\\&\left.\hspace{3cm}\left[-a_{1,0}\alpha\sinh(u)+a_{1,1}(\sinh(u)+\alpha u \cosh(u))\right]\frac{\<x,x'\>}{(1 - \<x,x'\>^2)^{1/2}}\right)\\&\pushright{\left.\times\frac{\<x,V_{x'}\>}{(1 - \<x,x'\>^2)^{1/2}}\frac{\<V_x,x'\>}{(1 - \<x,x'\>^2)^{1/2}}\right]}.
    \end{align*}
%}
For $s=2$,
\begin{align*}
    K_{5/2}'(u) &=\frac{\sigma^2}{C_{\nu,5/2}}\left[-a_{2,0}\alpha\sinh(u)+a_{2,1}(\sinh(u)+\alpha u \cosh(u))+a_{22}(2u\cosh(u)-\alpha u^2\sinh(u))\right],\\
    K_{5/2}''(u) &=\frac{\sigma^2}{C_{\nu,5/2}}\left[-a_{2,0}\alpha^2\cosh(u)+a_{2,1}\left(\alpha\cosh(u)+\alpha\cosh(u)-\alpha^2u\sinh(u)\right)\right.\\&\pushright{+\left.a_{22}\left(2\cosh(u)-2\alpha u\sinh(u)-2\alpha u\sinh(u)-\alpha^2u^2\cosh(u)\right)\right]}\\
    & = \frac{\sigma^2}{C_{\nu,5/2}}\left[-a_{2,0}\alpha^2\cosh(u)+a_{2,1}\left(2\alpha\cosh(u)-\alpha^2u\sinh(u)\right)\right.\\
    &\pushright{+\left.a_{22}\left(2\cosh(u)-4\alpha u\sinh(u)-\alpha^2u^2\cosh(u)\right)\right]}.
\end{align*}
%\ah{ \vphantom{}
Then we have 
 \begin{align*}
    &K_{V,5/2}(x,x')\\&=-\frac{K_{5/2}'(s)\<V_x,V_{x'}\>}{\sqrt{1 - \<x,x'\>^2}}+\left(K_{5/2}''(s) - K_{5/2}'(s)\frac{\<x,x'\>}{(1 - \<x,x'\>^2)^{1/2}}\right)\frac{\<x,V_{x'}\>}{(1 - \<x,x'\>^2)^{1/2}}\frac{\<V_x,x'\>}{(1 - \<x,x'\>^2)^{1/2}}\\
    &=-\frac{\sigma^2}{C_{\nu,5/2}}\left[\left\{-a_{2,0}\alpha\sinh(u)+a_{2,1}(\sinh(u)+\alpha u \cosh(u))\right.\right.\\
    &\hspace{2cm}+\left.\left.a_{22}(2u\cosh(u)-\alpha u^2\sinh(u))\right\}\frac{\<V_x,V_{x'}\>}{\sqrt{1 - \<x,x'\>^2}} + \left(\frac{\<x,x'\>}{(1 - \<x,x'\>^2)^{1/2}}-a_{2,0}\alpha^2\cosh(u)\right.\right.\\
    &\hspace{2cm}+\left.\left.a_{2,1}\left(2\alpha\cosh(u)-\alpha^2u\sinh(u)\right)+a_{22}\left(2\cosh(u)-4\alpha u\sinh(u)-\alpha^2u^2\cosh(u)\right)\right.\right.\\
    &\hspace{2cm}-\left.\left[-a_{2,0}\alpha\sinh(u)+a_{2,1}(\sinh(u)+\alpha u \cosh(u))\right.\right.\\
    &\hspace{2cm}+\left.\left.a_{22}(2u\cosh(u)-\alpha u^2\sinh(u))\right]\frac{\<x,x'\>}{(1 - \<x,x'\>^2)^{1/2}}\right)\left.\frac{\<x,V_{x'}\>}{(1 - \<x,x'\>^2)^{1/2}}\frac{\<V_x,x'\>}{(1 - \<x,x'\>^2)^{1/2}}\right].
    % &=\frac{\sigma^2}{C_{\nu,5/2}}\left( \frac{\left[-a_{2,0}\alpha^2\cosh(u)+a_{2,1}\left(2\alpha\cosh(u)-\alpha^2u\sinh(u)\right)+a_{22}\left(2\cosh(u)-4\alpha u\sinh(u)-\alpha^2u^2\cosh(u)\right)\right]\<V_x,x'\> \<x,V_{x'}\> } {{1 - \<x,x'\>^2}} -\right.\\
    % &\left.\hspace{1cm}\frac{\left[-a_{2,0}\alpha\sinh(u)+a_{2,1}(\sinh(u)+\alpha u \cosh(u))+a_{22}(2u\cosh(u)-\alpha u^2\sinh(u))\right]\<V_x,V_{x'}\>}{\sqrt{1 - \<x,x'\>^2}}-\right.\\&\hspace{1cm}\left.\frac{ \left[-a_{2,0}\alpha\sinh(u)+a_{2,1}(\sinh(u)+\alpha u \cosh(u))+a_{22}(2u\cosh(u)-\alpha u^2\sinh(u))\right] \<x,V_{x'}\>\<x,x'\>\<V_x,x'\> }{\sqrt{1 - \<x,x'\>^2}^3}\right).
    \end{align*}
%}
\end{proof}

\section{Computational Details for Posterior Inference on \texorpdfstring{$D_VZ$}{DVZ} over \texorpdfstring{$\mathbb{S}^2$}{S2}}\label{sec:grad-posterior}
We use the truncated Legendre-Mat\'ern covariance (see \Cref{tab:S2_Kernels}) to derive the posterior for derivatives using \Cref{eqn:D_VZ_Sp,eq:ZD_VZ_Sp}. Let $a_l = (\alpha^2+l^2)^{-\nu-\frac{1}{2}}$ then, ${\widetilde{K}}'(t)= \sum_{l=0}^{T-1}a_l P_l'(\cos t)(-\sin t)$ and ${\widetilde{K}}''(t)= \sum_{l=0}^{T-1} a_l\left\{\sin^2 t \;P_l''(\cos t)-\cos t\; P_l'(\cos t)\right\}$. We use standard results in spherical harmonics: (a) $P_l'(x)=\frac{l}{x^2-1}(x P_l(x)-P_{l-1}(x))$ and (b) $(1-x^2)P_l''(x)-2xP_l'(x)+l(l+1)P_l(x) = 0$. Substituting (a) in ${\widetilde{K}}'(t)$ yields
\begin{equation*}
    {\widetilde{K}}'(t)= \sum\limits_{l=1}^{T-1} a_l\frac{l}{\sin t}\left\{\cos t\; P_l(\cos t)-P_{l-1}(\cos t)\right\}.
\end{equation*}
Substituting the recurrence relation in ${\widetilde{K}}''(t)$ and simplifying we obtain
\begin{equation*}
    {\widetilde{K}}''(t)= \sum\limits_{l=1}^{T-1} a_l\left[\left(-\frac{l\cos t}{\sin^2 t}\right)\left\{\cos t\; P_l(\cos t)-P_{l-1}(\cos t)\right\}-l(l+1)P_l(\cos t)\right].
\end{equation*}
For a point $x = (x^1, x^2, x^3) \in \mathbb{S}^2$, we will use the rotational vector field, also known as the longitude vector field: $V_x\coloneqq (-x^2,x^1,0)$. For $d=\arccos\langle x, x'\rangle$, we have:
\begin{equation*}
    \begin{split}
        \cov(Z(x), D_VZ(x'))=-\frac{{\widetilde{K}}'(d)(-x^1{x'}^2+x^2{x'}^1)}{\{1-(x^1{x'}^1+x^2{x'}^2+x^3{x'}^3)^2\}^{\frac{1}{2}}},
    \end{split}
\end{equation*}
and $K_V(x, x)=-{\widetilde{K}}''(0)(x^2{x'}^2+x^1{x'}^1)$. We evaluate ${\widetilde{K}}''(0)$ as $\lim_{t\to 0}{\widetilde{K}}''(t)$. For small $t$, $\cos t = 1- \frac{t^2}{2}+O(t^4)$, $\sin t = t+O(t^3)$ and $P_l(t)=1+P_l'(1)(t-1)+O((t-1)^2)$. Now $P_l'(1)=\frac{l(l+1)}{2}$ and $\cos t-1 = - \frac{t^2}{2}+O(t^4)$ hence, $P_l(\cos t) = 1-\frac{l(l+1)}{2}\frac{t^2}{2}+O(t^4)$ and $P_{l-1}(\cos t) = 1-\frac{(l-1)l}{2}\frac{t^2}{2}+O(t^4)$. Therefore, $\cos tP_l(\cos t)-P_{l-1}(\cos t) = -\frac{l+1}{2}t^2+O(t^4)$. Evaluating the limit, we have ${\widetilde{K}}''(0)=-\frac{1}{2}\sum_{l=1}^{T}a_ll(l+1)$ resulting in 
\begin{equation*}
    K_V(x, x)=\left[\frac{1}{2}\sum_{l=1}^{T-1}a_l\,l(l+1)\right]\left({x^2}^2+{x^1}^2\right).
\end{equation*}
Using \Cref{eq:grad-posterior}, the conditional predictive distribution is 
$D_VZ(x_0)\mid\;\Z,\theta\sim N(\mu_1, \Sigma_1)$ where $\mu_1 = {C_{\Z, D_V\Z}}\T {C_{\Z,\Z}}^{-1}\Z$ and $\Sigma_1 = K_V(x_0,x_0)-{C_{D_VZ, \Z}}\T {C_{\Z,\Z}}^{-1}C_{\Z,D_VZ}$, where $K_V(x_0,x_0)$ is a scalar, while $C_{\Z,D_VZ}=\left(\cov(Z(x_1),D_VZ(x_0)), \ldots, \cov(Z(x_n),D_VZ(x_0))\right)\T$ and $C_{D_VZ,\Z}=(\cov(D_VZ(x_0), Z(x_1)), \ldots, \cov(D_VZ(x_0), Z(x_n))\T$ are $n\times 1$ vectors. Also, $\cov(Z(x),D_VZ(x'))=-\cov(D_VZ(x),Z(x'))$ implying $C_{D_VZ, \Z}=-C_{\Z,D_VZ}$.
% % If there are more than one appendix, then please refer to it
% % as \ldots\ in Appendix \ref{appA}, Appendix \ref{appB}, etc.

% \section{Title of the second appendix}\label{appB}
% \subsection{First subsection of Appendix \protect\ref{appB}}

% Use the standard \LaTeX\ commands for headings in \verb|{appendix}|.
% Headings and other objects will be numbered automatically.
% \begin{equation}
% \mathcal{P}=(j_{k,1},j_{k,2},\dots,j_{k,m(k)}). \label{path}
% \end{equation}

% Sample of cross-reference to the formula (\ref{path}) in Appendix \ref{appB}.
\section{Barycentric Coordinates---Mesh Sampling and Interpolation}\label{sec: barycentric_sampling_interpolation}
We work with $Z:M\to \RR$ and $Z(x)\sim GP(\mu(x),K(\cdot;\theta))$. Generating random samples requires generating scattered locations on $M$. Subsequently, the eigenfunctions need to be interpolated at these scattered locations. We use a \emph{linear} (P1) barycentric coordinate system that is defined using the vertices for triangles of the mesh, $M$ to achieve these goals. % Generating scattered data involves the following steps. 
To generate $N$ scattered locations on $M$, 
\begin{enumerate}
    \item Compute triangle areas, $A_i$ and sampling weights, $p_{t_i}=\frac{A_i}{\sum_iA_i}$, $i=1,\ldots,N_T$;
    \item Draw a weighted random sample with replacement of triangle indices of size $N$ from $\{1,\ldots, N_T\}$ using the sampling weights, $p_{t_1},\dots, p_{t_{N_T}}$ and denote the triangle vertices of the $j$-th sample as $(v^j_{k_1}, v^j_{k_2},v^j_{k_3})$, $k_1, k_2, k_3\in\{1,\ldots,K\}$;
    \item Draw random samples $a^1_1,\ldots,a^1_{N}\stackrel{iid}{\sim}U(0,1)$, $a^2_1,\ldots,a^2_{N}\stackrel{iid}{\sim}U(0,1)$. If for any $j$, $a^1_j + a^2_j>1$ then, $a^1_j \leftarrow 1-a^1_j$ and $a^2_j\leftarrow1-a^2_j$ and define $a^3_j=1-a^1_j-a^2_j$, $j=1,\dots,N$ which are the barycentric coordinates;
    \item Get random locations $x_j = a^1_jv^j_{k_1}+a^2_jv^j_{k_2}+a^3_jv^j_{k_3}$, $j=1,\dots,N$ on $M$. 
\end{enumerate}
The \emph{barycentric coordinates} for the random sample are, $a_j=(a^1_j,a^2_j,a^3_j)$, $j=1,\dots,N$. Approximate eigenfunctions at the point $x_j$ denoted by, $f_l(x_j)=a^1_jf_l(v^j_{k_1})+a^2_jf_l(v^j_{k_2})+a^3_jf_l(v^j_{k_3})$ which is the \emph{barycentric interpolation}. For more details see e.g., \cite{reuter2006laplace,brenner2008mathematical,botsch2010polygon}.

\section{Farthest Point Sampling (FPS)---Grids over Manifolds}\label{sec: farthest_point_sampling} 
We require a point cloud resembling an equally spaced grid in the Euclidean space for interpolating $D_VZ$ on $M$. To generate such a point cloud on $M$ we first generate a dense (large sized, i.e. $N>>N_T$) random sample (point cloud) on $M$ using the steps detailed in \Cref{sec: barycentric_sampling_interpolation}. To generate a grid-like point cloud consisting of $N_{G}$ points, we down-sample using FPS % which is a greedy strategy 
\citep[see e.g.,][]{eldar1997farthest,pauly2002efficient} we follow:
\begin{enumerate}
    \item[(a)] Select a random point from the dense point cloud;
    \item[(b)] For each remaining point calculate its distance from the nearest already sampled point;
    \item[(c)] Choose point that has the maximum distance to the selected point and add it to the grid;
    \item[(d)] Repeat steps (b) and (c) until $N_G$ points are obtained.
\end{enumerate}
We use Euclidean distances to perform the selection step. Future work can consider geodesic distances for $\MM$. We use $N_G=400$ for our experiments in \Cref{sec: bunny}.

\section{Surface Registration using Scattered Data}
\Cref{fig:S1andS2,fig:bunny-grad-sim,fig: bunny_eigen} are the results of surface interpolation from scattered data on $M$%, which allows for visualizing the proposed methods
. The surfaces are generated from observations at $N$ scattered locations on $M$. We first \emph{scatter} the partially observed data to the vertices using the barycentric coordinates obtained from the sample. At vertex $v_k$, $Z(v_k)=\frac{\sum_{j\in \mathcal{I}(v_k)}a_jZ(x_j)}{\sum_{j\in \mathcal{I}(v_k)}a_j}$, where $x_j$ are sampled locations whose triangles include vertex $v_k$; $a_j$ is the barycentric weight of $x_j$; and $\mathcal{I}(v_k)$ is the set of samples contributing to $v_k$. We use splines to smooth $Z$ which involves solving $\min_{Z^{\rm sm}}\sum_k w_k(Z^{\rm sm}(v_k)-Z(v_k))^2+\iota\int_\MM||\nabla Z^{\rm sm}||^2$ on $M$, where $\iota$ is the penalty term \citep[see e.g.][]{wahba1981spline}. The regularization involves an Euler-Lagrange equation. Taking the first variation requires the cotangent Laplacian, $\Delta$. We use the smooth interpolated $Z^{\rm sm}$ to generate surface plots for the manuscript. Future work can extend to methods resembling a combination of \cite{wahba1981spline,finley2024mba} to generate smooth surface interpolations on %compact Riemannian manifolds,
$\MM$.

\section{Computational Details for Surfaces (SB)}
%We note that fitting the GP using a truncated kernel, $K(x,x')=\frac{\sigma^2}{C_{\nu,\alpha}(x)}\sum_{l=0}^{T}(\alpha^2+\lambda_l)^{-\nu-\frac{p}{2}}f_l(x)f_l(x')$, where $C_{\nu,\alpha}(x)=\sum_{l=0}^{T}(\alpha^2+\lambda_l)^{-\nu-\frac{p}{2}}f_l(x)^2$ depends on $x$, impacts the scaling factor. To avoid unnecessary complications in the ensuing calculations for the cross-covariance matrix (see \Cref{lem:ZD_VZ,eq:gp-grad}), we fit the GP with observation-level scaling, $C_{\nu,\alpha}(x_j)$, $j=1,\dots,N$ to obtain posterior estimates for $\sigma^2, \alpha$ and then set $\widetilde{C}_{\nu,\alpha}=\frac{1}{N}\sum_{i=1}^{N}C_{\nu,\alpha}(x_j)$. Also, we note from the expressions in \Cref{cly:ZD_VZMat,eq:grad-posterior} that the posterior mean of $D_VZ$ is scale-free (i.e., no need for $\sigma^2/C_{\nu,\alpha}$) and scaling is required only for the posterior variance where we use $\sigma^2/\widetilde{C}_{\nu,\alpha}$.
\begin{figure}[t]
    \centering
    \includegraphics[width=0.7\linewidth]{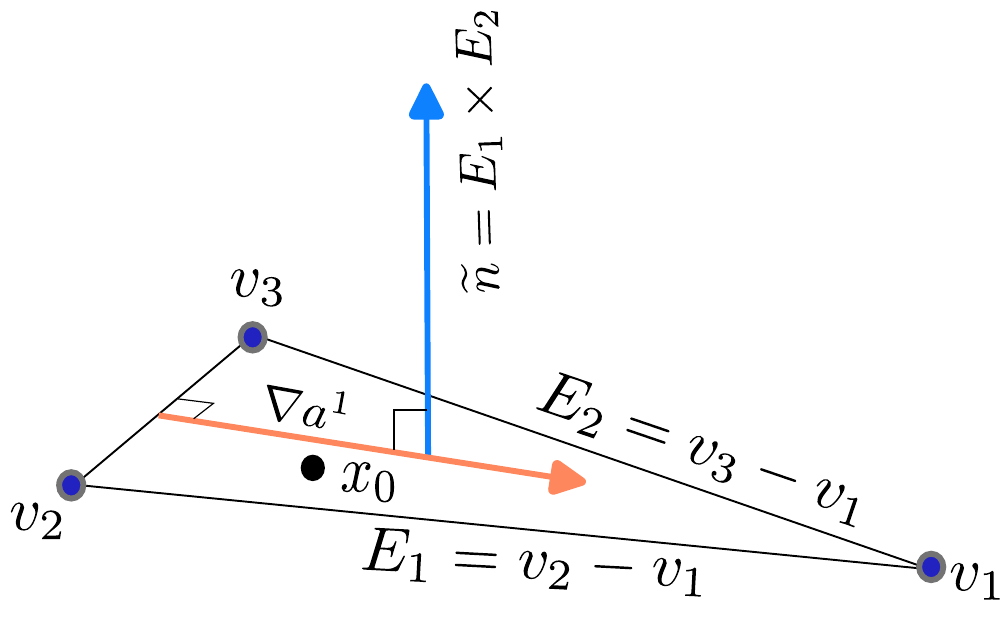}
    \caption{Components required for computing the differential of barycentric coordinates for a triangle within the mesh; here, $x_0$ is an arbitrary grid point where we seek to infer on $D_VZ$. The blue arrow indicates the normal to the triangle projecting out of the plane while the red arrow is the normal to the edge $(v_2, v_3)$ lying on the plane.}
    \label{fig: diff-bary}
\end{figure}
We provide computational details for posterior inference on $D_VZ(x_0)$ at an arbitrary point, $x_0$ lying on a grid constructed using FPS (see \Cref{sec: farthest_point_sampling}) on $M$. For each triangle we use, $\gamma(t)=x_0+tV_{x_0}$ as the linear approximation to a geodesic starting from $x_0$ along $Z$, $\gamma(0)=x_0$ and $\gamma'(0) = V_{x_0}$. For a realization, $Z(x_j)$, the expressions in Corollary 3.7 yield
\begin{equation}\label{eq: cross-cov}
\begin{split}
    {\rm Cov}(Z(x_j), D_VZ(x_0)) &= \sum_{l=0}^T(\alpha^2+\lambda_l)^{-\nu-\frac{p}{2}}f_l(x_j)\,\<\nabla f_l(x_0), V_{x_0}\>,\\
    {\rm Cov}(D_VZ(x_0), D_VZ(x_0)) &=\sum_{l=0}^T(\alpha^2+\lambda_l)^{-\nu-\frac{p}{2}}\<\nabla f_l(x_0), V_{x_0}\>^2,
\end{split}
\end{equation}
up to a scaling constant. On our mesh, $M$, the eigenfunctions $f_l(x_0)$ are barycentric interpolations (see \Cref{sec: barycentric_sampling_interpolation}). Hence, to obtain $\nabla f_l(x_0)$, we need \emph{the differentials of barycentric coordinates} \citep[see e.g.,][]{akenine2021differential}. Below, we outline the strategy to obtain them for one grid point. Each triangle is flat; consequently, the inner-product in the Riemannian sense in expressions within \Cref{eq: cross-cov}, $\<\cdot,\cdot\>$ reduces to the usual dot product for vectors in Euclidean space. 

Let the triangle containing the grid point $x_0$ have vertices $v_1,v_2,v_3\in\RR^3$. Then, from the discussion in \Cref{sec: barycentric_sampling_interpolation}, we have the interpolated eigenfunction $f_l(x_0) = a^1(x_0)\,f_l(v_1)+a^2(x_0)\,f_l(v_2)+a^3(x_0)\,f_l(v_3)$
with $a^1(x_0)+a^2(x_0)+a^3(x_0)=1$ and $a^i(v_j)=\delta_{ij}$. Hence, $\nabla f_l(x_0) = (\nabla a^1(x_0))\,f_l(v_1)+ (\nabla a^2(x_0))\,f_l(v_2)+ (\nabla a^3(x_0))\,f_l(v_3)$. For our \emph{linear} (P1) barycentric system, $a^i(x)$, $i=1,2,3$, are linear functions implying $\nabla a^i(x)$ is constant for the face of the triangle. Define edges, $E_1 = v_2-v_1$, $E_2=v_3-v_1$ and $n=(2A)^{-1}(E_1\times E_2)$, where $\times$ denotes the cross-product and $2A=||E_1\times E_2||$ is the area of the triangle (see \Cref{fig: diff-bary}). We derive $\nabla a^1(x)$, as the others have similar expressions. Note that $a^1(x)=1$ at $v_1$ and 0 on opposite edges $v_2, v_3$. We denote, $\nabla a^1(x)\coloneqq\nabla a^1$. Hence, $\nabla a^1\perp (v_3-v_2)$ and $n\T\nabla a^1=0$ implying $\nabla a^1$ lies in the plane of the triangle. A direction in the triangular plane perpendicular to $v_3-v_2$ is given by, $n\times (v_3-v_2)$. Clearly, $\nabla a^1\propto n\times (v_3-v_2)$. Moving from the edge at $(v_2,v_3)$ to vertex $v_1$, $a^1$ increases by 1. Along the inward unit normal to the edge in the plane the change is: $||\nabla a^1||h_1=1$, where $h_1$ is the altitude from $v_1$ to the edge $(v_2,v_3)$. But, $A=\frac{1}{2}||v_3-v_2||h_1$ and $||n\times (v_3-v_2)||=||v_3-v_2||$ since, $n$ is the unit normal. Thus, $||\nabla a^1||\propto ||v_3-v_2||$. Plugging this into $||\nabla a^1||h_1=1$ yields, $\nabla a^1=(2A)^{-1}(n\times (v_3-v_2))$. Similarly, $\nabla a^2=(2A)^{-1}(n\times (v_1-v_3))$ and  $\nabla a^3=(2A)^{-1}(n\times (v_2-v_1))$. These are required to evaluate $\nabla f_l(x_0)$, which can then be iterated for all grid points. We used P1 FEM functions for derivatives, future work involving inference on the curvature process would require quadratic (P2) functions.

\section{Computation and Code Availability}\label{sec: computation_and_code}
All computation was performed in the \texttt{R} statistical environment \citep[][]{rcite} on an Apple Mac Mini M4 Pro with 64GB of RAM and 14 cores running macOS Tahoe 26.2. % The estimated run-time for the entire pipeline (loading the mesh into the environment $\to$ obtaining posterior inference on gradients) is 5.33 minutes. 
The 3-dimensional figures are generated using the Python API for Blender 4.5.4 LTS \citep[][]{blender50}. The code can be made available upon request. % The computational subroutines and mesh files (PLY format) are available for public testing and use at the anonymous GitHub repository: \url{https://github.com/arh926/manifoldGPgrad}.
\end{appendix}

\bibliographystyle{plainnat}
\bibliography{aos-main} 

@article{guinness2016isotropic,
  title={Isotropic covariance functions on spheres: Some properties and modeling considerations},
  author={Guinness, Joseph and Fuentes, Montserrat},
  journal={Journal of Multivariate Analysis},
  volume={143},
  pages={143--152},
  year={2016},
  publisher={Elsevier}
}

@article{banerjee2006bayesian,
  title={Bayesian wombling: Curvilinear gradient assessment under spatial process models},
  author={Banerjee, Sudipto and Gelfand, Alan E},
  journal={Journal of the American Statistical Association},
  volume={101},
  number={476},
  pages={1487--1501},
  year={2006},
  publisher={Taylor \& Francis}
}

@article{donnelly2001bounds,
  title={Bounds for eigenfunctions of the {L}aplacian on compact {R}iemannian manifolds},
  author={Donnelly, Harold},
  journal={Journal of Functional Analysis},
  volume={187},
  number={1},
  pages={247--261},
  year={2001},
  publisher={Elsevier}
}

@article{arnaudon2020gradient,
  title={Gradient estimates on {D}irichlet and {N}eumann eigenfunctions},
  author={Arnaudon, Marc and Thalmaier, Anton and Wang, Feng-Yu},
  journal={International Mathematics Research Notices},
  volume={2020},
  number={20},
  pages={7279--7305},
  year={2020},
  publisher={Oxford University Press}
}

@article{banerjee2003directional,
  title={Directional rates of change under spatial process models},
  author={Banerjee, Sudipto and Gelfand, Alan E and Sirmans, CF},
  journal={Journal of the American Statistical Association},
  volume={98},
  number={464},
  pages={946--954},
  year={2003},
  publisher={Taylor \& Francis}
}

@article{borovitskiy2020matern,
  title={Mat{\'e}rn {G}aussian processes on {R}iemannian manifolds},
  author={Borovitskiy, Viacheslav and Terenin, Alexander and Mostowsky, Peter and others},
  journal={Advances in Neural Information Processing Systems},
  volume={33},
  pages={12426--12437},
  year={2020}
}

@article{shi2010gradient,
  title={Gradient estimate of an eigenfunction on a compact {R}iemannian manifold without boundary},
  author={Shi, Yiqian and Xu, Bin},
  journal={Annals of Global Analysis and Geometry},
  volume={38},
  pages={21--26},
  year={2010},
  publisher={Springer}
}

@article{cheng2024hessian,
  title={Hessian estimates for Dirichlet and Neumann eigenfunctions of Laplacian},
  author={Cheng, Li-Juan and Thalmaier, Anton and Wang, Feng-Yu},
  journal={International Mathematics Research Notices},
  volume={2024},
  number={21},
  pages={13563--13585},
  year={2024},
  publisher={Oxford University Press}
}

@book{stein1999interpolation,
  title={Interpolation of spatial data: some theory for kriging},
  author={Stein, Michael L},
  year={1999},
  publisher={Springer Science \& Business Media}
}

@article{halder2024bayesian,
  title={Bayesian modeling with spatial curvature processes},
  author={Halder, Aritra and Banerjee, Sudipto and Dey, Dipak K},
  journal={Journal of the American Statistical Association},
  volume={119},
  number={546},
  pages={1155--1167},
  year={2024},
  publisher={Taylor \& Francis}
}

@article{hitczenko2012some,
  title={Some theory for anisotropic processes on the sphere},
  author={Hitczenko, Marcin and Stein, Michael L},
  journal={Statistical Methodology},
  volume={9},
  number={1-2},
  pages={211--227},
  year={2012},
  publisher={Elsevier}
}

@book{williams2006gaussian,
	title={Gaussian {P}rocesses {F}or {M}achine {L}earning},
	author={Williams, Christopher KI and Rasmussen, Carl Edward},
	year={2006},
	publisher={MIT press Cambridge, MA}
}

@book{adler1981geometry,
	title={The {G}eometry {O}f {R}andom {F}ields},
	author={Adler, Robert J},
	year={1981},
	publisher={SIAM}
}

@article{kent1989continuity,
  title={Continuity properties for random fields},
  author={Kent, John T},
  journal={The Annals of Probability},
  pages={1432--1440},
  year={1989},
  publisher={JSTOR}
}

@article{banerjee2003smoothness,
  title={On smoothness properties of spatial processes},
  author={Banerjee, Sudipto and Gelfand, AE},
  journal={Journal of Multivariate Analysis},
  volume={84},
  number={1},
  pages={85--100},
  year={2003},
  publisher={Elsevier}
}

@article{morris1993bayesian,
  title={Bayesian {D}esign and {A}nalysis {O}f {C}omputer {E}xperiments: {U}se {O}f {D}erivatives {I}n {S}urface {P}rediction},
  author={Morris, Max D and Mitchell, Toby J and Ylvisaker, Donald},
  journal={Technometrics},
  volume={35},
  number={3},
  pages={243--255},
  year={1993},
  publisher={Taylor \& Francis}
}

@article{majumdar2006gradients,
	title={Gradients {I}n {S}patial {R}esponse {S}urfaces {W}ith {A}pplication {T}o {U}rban {L}and {V}alues},
	author={Majumdar, Anandamayee and Munneke, Henry J and Gelfand, Alan E and Banerjee, Sudipto and Sirmans, CF},
	journal={Journal of Business \& Economic Statistics},
	volume={24},
	number={1},
	pages={77--90},
	year={2006},
	publisher={Taylor \& Francis}
}

@article{quick2015bayesian,
  title={Bayesian {M}odeling {A}nd {A}nalysis {F}or {G}radients {I}n {S}patiotemporal {P}rocesses},
  author={Quick, Harrison and Banerjee, Sudipto and Carlin, Bradley P},
  journal={Biometrics},
  volume={71},
  number={3},
  pages={575--584},
  year={2015},
  publisher={Wiley Online Library}
}

@article{wang2016estimating,
	title={Estimating {S}hape {C}onstrained {F}unctions {U}sing {G}aussian {P}rocesses},
	author={Wang, Xiaojing and Berger, James O},
	journal={SIAM/ASA Journal on Uncertainty Quantification},
	volume={4},
	number={1},
	pages={1--25},
	year={2016},
	publisher={SIAM}
}

@article{terres2016spatial,
	title={Spatial {P}rocess {G}radients {A}nd {T}heir {U}se {I}n {S}ensitivity {A}nalysis {F}or {E}nvironmental {P}rocesses},
	author={Terres, Maria A and Gelfand, Alan E},
	journal={Journal of Statistical Planning and Inference},
	volume={168},
	pages={106--119},
	year={2016},
	publisher={Elsevier}
}

@article{wang2018process,
	title={Process {M}odeling {F}or {S}lope {A}nd {A}spect {W}ith {A}pplication {T}o {E}levation {D}ata {M}aps},
	author={Wang, Fangpo and Bhattacharya, Anirban and Gelfand, Alan E},
	journal={Test},
	volume={27},
	number={4},
	pages={749--772},
	year={2018},
	publisher={Springer}
}

@article{womble1951differential,
  title={Differential {S}ystematics},
  author={Womble, William H},
  journal={Science},
  volume={114},
  number={2961},
  pages={315--322},
  year={1951},
  publisher={JSTOR}
}

@article{liang2009bayesian,
	title={Bayesian {W}ombling {F}or {S}patial {P}oint {P}rocesses},
	author={Liang, Shengde and Banerjee, Sudipto and Carlin, Bradley P},
	journal={Biometrics},
	volume={65},
	number={4},
	pages={1243--1253},
	year={2009},
	publisher={Wiley Online Library}
}

@article{heaton2014wombling,
	title={Wombling {A}nalysis {O}f {C}hildhood {T}umor {R}ates {I}n {F}lorida},
	author={Heaton, Matthew J},
	journal={Statistics and Public Policy},
	volume={1},
	number={1},
	pages={60--67},
	year={2014},
	publisher={Taylor \& Francis}
}

@article{terres2015using,
	title={Using {S}patial {G}radient {A}nalysis {T}o {C}larify {S}pecies {D}istributions {W}ith {A}pplication {T}o {S}outh {A}frican {P}rotea},
	author={Terres, Maria A and Gelfand, Alan E},
	journal={Journal of Geographical Systems},
	volume={17},
	number={3},
	pages={227--247},
	year={2015},
	publisher={Springer}
}

@article{banerjee2005geodetic,
  title={On geodetic distance computations in spatial modeling},
  author={Banerjee, Sudipto},
  journal={Biometrics},
  volume={61},
  number={2},
  pages={617--625},
  year={2005},
  publisher={Wiley Online Library}
}

@article{jeong2015class,
  title={A class of Mat{\'e}rn-like covariance functions for smooth processes on a sphere},
  author={Jeong, Jaehong and Jun, Mikyoung},
  journal={Spatial Statistics},
  volume={11},
  pages={1--18},
  year={2015},
  publisher={Elsevier}
}

@article{gneiting2013strictly,
  title={Strictly and non-strictly positive definite functions on spheres},
  author={Gneiting, Tilmann},
  journal={Bernoulli},
  pages={1327--1349},
  year={2013},
  publisher={JSTOR}
}

@inproceedings{feragen2015geodesic,
  title={Geodesic exponential kernels: When curvature and linearity conflict},
  author={Feragen, Aasa and Lauze, Francois and Hauberg, Soren},
  booktitle={Proceedings of the IEEE conference on computer vision and pattern recognition},
  pages={3032--3042},
  year={2015}
}

@article{niu2019intrinsic,
  title={Intrinsic Gaussian processes on complex constrained domains},
  author={Niu, Mu and Cheung, Pokman and Lin, Lizhen and Dai, Zhenwen and Lawrence, Neil and Dunson, David},
  journal={Journal of the Royal Statistical Society Series B: Statistical Methodology},
  volume={81},
  number={3},
  pages={603--627},
  year={2019},
  publisher={Oxford University Press}
}

@article{jun2008nonstationary,
  title={Nonstationary covariance models for global data},
  author={Jun, Mikyoung and Stein, Michael L},
  journal ={The Annals of Applied Statistics},
  volume={2},
  number={4},
  pages={1271--1289},
  year={2008},
  publisher={IMS}
}

@article{li2023inference,
  title={Inference for {G}aussian processes with {M}at{\'e}rn covariogram on compact {R}iemannian manifolds},
  author={Li, Didong and Tang, Wenpin and Banerjee, Sudipto},
  journal={Journal of Machine Learning Research},
  volume={24},
  number={101},
  pages={1--26},
  year={2023}
}

@article{dunson2022graph,
  title={Graph based Gaussian processes on restricted domains},
  author={Dunson, David B and Wu, Hau-Tieng and Wu, Nan},
  journal={Journal of the Royal Statistical Society Series B: Statistical Methodology},
  volume={84},
  number={2},
  pages={414--439},
  year={2022},
  publisher={Oxford University Press}
}

@article{lindgren2011explicit,
    AUTHOR = {Lindgren, Finn and Rue, H\aa vard and Lindstr\"{o}m, Johan},
     TITLE = {An explicit link between {G}aussian fields and {G}aussian
              {M}arkov random fields: the stochastic partial differential
              equation approach},
   JOURNAL = {Journal of the Royal Statistical Society: Series B (Methodology)},
    VOLUME = {73},
      YEAR = {2011},
    NUMBER = {4},
     PAGES = {423--498}
}

@article{lang2015isotropic,
    AUTHOR = {Lang, Annika and Schwab, Christoph},
     TITLE = {Isotropic {G}aussian random fields on the sphere: regularity,
              fast simulation and stochastic partial differential equations},
   JOURNAL = {Annals of Applied Probability},
    VOLUME = {25},
      YEAR = {2015},
    NUMBER = {6},
     PAGES = {3047--3094}
}

@article{herrmann2020multilevel,
	title={Multilevel approximation of {G}aussian random fields: {F}ast simulation},
	author={Herrmann, Lukas and Kirchner, Kristin and Schwab, Christoph},
	journal={Mathematical Models and Methods in Applied Sciences},
	volume={30},
	number={01},
	pages={181--223},
	year={2020}
}

@inproceedings{borovitskiy2021matern,
	title={Mat{\'e}rn {G}aussian processes on graphs},
	author={Borovitskiy, Viacheslav and Azangulov, Iskander and Terenin, Alexander and Mostowsky, Peter and Deisenroth, Marc and Durrande, Nicolas},
	booktitle={AISTATS},
	pages={2593--2601},
	year={2021}
}

@article{porcu2016spatio,
    AUTHOR = {Porcu, Emilio and Bevilacqua, Moreno and Genton, Marc G.},
     TITLE = {Spatio-temporal covariance and cross-covariance functions of
              the great circle distance on a sphere},
   JOURNAL = {Journal of the American Statistical Association},
    VOLUME = {111},
      YEAR = {2016},
    NUMBER = {514},
     PAGES = {888--898}
}

@article{de2018regularity,
    AUTHOR = {Clarke De la Cerda, Jorge and Alegr\'{\i}a, Alfredo and Porcu,
              Emilio},
     TITLE = {Regularity properties and simulations of {G}aussian random
              fields on the sphere cross time},
   JOURNAL = {Electronic Journal of Statistics},
    VOLUME = {12},
      YEAR = {2018},
    NUMBER = {1},
     PAGES = {399--426}
}

@article{guella2018strictly,
    AUTHOR = {Guella, Jean Carlo and Menegatto, Valdir Antonio and Porcu,
              Emilio},
     TITLE = {Strictly positive definite multivariate covariance functions
              on spheres},
   JOURNAL = {Journal of Multivariate Analysis},
    VOLUME = {166},
      YEAR = {2018},
     PAGES = {150--159}
}

@article{alegria2021f,
    AUTHOR = {Alegr\'{\i}a, A. and Cuevas-Pacheco, F. and Diggle, P. and Porcu,
              E.},
     TITLE = {The $\mathcal{F}$-family of covariance functions: a {M}at\'{e}rn
              analogue for modeling random fields on spheres},
   JOURNAL = {Spat. Stat.},
    VOLUME = {43},
      YEAR = {2021},
     PAGES = {Paper No. 100512, 25}
}

@Manual{finley2024mba,
    title = {MBA: Multilevel B-Spline Approximation},
    author = {Andrew Finley and Sudipto Banerjee and Øyvind Hjelle},
    year = {2024},
    note = {R package version 0.1-2}
  }

@inproceedings{turk1994zippered,
  title={Zippered polygon meshes from range images},
  author={Turk, Greg and Levoy, Marc},
  booktitle={Proceedings of the 21st annual conference on Computer graphics and interactive techniques},
  pages={311--318},
  year={1994}
}

@book{ciarlet2002finite,
  title={The finite element method for elliptic problems},
  author={Ciarlet, Philippe G},
  year={2002},
  publisher={SIAM}
}

@Manual{rcite,
    title = {R: A Language and Environment for Statistical Computing},
    author = {{R Core Team}},
    organization = {R Foundation for Statistical Computing},
    address = {Vienna, Austria},
    year = {2025},
  }

@Manual{blender50,
   title = {Blender - a 3D modelling and rendering package},
   author = {{Blender Online Community}},
   organization = {Blender Foundation},
   address = {Blender Institute, Amsterdam},
   year = {2025},
   url = {http://www.blender.org},
 }

@book{brenner2008mathematical,
  title={The mathematical theory of finite element methods},
  author={Brenner, Susanne C and Scott, L Ridgway},
  year={2008},
  publisher={Springer}
}

@article{reuter2006laplace,
  title={Laplace--Beltrami spectra as ‘Shape-DNA’of surfaces and solids},
  author={Reuter, Martin and Wolter, Franz-Erich and Peinecke, Niklas},
  journal={Computer-Aided Design},
  volume={38},
  number={4},
  pages={342--366},
  year={2006},
  publisher={Elsevier}
}

@article{eldar1997farthest,
  title={The farthest point strategy for progressive image sampling},
  author={Eldar, Yuval and Lindenbaum, Michael and Porat, Moshe and Zeevi, Yehoshua Y},
  journal={IEEE transactions on image processing},
  volume={6},
  number={9},
  pages={1305--1315},
  year={1997},
  publisher={IEEE}
}

@inproceedings{pauly2002efficient,
  title={Efficient simplification of point-sampled surfaces},
  author={Pauly, Mark and Gross, Markus and Kobbelt, Leif P},
  booktitle={IEEE Visualization, 2002. VIS 2002.},
  pages={163--170},
  year={2002},
  organization={IEEE}
}

@incollection{meyer2003discrete,
  title={Discrete differential-geometry operators for triangulated 2-manifolds},
  author={Meyer, Mark and Desbrun, Mathieu and Schr{\"o}der, Peter and Barr, Alan H},
  booktitle={Visualization and mathematics III},
  pages={35--57},
  year={2003},
  publisher={Springer}
}

@book{botsch2010polygon,
  title={Polygon mesh processing},
  author={Botsch, Mario and Kobbelt, Leif and Pauly, Mark and Alliez, Pierre and L{\'e}vy, Bruno},
  year={2010},
  publisher={CRC press}
}

@article{wahba1981spline,
  title={Spline interpolation and smoothing on the sphere},
  author={Wahba, Grace},
  journal={SIAM Journal on Scientific and Statistical Computing},
  volume={2},
  number={1},
  pages={5--16},
  year={1981},
  publisher={SIAM}
}

@incollection{akenine2021differential,
  title={Differential Barycentric Coordinates},
  author={Akenine-M{\"o}ller, Tomas},
  booktitle={Ray Tracing Gems II: Next Generation Real-Time Rendering with DXR, Vulkan, and OptiX},
  pages={89--94},
  year={2021},
  publisher={Springer}
}

@inproceedings{ni_modeling_2019,
	address = {New York, NY, USA},
	series = {{WWW} '19},
	title = {Modeling {Heart} {Rate} and {Activity} {Data} for {Personalized} {Fitness} {Recommendation}},
	isbn = {978-1-4503-6674-8},
	urldate = {2026-09-24},
	booktitle = {The {World} {Wide} {Web} {Conference}},
	publisher = {Association for Computing Machinery},
	author = {Ni, Jianmo and Muhlstein, Larry and McAuley, Julian},
	year = {2019},
	pages = {1343--1353},
}

@article{alaimo2023bayesian,
  title={Bayesian {H}ierarchical {M}odeling and {A}nalysis for {A}ctigraph {D}ata from {W}earable {D}evices},
  author={Alaimo Di Loro, Pierfrancesco and Mingione, Marco and Lipsitt, Jonah and Batteate, Christina M and Jerrett, Michael and Banerjee, Sudipto},
  journal={The Annals of Applied Statistics},
  volume={17},
  number={4},
  pages={2865--2886},
  year={2023},
  publisher={Institute of Mathematical Statistics}
}

@misc{wakayama_process-based_2024,
    title = {Process-based {Inference} for {Spatial} {Energetics} {Using} {Bayesian} {Predictive} {Stacking}},
    urldate = {2026-10-02},
    publisher = {arXiv},
    author = {Wakayama, Tomoya and Banerjee, Sudipto},
    year = {2024},
    note = {arXiv:2405.09906 [stat.ME]},
}

@misc{mahmoudi_cortical_2026,
    title = {Cortical {Thickness} and {Curvature} in {Autism} and {ADHD}: {A} {Mega}-{Analysis}},
    copyright = {© 2026, Posted by openRxiv. This pre-print is available under a Creative Commons License (Attribution 4.0 International), CC BY 4.0, as described at http://creativecommons.org/licenses/by/4.0/},
    shorttitle = {Cortical {Thickness} and {Curvature} in {Autism} and {ADHD}},
    language = {en},
    urldate = {2026-10-03},
    publisher = {bioRxiv},
    author = {Mahmoudi, Maryam and Moore, Lucille A. and Lundquist, Jacob and Stier, Andrew and Anderson, Michael and Fayzullobekova, Begim and Hermosillo, Robert and Houghton, Audrey and Madison, Thomas J. and McCollum, Rae and Weldon, Kimberly B. and Nelson, Steve and Esler, Amy and Miranda-Dominguez, Oscar and Fair, Damien A. and Tervo-Clemmens, Brenden and Feczko, Eric},
    year = {2026},
    note = {ISSN: 2692-8205
Pages: 2026.01.05.697782
Section: New Results},
}

@article{grieser_uniform_2002,
    title = {Uniform {Bounds} for {Eigenfunctions} of the {Laplacian} on {Manifolds} with {Boundary}*},
    volume = {27},
    issn = {0360-5302},
    number = {7-8},
    urldate = {2026-10-04},
    journal = {Communications in Partial Differential Equations},
    publisher = {Taylor \& Francis},
    author = {Grieser, D.},
    year = {2002},
    pages = {1283--1299},
}

@book{lee_introduction_2012,
  address = {New York, NY},
  series = {Graduate {Texts} in {Mathematics}},
  title = {Introduction to {Smooth} {Manifolds}},
  volume = {218},
  copyright = {https://www.springernature.com/gp/researchers/text-and-data-mining},
  isbn = {978-1-4419-9981-8 978-1-4419-9982-5},
  language = {en},
  urldate = {2026-10-19},
  publisher = {Springer},
  author = {Lee, John M.},
  year = {2012},
}

@article{hu_gradient_2015,
  title = {The gradient estimate of a {Neumann} eigenfunction on a compact manifold with boundary},
  volume = {36},
  issn = {1860-6261},
  language = {en},
  number = {6},
  urldate = {2026-10-20},
  journal = {Chinese Annals of Mathematics, Series B},
  author = {Hu, Jingchen and Shi, Yiqian and Xu, Bin},
  year = {2015},
  pages = {991--1000},
}

@article{seeley_estimate_1980,
  title = {An {Estimate} {Near} the {Boundary} for the {Spectral} {Function} of the {Laplace} {Operator}},
  volume = {102},
  issn = {0002-9327},
  number = {5},
  urldate = {2026-10-20},
  journal = {American Journal of Mathematics},
  publisher = {Johns Hopkins University Press},
  author = {Seeley, R.},
  year = {1980},
  pages = {869--902},
}

@book{lee_introduction_2018,
  address = {Cham},
  series = {Graduate {Texts} in {Mathematics}},
  title = {Introduction to {Riemannian} {Manifolds}},
  volume = {176},
  copyright = {http://www.springer.com/tdm},
  isbn = {978-3-319-91754-2 978-3-319-91755-9},
  urldate = {2026-10-20},
  publisher = {Springer International Publishing},
  author = {Lee, John M.},
  year = {2018},
}

@book{petersen_riemannian_2016,
  address = {Cham},
  series = {Graduate {Texts} in {Mathematics}},
  title = {Riemannian {Geometry}},
  volume = {171},
  copyright = {http://www.springer.com/tdm},
  isbn = {978-3-319-26652-7 978-3-319-26654-1},
  urldate = {2026-10-21},
  publisher = {Springer International Publishing},
  author = {Petersen, Peter},
  year = {2016},
}

@article{banerjee2008gaussian,
	title = {Gaussian {Predictive} {Process} {Models} for {Large} {Spatial} {Data} {Sets}},
	volume = {70},
	issn = {1369-7412},
	number = {4},
	journal = {Journal of the Royal Statistical Society Series B: Statistical Methodology},
	author = {Banerjee, Sudipto and Gelfand, Alan E. and Finley, Andrew O. and Sang, Huiyan},
	year = {2008},
	pages = {825--848},
}

@article{wikle2010low,
  title={Low-rank representations for spatial processes},
  author={Wikle, Christopher K},
  journal={Handbook of spatial statistics},
  volume={107},
  pages={118},
  year={2010}
}

@article{quinonero-candela_unifying_2005,
  title = {A {Unifying} {View} of {Sparse} {Approximate} {Gaussian} {Process} {Regression}},
  volume = {6},
  issn = {1533-7928},
  number = {65},
  urldate = {2026-10-22},
  journal = {Journal of Machine Learning Research},
  author = {Quiñonero-Candela, Joaquin and Rasmussen, Carl Edward},
  year = {2005},
  pages = {1939--1959},
}

@article{halder_bayesian_2026,
  title = {Bayesian {Spatiotemporal} {Wombling}},
  volume = {0},
  issn = {0162-1459},
  number = {0},
  urldate = {2026-10-22},
  journal = {Journal of the American Statistical Association},
  publisher = {Taylor \& Francis},
  author = {Halder, Aritra and Li, Didong and Banerjee, Sudipto},
  year = {2026},
  pages = {1--14},
}

@article{di_martino_autism_2014,
  title = {The autism brain imaging data exchange: towards a large-scale evaluation of the intrinsic brain architecture in autism},
  volume = {19},
  copyright = {2014 Macmillan Publishers Limited},
  issn = {1476-5578},
  shorttitle = {The autism brain imaging data exchange},
  language = {en},
  number = {6},
  urldate = {2026-10-23},
  journal = {Molecular Psychiatry},
  publisher = {Nature Publishing Group},
  author = {Di Martino, A. and Yan, C.-G. and Li, Q. and Denio, E. and Castellanos, F. X. and Alaerts, K. and Anderson, J. S. and Assaf, M. and Bookheimer, S. Y. and Dapretto, M. and Deen, B. and Delmonte, S. and Dinstein, I. and Ertl-Wagner, B. and Fair, D. A. and Gallagher, L. and Kennedy, D. P. and Keown, C. L. and Keysers, C. and Lainhart, J. E. and Lord, C. and Luna, B. and Menon, V. and Minshew, N. J. and Monk, C. S. and Mueller, S. and Müller, R.-A. and Nebel, M. B. and Nigg, J. T. and O'Hearn, K. and Pelphrey, K. A. and Peltier, S. J. and Rudie, J. D. and Sunaert, S. and Thioux, M. and Tyszka, J. M. and Uddin, L. Q. and Verhoeven, J. S. and Wenderoth, N. and Wiggins, J. L. and Mostofsky, S. H. and Milham, M. P.},
  year = {2014},
  pages = {659--667},
}

@article{fischl_measuring_2000,
  title = {Measuring the thickness of the human cerebral cortex from magnetic resonance images},
  volume = {97},
  number = {20},
  urldate = {2026-10-23},
  journal = {Proceedings of the National Academy of Sciences},
  publisher = {Proceedings of the National Academy of Sciences},
  author = {Fischl, Bruce and Dale, Anders M.},
  year = {2000},
  pages = {11050--11055},
}

@article{ecker_interindividual_2022,
  title = {Interindividual {Differences} in {Cortical} {Thickness} and {Their} {Genomic} {Underpinnings} in {Autism} {Spectrum} {Disorder}},
  volume = {179},
  urldate = {2026-10-27},
  journal = {American Journal of Psychiatry},
  publisher = {American Psychiatric Publishing},
  author = {Ecker, Christine and Pretzsch, Charlotte M. and Bletsch, Anke and Mann, Caroline and Schaefer, Tim and Ambrosino, Sara and Tillmann, Julian and Yousaf, Afsheen and Chiocchetti, Andreas and Lombardo, Michael V. and Warrier, Varun and Bast, Nico and Moessnang, Carolin and Baumeister, Sarah and Dell’Acqua, Flavio and Floris, Dorothea L. and Zabihi, Mariam and Marquand, Andre and Cliquet, Freddy and Leblond, Claire and Moreau, Clara and Puts, Nick and Banaschewski, Tobias and Jones, Emily J.H. and Mason, Luke and Bölte, Sven and Meyer-Lindenberg, Andreas and Persico, Antonio M. and Durston, Sarah and Baron-Cohen, Simon and Spooren, Will and Loth, Eva and Freitag, Christine M. and Charman, Tony and Dumas, Guillaume and Bourgeron, Thomas and Beckmann, Christian F. and Buitelaar, Jan K. and Murphy, Declan G.M.},
  year = {2022},
  pages = {242--254},
}

@article{khundrakpam_cortical_2017,
  title = {Cortical {Thickness} {Abnormalities} in {Autism} {Spectrum} {Disorders} {Through} {Late} {Childhood}, {Adolescence}, and {Adulthood}: {A} {Large}-{Scale} {MRI} {Study}},
  volume = {27},
  issn = {1047-3211},
  shorttitle = {Cortical {Thickness} {Abnormalities} in {Autism} {Spectrum} {Disorders} {Through} {Late} {Childhood}, {Adolescence}, and {Adulthood}},
  number = {3},
  urldate = {2026-10-27},
  journal = {Cerebral Cortex},
  author = {Khundrakpam, Budhachandra S. and Lewis, John D. and Kostopoulos, Penelope and Carbonell, Felix and Evans, Alan C.},
  year = {2017},
  pages = {1721--1731},
}

@article{bedford_large-scale_2020,
  title = {Large-scale analyses of the relationship between sex, age and intelligence quotient heterogeneity and cortical morphometry in autism spectrum disorder},
  volume = {25},
  copyright = {2019 Springer Nature Limited},
  issn = {1476-5578},
  language = {en},
  number = {3},
  urldate = {2026-10-27},
  journal = {Molecular Psychiatry},
  publisher = {Nature Publishing Group},
  author = {Bedford, Saashi A. and Park, Min Tae M. and Devenyi, Gabriel A. and Tullo, Stephanie and Germann, Jurgen and Patel, Raihaan and Anagnostou, Evdokia and Baron-Cohen, Simon and Bullmore, Edward T. and Chura, Lindsay R. and Craig, Michael C. and Ecker, Christine and Floris, Dorothea L. and Holt, Rosemary J. and Lenroot, Rhoshel and Lerch, Jason P. and Lombardo, Michael V. and Murphy, Declan G. M. and Raznahan, Armin and Ruigrok, Amber N. V. and Smith, Elizabeth and Spencer, Michael D. and Suckling, John and Taylor, Margot J. and Thurm, Audrey and Lai, Meng-Chuan and Chakravarty, M. Mallar},
  year = {2020},
  pages = {614--628},
}

@article{finley_improving_2009,
  title = {Improving the performance of predictive process modeling for large datasets},
  volume = {53},
  issn = {0167-9473},
  number = {8},
  urldate = {2026-03-28},
  journal = {Computational Statistics \& Data Analysis},
  author = {Finley, Andrew O. and Sang, Huiyan and Banerjee, Sudipto and Gelfand, Alan E.},
  year = {2009},
  pages = {2873--2884},
}

@article{banerjee_high-dimensional_2017,
  title = {High-{Dimensional} {Bayesian} {Geostatistics}},
  volume = {12},
  issn = {1936-0975, 1931-6690},
  language = {en},
  number = {2},
  urldate = {2026-11-01},
  journal = {Bayesian Analysis},
  publisher = {International Society for Bayesian Analysis},
  author = {Banerjee, Sudipto},
  year = {2017},
  pages = {583--614},
}

@book{bhatia_matrix_1997,
  address = {New York, NY},
  series = {Graduate {Texts} in {Mathematics}},
  title = {Matrix {Analysis}},
  volume = {169},
  copyright = {http://www.springer.com/tdm},
  isbn = {978-1-4612-6857-4 978-1-4612-0653-8},
  urldate = {2026-11-02},
  publisher = {Springer},
  author = {Bhatia, Rajendra},
  year = {1997},
}

@inproceedings{wang_probabilistic_2024,
  title = {Probabilistic size-and-shape functional mixed models},
  volume = {37},
  urldate = {2026-11-05},
  booktitle = {Advances in {Neural} {Information} {Processing} {Systems}},
  publisher = {Curran Associates, Inc.},
  author = {Wang, Fangyi and Bharath, Karthik and Chkrebtii, Oksana and Kurtek, Sebastian},
  year = {2024},
  pages = {50031--50061},
}

@article{bharath_analysis_2020,
  title = {Analysis of shape data: {From} landmarks to elastic curves},
  volume = {12},
  copyright = {© 2020 Wiley Periodicals, Inc.},
  issn = {1939-0068},
  shorttitle = {Analysis of shape data},
  language = {en},
  number = {3},
  urldate = {2026-11-05},
  journal = {WIREs Computational Statistics},
  author = {Bharath, Karthik and Kurtek, Sebastian},
  year = {2020},
  pages = {e1495},
}

@misc{preston_rolled_2026,
  title = {Rolled {Gaussian} process models for curves on manifolds},
  urldate = {2026-11-05},
  publisher = {arXiv},
  author = {Preston, Simon and Bharath, Karthik and Lopez-Custodio, Pablo and Kume, Alfred},
  year = {2026},
  note = {arXiv:2503.21980 [math.ST]},
}
\end{document}